\documentclass[a4paper,11pt]{amsart}
\usepackage{graphicx,amssymb,amstext,amsmath,amsfonts,amsthm}
\usepackage{latexsym}
\usepackage{bbm}
\usepackage{ulem}
\usepackage{nicefrac}
\usepackage{hyperref}
\usepackage{atbegshi}
\usepackage{anysize}
\marginsize{3cm}{2cm}{3cm}{2cm}
\numberwithin{equation}{section}
\theoremstyle{plain}
\newtheorem{thm}{Theorem}[section]
\newtheorem{rem}{Remark}[section]
\newtheorem{lem}{Lemma}[section]    
\newtheorem{prop}{Proposition}[section]
\newtheorem{cor}{Corollary}[section]
\newtheorem{hyp}{Hypothesis}[section]

\newcommand{\<}{\langle}
\renewcommand{\>}{\rangle}

\newcommand{\W}{\mathcal{W}_p}
\newcommand{\Wp}{\mathcal{W}_{p'}}

\newcommand{\CC}{\mathbb{C}}

\newcommand{\EE}{\mathbb{E}}

\newcommand{\NN}{\mathbb{N}}

\newcommand{\PP}{\mathbb{P}}

\newcommand{\RR}{\mathbb{R}}

\DeclareMathOperator{\diag}{\ensuremath{diag}}

\newcommand{\aA}{\mathcal{A}}

\newcommand{\hH}{\mathcal{H}}

\newcommand{\lL}{\mathcal{L}}

\newcommand{\oO}{\mathcal{O}}
\newcommand{\pP}{\mathcal{P}}
\newcommand{\qQ}{\mathcal{Q}}

\newcommand{\sS}{\mathcal{S}}

\newcommand{\uU}{\mathcal{U}}

\newcommand{\wW}{\mathcal{W}}

\newcommand{\fq}{\mathfrak{q}}
\DeclareMathOperator{\trace}{\ensuremath{trace}}

\newcommand{\al}{\alpha}

\newcommand{\e}{\varepsilon}
\newcommand{\la}{\lambda}

\newcommand{\si}{\sigma}

\newcommand{\om}{\omega}

\newcommand{\z}{\zeta}

\newcommand{\pd}{\partial}

\newcommand{\ra}{\rightarrow}

\newcommand{\lra}{\longrightarrow}
\newcommand{\ti}{\tilde}

\newcommand{\ind}{\mathbf{1}}

\newcommand{\lqq}{\leqslant}
\newcommand{\gqq}{\geqslant}

\DeclareMathSymbol{\ophi}{\mathalpha}{letters}{"1E}

\renewcommand{\phi}{\varphi}
\newcommand{\be}{\begin{equation}}
\newcommand{\ee}{\end{equation}}
\newcommand{\ben}{\begin{equation*}}
\newcommand{\een}{\end{equation*}}
\newcommand{\ba}{\begin{equation}\begin{aligned}}
\newcommand{\ea}{\end{aligned}\end{equation}}

\newcommand{\ft}{\mathfrak{t}}

\usepackage[]{color}
\definecolor{DarkGreen}{rgb}{0.1,0.7,0.3}   
\definecolor{DarkBlue}{rgb}{0,0,0.7}
\definecolor{DarkRed}{rgb}{0.95,0,0}
\definecolor{DarkGrey}{rgb}{0.6,0.6,0.6}

\newcommand{\cblue}[1]{\textcolor{DarkBlue}{#1}}

\numberwithin{equation}{section}
\newcommand{\ud}{\mathrm{d}}

\title{Cutoff thermalization for Ornstein-Uhlenbeck systems 
with small L\'evy noise 
in the Wasserstein distance
}

\author{G. Barrera}
\address{University of Helsinki, Finland}
\email{gerardo.barreravargas@helsinki.fi}
\author{M.A. H\"ogele}
\address{Universidad de los Andes, Bogot\'a, Colombia}
\email{ma.hoegele@uniandes.edu.co}
\author{J.C. Pardo}
\address{
 CIMAT. Jalisco S/N, Valenciana, CP $36240$. Guanajuato, Guanajuato, M\'exico.}
\email{jcpardo@cimat.mx}

\begin{document}
\begin{abstract} 
This article establishes \textit{cutoff thermalization} (also known as the \textit{cutoff phenomenon}) for a class of generalized Ornstein-Uhlenbeck systems $(X^\e_t(x))_{t\gqq 0}$ with $\e$-small additive L\'evy noise and initial value $x$.  
The driving noise processes include Brownian motion, $\alpha$-stable L\'evy flights, finite intensity compound Poisson processes, and red noises, and may be highly degenerate. 
\textit{Window cutoff thermalization} is shown under mild generic assumptions; 
that is, we see an asymptotically sharp $\infty/0$-collapse 
of the renormalized Wasserstein distance from the current 
state to the equilibrium measure $\mu^\e$  
along a time window centered on a precise $\e$-dependent time scale $\ft_\e$. 
In many interesting situations such as reversible (L\'evy) diffusions 
it is possible to prove the existence of an explicit, universal, deterministic 
\textit{cutoff thermalization profile}. That is, 
for generic initial data $x$ we obtain the stronger result 
$\mathcal{W}_p(X^\varepsilon_{t_\varepsilon + r}(x), \mu^\varepsilon) \cdot \varepsilon^{-1} \rightarrow K\cdot e^{-q r}$ for any $r\in \mathbb{R}$ as $\varepsilon \rightarrow 0$ for some spectral constants $K, q>0$ and any $p\geqslant 1$ whenever the distance is finite. 
The existence of this limit is characterized 
by the absence of non-normal growth patterns in terms of 
an orthogonality condition on a computable family of generalized eigenvectors of $\qQ$. 
Precise error bounds are given. Using these results, this article provides a complete discussion of the cutoff phenomenon  
for the classical linear oscillator with friction subject to $\e$-small Brownian motion 
or $\alpha$-stable L\'evy flights. Furthermore, we cover the highly degenerate 
case of a linear chain of oscillators in a generalized heat bath at low temperature. 
\end{abstract}

\maketitle
\markboth{Cutoff thermalization of Ornstein-Uhlenbeck type processes in the Wasserstein distance}{
Cutoff thermalization of Ornstein-Uhlenbeck type processes in the Wasserstein distance}
\section{\textbf{Introduction}}\label{intro}

The notion of cutoff thermalization (also known as the cutoff phenomenon or abrupt thermalization in the literature) has gained growing attention in recent years in the physics literature 
with applications to quantum Markov chains \cite{Kastoryano12}, chemical kinetics \cite{BOK10}, 
quantum information processing \cite{Kastoryano13}, the Ising model \cite{LubetzkySly13}, 
coagulation-fragmentation equations \cite{Pego17}\cite{Pego16}, dissipative quantum circuits \cite{JohnssonTicozziViola17} and open quadratic fermionic systems \cite{Vernier20}. 
The term ``cutoff'' was originally coined in 1986 by Aldous \& Diaconis in their celebrated paper \cite{AD} on card shuffling, 
where they observed and conceptualized the asymptotically abrupt collapse of the total variation distance between the current state of their Markov chain and the uniform limiting distribution at a precise deterministic time scale.   

At this point we refrain from giving a full account on the mathematical literature on the cutoff phenomenon and refer to the overview article \cite{DI} and the introduction of \cite{BY1}. 
Standard references in the mathematics literature 
on the cutoff phenomenon for discrete time and space include 
\cite{Al83, AD87, AD, BD92, BLY06, BBF08, CSC08, DGM90, DS, DI87, LL19, La16,LanciaNardiScoppola12, LLP10, LPW,Meliot14, Trefethenbrothers, Yc99}.
As introductory texts on the cutoff phenomenon 
in discrete time and space we recommend \cite{Jonsson97} and Chapter~18 in the monograph~\cite{LPW}.

Although shown to be present in many important Markov chain models, cutoff thermalization is not universal. For instance, 
for reversible Markov chains Y. Peres formulated the widely used \textit{product condition}, that is, the divergence of the product between the mixing time and the spectral gap for growing dimension, see introduction of \cite{Hermon}. The product condition is a necessary condition for pre-cutoff in total variation (see Proposition~18.4 in \cite{LPW}), and a necessary and sufficient
condition for cutoff in the $L^2$ sense (see \cite{CSC08}).
This condition can be used to characterize cutoff for a large class of Markov chains, but it fails in general, see Chapter~18 in \cite{LPW} for the details. 
The alternative condition that the product of the spectral gap and
the maximal (expected) hitting time diverges is studied in \cite{Al89} and [\cite{Hermon19}, Theorem 1]. To the best of our knowledge, there are no general criteria available for the cutoff phenomenon of non-reversible Markov chains in discrete space.

This article establishes just such a criterion for the class of general (reversible and non-reversible) ergodic multidimensional L\'evy-driven Ornstein-Uhlenbeck processes in continuous space and time for small noise amplitude with respect to the (Kantorovich-Rubinstein-) Wasserstein-distance. 
Recall that 
the classical $d$-dimensional Ornstein-Uhlenbeck process is given as the unique strong solution of 
\begin{align}\label{eq: linear Brownian SDE}
\ud X_t^\e = -\qQ X_t^\e\ud t + \e \ud B_t,\qquad X_0^\e= x, \quad \e>0,
\end{align}
where $\qQ$ is a square matrix 
and $B = (B_t)_{t\gqq 0}$ a given $d$-dimensional Brownian motion. For {the definitions} see for instance 
\cite{Pavli, RI}. The marginal $X^\e_t(x)$ at a fixed time $t>0$ has the Gaussian distribution $N(0, \e^2 \Sigma_t)$, where the covariance matrix $\Sigma_t$ has an integral representation given in Theorem~3.1 of \cite{SaYa} or Subsection \ref{ss:ext1} of this article.
Furthermore, if $\qQ$ has eigenvalues with positive real parts, the process $(X^\e_t(x))_{t\gqq 0}$ has the unique limiting distribution $\mu^\e  = N(0, \e^2 \Sigma_\infty)$, where $\Sigma_\infty = \lim_{t\ra\infty} \Sigma_t$, see Theorem 4.1 and 4.2 in \cite{SaYa}. 
Since $\qQ$ has full rank, $\Sigma_\infty$ is known to be invertible. Moreover, 
the Gaussianity of the marginals and the limiting distribution 
leads {to} an explicit formula for the relative entropy  
\begin{equation}\label{eq:explicitentropy}
H(X_t^\e(x)~|~ \mu^\e) = \frac{1}{2}\left(\frac{1}{\e^2}(e^{-\qQ t}x)^*\Sigma^{-1}_\infty (e^{-\qQ t}x) +
\mathrm{Tr}(\Sigma^{-1}_\infty \Sigma_t)-d+
\log \frac{\|\Sigma_\infty\|}{\|\Sigma_t\|}\right),
\end{equation}
where $\|\cdot\|=\textrm{det}$. 
Note that $\mathrm{Tr}(\Sigma^{-1}_\infty \Sigma_t)-d +\log \big(\|\Sigma_\infty\| / \|\Sigma_t\|\big)\ra 0$ for any time scale $t \ra \infty$, thus the first term in formula (\ref{eq:explicitentropy}) turns out to be asymptotically decisive, when $t$ is replaced by some $t_\e \ra \infty$ as $\e \ra 0$. 
In particular, for a positive multiple of the identity, $\qQ = q \cdot I_d, q>0$, and $\ft_\e := q^{-1} |\ln(\e)|$, 
the following dichotomy holds for any $x \neq 0$: 
\begin{equation}\label{eq: cutoffentropy}
H(X_{\delta \cdot \ft_\e}^\e(x)~|~ \mu^\e) 
~\approx_\e ~\frac{1}{2} \bigg(\frac{|\Sigma_\infty^{-\frac{1}{2}} e^{-\qQ \,\delta \cdot \ft_\e}x|}{\e}\bigg)^2
~\propto ~\e^{2(\delta -1)}~\stackrel{\e\ra 0}\lra ~\left.
\begin{cases} 
\infty &\mbox{ for }   \delta \in (0,1)\\
0 &\mbox{ for }  \delta>1
\end{cases}\right\}. 
\end{equation}
The discussion of formula (\ref{eq: cutoffentropy}) for a general asymptotically exponentially stable matrix $-\qQ$ is given in Subsection 6.1 of this article. The fine study of the dichotomy in (\ref{eq: cutoffentropy}) and its dependence on $x$ {for general $\qQ$}, 
is the core of \textit{cutoff thermalization} for relative entropy 
in the context of continuous time and space. 
The main shortcoming of formula (\ref{eq:explicitentropy}) is that it is not robust and hard to generalize to 
\begin{enumerate}
 \item[(I)] general degenerate noise such as the linear oscillator with noise only in the position and 
 \item[(II)] non-Gaussian white L\'evy noise processes or red noise processes, such as $\alpha$-stable L\'evy flights, Poissonian jumps, Ornstein-Uhlenbeck processes, or even deterministic drifts. 
 \end{enumerate}
 Additionally, it is not obvious in general how formula (\ref{eq:explicitentropy}) would imply an analogous dichotomy to the asymptotics in (\ref{eq: cutoffentropy}) for 
 \begin{enumerate}
 \item[(III)] statistically more tractable distances such as the total variation or the Wasserstein distance.
\end{enumerate}
In \cite{BP} items (I) and (II) have been addressed for smooth density situations in the technically demanding total variation distance under natural but statistically hardly verifiable regularization conditions. 
In this article, we study the generalized Ornstein-Uhlenbeck process $X^\e_\cdot(x) = (X^\e_t(x))_{t\gqq 0}$ given as the unique strong solution of the linear ordinary stochastic differential equation with additive L\'evy noise 
\begin{align}\label{eq: linear SDE}
\ud X_t^\e = -\qQ X_t^\e ~\ud t + \e \ud L_t,\qquad X_0^\e= x, 
\end{align}
with the cutoff parameter $\e>0$, where $\qQ$ is a general $d$-dimensional square matrix that has eigenvalues with positive real parts 
and $L = (L_t)_{t\gqq 0}$ is a general (possibly degenerate) 
L\'evy process with values in $\RR^d$.  
The purpose of this article is twofold. First, it establishes \textit{window cutoff thermalization} in the limit of small $\e$ for the family of processes $(X^\e_{\cdot}(x))_{\e \in (0, 1]}$ 
in terms of the renormalized Wasserstein distance whenever the latter is finite and $X^\e_\cdot(x)$ has a unique limiting distribution $\mu^\e$ for each $\e$. The notion of window cutoff thermalization turns out to be a refined and robust analogue 
of the dichotomy (\ref{eq: cutoffentropy}) 
{which addresses the issues (I)-(III)} 
for the renormalized Wasserstein distance, that is, informally, with a limit of the following type 
\begin{equation}\label{eq: delta-window} 
\lim_{\e \ra 0} \wW_p(X^\e_{\delta\cdot \ft_\e}(x), \mu^\e) \cdot \e^{-1} 
 = ~\left.
\begin{cases} 
\infty &\mbox{ for }   \delta \in (0,1)\\
0 &\mbox{ for }  \delta>1
\end{cases}\right\}. 
\end{equation}
Secondly, we study the stronger notion of a \textit{cutoff thermalization profile}, that is, 
the existence of the limit {for any fixed $r\in \RR$} 
\begin{equation}\label{e: introprofile} 
\lim_{\e \ra 0} \wW_p(X^\e_{\ft_\e + r}(x), \mu^\e) \cdot \e^{-1} = \pP_x(r).
\end{equation}
The presence of a cutoff thermalization profile for generic $x$
turns out to be characterized by the absence of non-normal growth effects, that is, the orthogonality of asymptotic ($t\ra \infty$) generalized eigenvectors of the {exponential matrix} $e^{-\qQ t}$. In \cite{BJ, BP} such limits have been studied and characterized for the total variation distance. The limit there, however, turns out to be hard to calculate or even to simulate numerically, 
while in our setting for $p\gqq 1$ the limit (\ref{e: introprofile}) is shown to take the elementary explicit shape 
\begin{align*}\pP_x(r) = K_x \cdot e^{-\fq_x r}, \qquad r\in \RR,\end{align*} 
where the positive constants $K_x$ and $\fq_x$ in general depend on the initial condition {$x$.} 
For generic values of $x$, that is, $x$ having a non-trivial projection on one of the eigenspaces of the eigenvalues of $\qQ$ with smallest real part and highest multiplicity, it turns out to be the spectral gap of $\qQ$. 
In addition, our normal growth characterization is applicable in concrete examples of interest 
such as the linear oscillator. 
The Markovian dynamics of (\ref{eq: linear SDE}) implies 
(whenever regularity assumptions, such as hypoellipticity, are satisfied)
that the probability densities $p^\e_t$ of the marginals $X^\e_t(x)$ 
are governed by the Fokker-Planck or master equation 
\begin{align*}
\pd_t p^{\e}_t &=(\aA^\e)^* p^{\e}_t, 
\end{align*}
where the generator $\aA^\e$ in general amounts to a full-blown unbounded linear integro-differential operator. Therefore state-of-the-art analytic methods, at best, are capable of studying the spectrum of $\aA^\e$ (numerically), which yield 
an upper bound for exponential convergence to the equilibrium $\mu^\e$ for sufficiently large time in the case of the spectrum lying in the left open complex half-plane. See for instance 
\cite{PatiVaid20} Section ``Hypoelliptic Ornstein-Uhlenbeck semigroups'' or Theorem~3.1 in \cite{BK03}. However, these types of results can only establish (qualitative) upper bounds, which do not reflect the real convergence of $p^\e_t$ to the equilibrium distribution $\mu^\e$. It is with more flexible probabilistic techniques (coupling or replica) that it is possible to show cutoff thermalization in this level of generality. 

The first work on cutoff thermalization covering certain equations of the type (\ref{eq: linear Brownian SDE}) 
is by Barrera and Jara \cite{BJ} in 2015 for scalar nonlinear dissipative SDEs 
with a stable state and $\e$-small Brownian motion 
in the unnormalized total variation distance $d_{\mathrm{TV}}$ using coupling techniques. 
The authors show that for this natural $(d=1)$ gradient system, 
there always is a cutoff thermalization profile 
which can be given explicitly in terms of the Gauss error function. 
The follow-up work \cite{BJ1} 
covers cutoff thermalization with respect to the total variation distance 
for (\ref{eq: linear Brownian SDE}) in higher dimensions, 
where the picture is considerably richer, 
due to the presence of strong and complicated rotational patterns. 
Window cutoff thermalization is proved for the general case. 
In addition, the authors precisely characterize  
the existence of a cutoff thermalization profile 
in terms of the omega limit sets appearing 
in the long-term behavior of the 
matrix exponential function $e^{-\qQ t}x$ in Lemma B.2 \cite{BJ1}, 
which plays an analogous role in this article. 
We note that in (\ref{eq: linear Brownian SDE}) and \cite{BJ1} the Brownian perturbation 
is nondegenerate, and hence the examples of the linear oscillator 
or linear chains of oscillators subject to small Brownian motion are not covered there. 
{The results of \cite{BP} mentioned above cover 
cutoff thermalization for (\ref{eq: linear SDE}) 
for nondegenerate noise $\ud L$ in the total variation distance and 
yield many important applications such as the sample processes and the sample mean process. } The proof methods are based on concise Fourier inversion techniques. 
Due to the mentioned regularity issue concerning 
the total variation distance the authors state their results  
under the hypothesis of continuous densities of 
the marginals, which to date is mathematically not characterized 
in simple terms. 
Their profile function is naturally given as a shift error 
of the L\'evy-Ornstein-Uhlenbeck limiting measure for $\e=1$ and 
measured in the total variation distance. 
These quantities are theoretically highly 
insightful, but almost impossible to 
calculate and simulate in examples. 
Their abstract characterization of the existence of 
a cutoff-profile given in \cite{BJ1}, which assesses the behavior of the mentioned profile function on a suitably defined omega limit set, 
is shown to be also valid in our setting (see Theorem \ref{th:profileabstract}). 

While the total variation distance with which the cutoff phenomenon 
was originally stated is equivalent to the convergence in distribution in finite spaces, it is much more difficult to analyze in continuous space and 
is not robust to small non-smooth perturbations. 
There have also been attempts to describe 
the cutoff phenomenon for quantum systems 
in other types of metrics such as the trace norm, see for instance \cite{Kastoryano12}. 
In this context the Wasserstein setting 
of the present article has the following four 
advantages in contrast to the original total variation distance. 

(1) It does not require any regularity 
except some finite $p$-th moment, $p>0$. 
This allows us to treat degenerate noise and to cover second order equations. 
As an illustration we give a complete discussion of cutoff thermalization of the damped linear oscillator in the Wasserstein distance subject to Brownian motion, Poissonian jumps without any regularizing effect, $\al$-stable processes including the Cauchy process and a deterministic perturbation. In the same sense we cover chains of linear oscillators in a generalized heat bath at low temperature. 

 (2) In contrast to the relative entropy and the total variation distance the Wasserstein distance has the particular property of \textit{shift linearity} for $p\gqq 1$, which reduces the rather complicated profile functions of \cite{BJ, BJ1, BP} to a \textit{simple exponential function} with no need for costly and complex simulation. In addition, the profile is \textit{universal} and does not depend on which Wasserstein distance is applied nor on the statistical properties of the noise. For $p\in (0,1)$ shift linearity seems not to be feasible, however we give upper and lower bounds which essentially account for the same. Therefore we may cover the case of the linear oscillator under $\e$-small $\alpha$-stable perturbations including the Cauchy process for $\al=1$. 
 
(3) We also obtain cutoff thermalization for the physical observable finite $p$-th moments, which cannot be directly deduced from any result in  \cite{BJ, BJ1, BP}. Our findings also naturally extend to small red noise and general ergodic perturbations as explained in Section \ref{ss:ext2}. 

(4) Due to the homogeneity structure of the Wasserstein distance we give meaningful asymptotic error estimates and estimates on the smallness of $\e$ needed in order to observe cutoff thermalization on a finite interval $[0, T]$. 

The Wasserstein distance also entails certain minor drawbacks. First, a price to pay is to pass from the unnormalized total variation distance (due to $0$-homogeneity $d_{\mathrm{TV}}(\e U_1, \e U_2) = d_{\mathrm{TV}}(U_1, U_2)$) to the renormalized Wasserstein distance $\W / \e$. This is fairly natural to expect for any distance based on norms such as the $L^p$-norm, $p\gqq 1$ due to the $1$-homogeneity $\W(\e U_1, \e U_2) = \e \W(U_1, U_2)$. The second issue is that concrete evaluations of the Wasserstein distance are complicated in general. For $d=1$ and $1\lqq p < \infty$ the Wasserstein distance has the explicit shape of an $L^p$-distance for the quantiles $F_{U_1}^{-1}$ and $F_{U_2}^{-1}$
\[
\wW_p(U_1, U_2) = \Big(\int_0^1 |F_{U_1}^{-1}(\theta) -F_{U_2}^{-1}(\theta)|^p \ud\theta\Big)^\frac{1}{p}.
\]
However, there are no known higher dimensional counterparts of this formula. While by definition 
Wasserstein distances are minimizers of $L^p$-distances, they are always bounded above by the $L^p$-distance (by the natural coupling); however, lower bounds are typically hard to establish. 

The dynamics of models (\ref{eq: linear Brownian SDE}) 
with small Brownian motion have been studied since the early days of Arrhenius \cite{Arrhenius-89}, Ornstein and Uhlenbeck \cite{UO}, Eyring \cite{Eyring-35} Kramers \cite{Kramers-40}. Since then, an enormous body of physics literature has emerged and we refer to the overview articles \cite{HTB90} on the exponential rates and \cite{GHJM98} on the related phenomenon of stochastic resonance. For an overview on the Ornstein-Uhlenbeck process see \cite{Ja96}.
However, in many situations Brownian motion alone  is too restrictive for modeling the driving noise, 
as laid out in the article by Penland and Ewald \cite{PenlandE-08}, where the authors identify the physical origins of stochastic forcing and discuss the trade off between Gaussian vs. non-Gaussian white and colored noises. 
In particular, heavy-tailed L\'evy noise has been found to be present in physical systems such as 
for instance \cite{bodai2017predictability, ChenWuDuanKurts-19, Ditlevsen-99a, Ditlevsen-99b, DitDit09, GHKM17, sokolov}. 
In the mathematics literature the dynamics of the exit times of ordinary, delay and partial differential equations with respect to such kind of Brownian perturbations is often referred to as Freidlin-Wentzell
theory{.} 
It was studied in \cite{BerglundG-04,BerGen-10, BerGen-12, BovierEGK-04,Br96, Day83, Day96, Fr88, FW3, Siegert} and serves as the base on which metastability and stochastic resonance results are derived, for instance in \cite{BovierGK-05,BarBovMel10, FJL82, GalvesOV-87, SS19}. More recent extensions of this literature for the non-Brownian L\'evy case often including polynomial instead of exponential rates include \cite{DHI13, Godovanchuk-82, HP13, HP15, HP19, Ho19, IP08, IP06, ImkPavSta-10} and references therein. 
A different, recent line of research starting with the works of \cite{BDM11,BCD13,BN15,Budhiraja19, OGH20} treats $\e$-small and simultaneously $1/\e$-intensity accelerated Poisson random measures which yield large deviations for $\e$-parametrized L\'evy processes, also in the context of L\'evy processes, where this behavior typically fails to hold true. 

The paper is organized as follows.
After the setup and preliminary results 
the cutoff thermalization phenomenon is derived in Subsection \ref{ss:derivation}.
The main results on the stronger notion of profile cutoff thermalization are presented 
in Subsection \ref{ss:firstmainresult} followed by the generic results on the weaker notion of window cutoff thermalization in Subsection \ref{ss:secondmainresult}.
Section \ref{s:physicalexamples} is devoted to the applications in physics such as gradient systems and a complete discussion of the linear oscillator and numerical results of a linear Jacobi chains coupled to
 a heat bath.
In Section \ref{s:conceptual} several conceptual examples illustrate  certain mathematical features such as the fact that
leading complex eigenvalues not necessarily destroy the profile thermalization.
Moreover, we highlight
the dependence of the thermalization time scale
on the initial data $x$, and  Jordan block multiplicities of $\qQ$.
In {Section \ref{s:ext}} we discuss the pure Brownian case for relative entropy, the validity of the results for general ergodic driven noises such as red noise and derive conditions on $\e$ for observing the cutoff thermalization on a given finite time horizon. The proofs of the main results are given in the appendix. 


\section{\textbf{The setup}}\label{s:setup}

\subsection{\textbf{The L\'evy noise perturbation $\ud L$}}\label{ss:Levynoise}
Let $L = (L_t)_{t\gqq 0}$ be a L\'evy process with values in $\RR^d$, 
that is, a process with stationary and independent increments starting from $0$ almost surely, and c\`adl\`ag paths (right-continuous with left limits). 
The most prominent examples are the Brownian motion and the compound Poisson process.
For an introduction to the subject we refer to \cite{APPLEBAUM, Sa}.
The characteristic function of the marginal $L_t$ has the following (L\'evy-Khintchine) representation 
for any $t\gqq 0$ 
\[
u \mapsto \EE[e^{i\<u, L_t\>}] = \exp\Big(t \big[i\< u, b\>- \frac{1}{2} \<\Sigma u, u\> + \int_{\RR^d} \big(e^{i\<u, y\>}- 1 -i \<u, y\>\ind\{|y|\lqq 1\}\big)\nu(\ud y)\big]\Big),  
\]
for a drift vector $b\in \RR^d$, $\Sigma$ a $d\times d$ covariance matrix 
and $\nu$ a sigma-finite measure on $\RR^d$ with 
\[
\nu(\{0\}) = 0\qquad \textrm{ and }\qquad \int_{\RR^d} (1\wedge |z|^2) \nu(\ud z) < \infty.  
\]

\begin{hyp}[Finite $p$-th moment]\label{hyp:moments}
For $p>0$
the L\'evy process $L$ has finite $p$-th moments, 
which is equivalent to 
\[
\int_{|z|>1} |z|^p \nu(\ud z) < \infty, 
\]
where $\nu$ is the L\'evy jump measure.
\end{hyp}
\bigskip

\begin{rem}~
\begin{enumerate}
 \item In case of  $L = B$
being a Brownian motion
  Hypothesis \ref{hyp:moments} is true for any $p>0$.
 \item For $p \in (0,2)$ it also covers the case of $\alpha$-stable noise for $\alpha\in (p,2)$. 
Note that the latter only has moments of order $p< \alpha$ and hence no finite variance.
\end{enumerate}
\end{rem}

\subsection{\textbf{The Ornstein-Uhlenbeck process $(X^\e_t(x))_{t\gqq 0}$}}\label{ss:OUP} \hfill\\
We consider the following Ornstein-Uhlenbeck equation subject to $\e$-small L\'evy noise 
\begin{equation}\label{eq:linear}
\ud X^\e_t = -\qQ X^\e_t \ud t + \e  \ud L_t, \qquad X^\e_0 = x,   
\end{equation}
where $\qQ$ is a deterministic $d\times d$ matrix. 
For $\e>0$ and any $x\in \RR^d$
the SDE (\ref{eq:linear})
has a unique strong solution. 
By the variation of constant formula  
\begin{equation}\label{eq:voc}
X^\e_t(x) = e^{-\qQ t} x + \e \int_0^t e^{- \qQ (t-s)} \ud L_s=: e^{-\qQ t} x + \e\oO_t,
\end{equation}
where $\oO_t$ is a stochastic integral which 
is defined in our setting by the integration by parts formula
\[ 
\oO_t=L_t-\int_{0}^{t} e^{-\qQ(t-s)}\qQ L_s\ud s.
\]
In general, for $t> 0$ the marginals $X^\e_t(x)$ may not have densities and are only given in terms of its characteristics due to
the irregular non-Gaussian jump component, see Proposition 2.1 in \cite{MA}.
For the case of pure Brownian noise,
the marginal $X^\e_t(x)$ exhibits a Gaussian density. 
Its mean and covariance matrix are given explicitly in 
Section 3.7 in \cite{Pavli}.
\hfill

\subsection{\textbf{Asymptotic exponential stability of $-\qQ$}}\label{ss:aesQ}

\begin{hyp}[Asymptotic exponential stability of $-\qQ$]
\label{hyp:statibility}
The real parts of all eigenvalues of $\qQ$ are positive. 
\end{hyp}
By formula (\ref{eq:voc}) it is clearly seen, that the fine structure of $e^{-\qQ t}x$ determines its dynamics. In general, calculating matrix exponentials is complicated. For basic properties and some explicit formulas we refer to \cite{APOSTOL}, Chapters 7.10 and 7.14. Roughly speaking, for symmetric $\qQ$ and generic $x\in \RR^d$, $x\neq 0$, the behavior of $e^{-\qQ t}x$ is given by $e^{-\la t} \<v, x\> v + o(e^{-\la t})$ where $\la>0$ is the smallest eigenvalue of $\qQ$ and $v$ is its corresponding eigenvector. 
For asymmetric $\qQ$ the picture is considerably blurred by the occurrence of multiple rotations. 
The complete analysis reads as follows and is carried out in detail in the examples. 

\begin{lem}\label{jara}
Assume Hypothesis \ref{hyp:statibility}. 
Then for any initial value $x \in \mathbb{R}^d$, $x\neq 0$, there exist a rate $\fq:=\fq(x)>0$, 
multiplicities $\ell:=\ell(x)$, $m:=m(x)  \in  \{1,\ldots, d\}$, angles $\theta_1:=\theta_1(x),\dots,$ $\theta_{m}:=\theta_m(x) \in [0,2\pi)$ and a family of linearly independent vectors $v_1:=v_1(x),\dots, v_{m}:=v_m(x)$ in $\mathbb{C}^d$ such that
\begin{equation}\label{eq:asmatexp}
\lim_{t \to \infty} \left |\frac{e^{\fq t}}{t^{\ell-1}} e^{- \qQ t}x - \sum_{k=1}^{m} e^{i  t\theta_k} v_k\right | =0.
\end{equation}
Moreover,
\begin{equation}\label{eq:belowabove}
0<\liminf_{t\rightarrow \infty}\left|\sum_{k=1}^{m} e^{i  t\theta_k} v_k\right|
\lqq 
\limsup_{t\rightarrow \infty}\left|\sum_{k=1}^{m} e^{i  t\theta_k} v_k\right|\lqq 
\sum_{k=1}^{m}  |v_k|.
\end{equation}
The numbers $\{\fq \pm i \theta_k, k=1,\ldots, m\}$ are eigenvalues of the matrix $\qQ$ and the vectors $\{v_k, k=1,\ldots, m\}$ are generalized eigenvectors of $\qQ$. 
\end{lem}
The lemma is established as Lemma B.1 in \cite{BJ1}, p. 1195-1196, and proved there. 
It is stated there under the additional hypothesis of coercivity $\< \qQ x, x\> \gqq \delta |x|^2$ for some $\delta>0$ and any $x\in \RR^d$. However, inspecting the proof line by line 
it is seen that the authors only use Hypothesis \ref{hyp:statibility} of the matrix $\qQ$. 
Hence the result is valid under {the sole Hypothesis \ref{hyp:statibility}}. For a detailed {understanding of the computation of the exponential matrix} we refer to the notes of \cite{Wahlen}, in particular, Theorem 22 and Section~3.

\begin{rem}
The precise properties (\ref{eq:asmatexp}) and (\ref{eq:belowabove}) turn out to be crucial for the existence of a cutoff thermalization profile. 
Note that, in general, the limit
\[
\lim_{t\to \infty}\left|\sum_{k=1}^{m} e^{i  t\theta_k} v_k\right|
\]
does not exist. However, if in addition $\qQ$ is symmetric we have $\theta_1=\cdots=\theta_m=0$ and consequently,
\[
\lim_{t\to \infty}\left|\sum_{k=1}^{m} e^{i  t\theta_k} v_k\right|=\left|\sum_{k=1}^{m}  v_k\right|\neq  0.
\]
\end{rem}

\hfill

\subsection{\textbf{The Wasserstein distance $\wW_p$}}\label{ss:wasserstein}
\hfill
Given two probability distributions $\mu_1$ and $\mu_2$ on $\RR^d$
with finite $p$-th moment for some $p>0$, we 
define the Wasserstein distance of order $p$ as follows
\begin{equation}\label{def:wasserstein}
\W(\mu_1,\mu_2)=
\inf_{\Pi} \left(\int_{\RR^d\times \RR^d}|u-v|^p\Pi(\ud u,\ud v)\right)^{\min\{1/p,1\}},
\end{equation}
where the infimum is taken over all joint distributions (also called couplings) $\Pi$ with marginals $\mu_1$ and $\mu_2$. The Wasserstein distance quantifies the distance between probability measures, for an introduction we refer to \cite{Vi09}. 
For convenience of notation we do not distinguish a random variable $U$ and its law $\PP_U$ as an argument of $\W$. That is, for random variables $U_1$, $U_2$ and probability measure $\mu$ we write $\W(U_1, U_2)$ instead of $\W(\PP_{U_1}, \PP_{U_2})$, $\W(U_1, \mu)$ instead of $\W(\PP_{U_1}, \mu)$ etc. 

\begin{lem}[Properties of the Wasserstein distance]\label{lem:invariance}\hfill\\
Let $p>0$, $u_1,u_2\in \RR^d$ be deterministic vectors, $c\in \RR$ and $U_1, U_2$ be random vectors in $\RR^d$ with finite $p$-th moment. Then we have: 
\begin{itemize}
\item[a)] The Wasserstein distance is a metric, in the sense of being definite, symmetric and satisfying the triangle inequality {in the sense of} Definition 2.15 in \cite{Ru76}.
\item[b)] Translation invariance: 
$\W(u_1+U_1,u_2+U_2)=\W(u_1-u_2+U_1,U_2)$.
\item[c)] Homogeneity: 
\[
\W(cU_1,cU_2)=
\begin{cases}
|c|\;\W(U_1,U_2)&\textrm{ for } p\in [1,\infty),\\
|c|^{p}\;\W(U_1,U_2)&\textrm{ for } p\in (0,1).
\end{cases}
\]
\item[d)] Shift linearity: 
For $p\gqq 1$ it follows 
\begin{equation}\label{eq:shitflinearity}
\W(u_1+U_1,U_1)=|u_1|.
\end{equation}
For $p\in(0,1)$ equality (\ref{eq:shitflinearity}) is false in general. However we have the following inequality
\begin{equation}\label{ec:cotaabajop01}
\max\{|u_1|^{p}-2\EE[|U_1|^p],0\}\lqq 
\W(u_1+U_1,U_1)\lqq |u_1|^{p}.
\end{equation}
\item[e)] Domination: For any given coupling $\ti \Pi$ between 
$U_1$ and $U_2$ it follows  
\[
\W(U_1, U_2) \lqq \Big(\int_{\RR^d\times \RR^d} |v_1-v_2|^p \ti \Pi(\ud v_1,\ud v_2)\Big)^{\min\{1/p,1\}}.
\]
\item[f)] Characterization: Let $(U_n)_{n\in \NN}$ be a sequence of random vectors with finite $p$-th moments 
and $U$ a random vector with finite $p$-th moment the following are equivalent: 
\begin{enumerate}
 \item $\W(U_n, U) \ra 0$ as $n\ra \infty$. 
 \item $U_n \stackrel{d}{\lra} U$ as $n \ra \infty$ and $\EE[|U_n|^p] \ra \EE[|U|^p]$ as $n\ra \infty$. \\
\end{enumerate}
\item[g)] Contraction: Let $T:\RR^d \to \RR^k$, $k\in \NN$, be Lipschitz continuous with Lipschitz constant~$1$. Then for any $p>0$ 
\begin{equation}
\W(T(U_1),T(U_2))\lqq \W(U_1,U_2).
\end{equation}
\hfill
\end{itemize}
\end{lem}
\noindent {The proof of Lemma \ref{lem:invariance} is given in Appendix \ref{ap:A}. }

\begin{rem}
\begin{enumerate}
 \item Property d) is less widely known and turns out
to be crucial to simplify the thermalization profile for $p\gqq 1$ from a complicated stochastic quantity to a deterministic exponential function, while still being useful for $p\in (0,1)$. 
 \item In general, the projection of a vector-valued Markov process to single coordinates is known to be non-Markovian. However, not surprisingly property g) allows to estimate the Wasserstein distance of its projections. This is used in Subsection \ref{ss:ext2} for degenerate systems and mimics the analogous property for the total variation distance given in Theorem~5.2 in \cite{devro}.
\end{enumerate}
\end{rem}

\begin{lem}[{Wasserstein approximation of the total variation distance}]\label{lem:totalvariation}
Let $U_1$ and $U_2$ be two random variables taking values on $\RR^d$.
Assume that there exists $p\in (0,1)$ small enough such that $U_1$ and $U_2$ possesses finite $p$-th moments.
Then
\[
d_{\mathrm{TV}}(U_1,U_2)=\lim\limits_{p'\to 0} \Wp(U_1,U_2).
\]
\end{lem}
\noindent The content of this lemma is announced in Section 2.1 in \cite{Panaretos}. The proof is given in Appendix~\ref{ap:A}.

\begin{rem}
Assume that for any $x\neq 0$ and $p\in (0,1)$ the formula 
$\W(x+\oO_\infty,\oO_\infty)=  |x|^p$ is valid. 
By Lemma \ref{lem:totalvariation} we have
\begin{align*}
d_{\mathrm{TV}}(x+\oO_\infty, \oO_\infty)=\lim\limits_{p'\to 0}
\Wp(x+\oO_\infty, \oO_\infty)=\lim\limits_{p'\to 0}|x|^{p'}=1.
\end{align*}
Hence for any $x\neq 0$, $d_{\mathrm{TV}}(x+\oO_\infty,\oO_\infty)=1$ which in general false 
whenever $\oO_\infty$ has a continuous positive density in $\RR^d$, for instance, for $\oO_\infty$ being  $\alpha$-stable  with index $\alpha\in (0,2]$. In other words, $\W(x+\oO_\infty,\oO_\infty)=  |x|^p$ breaks down for $p$ sufficiently small in all smooth density situations.
\end{rem}

\subsection{\textbf{Limiting distribution $\mu^\e$}}\label{ss:limiting}
We fix $\e>0$. By Proposition 2.2 in \cite{MA}, Hypotheses \ref{hyp:moments} and~\ref{hyp:statibility} yield the existence of a unique equilibrium distribution $\mu^{\e}$ and its characteristics are given there. 
Moreover, the limiting distribution $\mu^\e$ has finite $p$-th moments. 
{It is} the distribution of $\e\oO_\infty$, 
where $\oO_\infty$ is the limiting distribution of $\oO_t$ as $t\to \infty$ (with respect to the weak convergence). In fact, it follows the stronger property. 

\begin{lem}\label{lem:erg5}
Let Hypotheses \ref{hyp:moments} and~\ref{hyp:statibility} be satisfied. Then for any $x\neq 0$, $\e>0$ and $0<  p'\lqq p$ we have $\Wp(X^\e_t(x),\mu^\e)\to 0$ as $t\to \infty$. 
\end{lem}
\begin{proof}
 First note, there exist positive constants $q_*$ and $C_0$ such that $|e^{-\qQ t}|\lqq C_0 e^{-q_* t}$ for any $t\gqq 0$ due to the usual Jordan decomposition and the estimate 
\begin{equation}\label{eq:initialgrowth}
|e^{-\qQ t}|\lqq  
\left(\sup_{s\gqq 0}
\max\limits_{0\lqq j\lqq d-1}
\frac{s^{j}}{j!}e^{-(q-q_*) s}\right) e^{-q_* t}=C_0e^{-q_* t},
\end{equation}
where $q$ is the minimum of the real parts of the eigenvalues of $\qQ$.
Then for any $t\gqq 0$ and $x,y \in \RR^d$ 
\[
|X^{\e}_t(x)-X^{\e}_t(y)|^{p'}\lqq |e^{-\qQ t}|^{p'}|x-y|^{p'}\lqq C^{p'}_0 e^{-q_* p' t}|x-y|^{p'}.
\]
Hence 
\[
\Wp(X^{\e}_t(x), X^{\e}_t(y)) \lqq \Big(C_0 e^{-q_* t}|x-y|\Big)^{\min\{1, p'\}}.
\]
By disintegration of the invariant distribution $\mu^\e$ we have
\begin{align}
\Wp(X^\e_t(x),\mu^\e)&=\Wp(X^\e_t(x),X^\e_t(\mu^\e)) \lqq \int_{\RR^d}
\Wp(X^\e_t(x),X^\e_t(y))\mu^{\e}(\ud y)\nonumber\\
&\lqq \Big(C_0 e^{-q_* t}\Big)^{\min\{1, p'\}}
\int_{\RR^d}|x-y|^{\min\{1, p'\}}\mu^{\e}(\ud y)\nonumber\\
&\lqq 
\Big(C_0 e^{-q_* t}|x|\Big)^{\min\{1, p'\}}+(C_0 e^{-q_* t})^{\min\{1, p'\}}
\int_{\RR^d}
|y|^{\min\{1, p'\}}\mu^{\e}(\ud y).\label{eq:disintegration}
\end{align}
Since
$X^\e_\infty=\e \oO_\infty$, it follows
\[
\int_{\RR^d}|y|^{\min\{1, p'\}}\mu^{\e}(\ud y)=\e^{\min\{1, p'\}} \EE[|\oO_\infty|^{\min\{1, p'\}}] < \infty.
\]
As a consequence,
$\lim\limits_{t\to \infty}\sup\limits_{|x|\lqq R} \Wp(X^\e_t(x),\mu^\e)=0$ for any $R>0$ and $\e>0$. 
\end{proof}

Observe that
\[
X^\e_t(x) = e^{-\qQ t} x + \e \oO_t,\quad
\textrm{ where } \oO_t:=\int_0^t e^{- \qQ (t-s)}  \ud L_s. 
\]
In particular,
\begin{equation}\label{eq: OUthermalizationW}
\Wp(\oO_t,\oO_\infty)\lqq \Big(C_0e^{-q_* t}\Big)^{\min\{1, p'\}}\EE[|\oO_\infty|^{\min\{1, p'\}}].
\end{equation}
By the exponential stability hypothesis we have 
$e^{-\qQ t} x\to 0$ as $t\to \infty$.
Therefore, Slutsky's theorem yields 
$X^\e_t(x)\stackrel{d}{\lra}\e \oO_\infty$ as $t\to \infty$, where $\oO_\infty$ has distribution $\mu^1$.

\begin{rem}\label{rem:invariantmoment}
{ It is not difficult to see that Hypothesis \ref{hyp:moments} yields for $1\lqq p' \lqq p$ 
\begin{align}\label{eq:invariantmoment}
\EE\big[|\oO_\infty|^{\min\{1,p'\}}\big]&\lqq  |b|\frac{C_0}{q_*}+
| \Sigma^{1/2}|\frac{C_0}{\sqrt{2q_*}}+
\frac{C_0}{q_*}\sqrt{\int_{|z|\lqq 1}|z|^2 \nu(\ud z)}+
\exp\left(\frac{C_0}{q_*}\int_{|z|>C^{-1}_0}|z|\nu (\ud z)\right),
\end{align} 
where $q_*>0$ is given at the beginning of the proof of the preceding Lemma \ref{lem:erg5}.} The proof of the jump part and drift is elementary and given in \cite{Wang-12} p.1000-1001. The Brownian component can be easily estimated by the It\^o isometry, see for instance \cite{KS98} Section 5.6.
\end{rem}

\section{\textbf{The main results}}\label{s:mainresults}
\subsection{\textbf{The derivation of cutoff thermalization}}\label{ss:derivation}
\subsubsection{\textbf{The key estimates for $p\gqq 1$}}\label{sss:keyestimate1}
Recall that $\mu^\e$ has the  distribution of $\e\oO_\infty$.
For transparency we start with $1 \lqq p' \lqq p$. 
On the one hand, 
by Lemma \ref{lem:invariance} properties a), b), c) and d) we have 
\begin{align}\label{eq:keyestp1u}
\Wp(X_t^\e(x), \mu^\e) 
&= \Wp(e^{-\qQ t}x+ \e \oO_t, \e \oO_\infty) \nonumber \\
&\lqq \Wp(e^{-\qQ t}x + \e \oO_t, e^{-\qQ t}x +\e \oO_\infty) + \Wp(e^{-\qQ t}x + \e \oO_\infty, \e \oO_\infty) \nonumber\\
&= \Wp(\e \oO_t, \e \oO_\infty) + |e^{-\qQ t} x| \nonumber\\
&= \e\Wp(\oO_t, \oO_\infty) + |e^{-\qQ t} x|.
\end{align}
On the other hand, since $p'\gqq 1$, property d) in Lemma \ref{lem:invariance}
with the help of properties a), b) and c)
 yields
\begin{align}\label{eq:keyestp1l}
|e^{-\qQ t} x|&= \Wp(e^{-\qQ t} x + \e \oO_\infty, \e \oO_\infty) \nonumber\\
&\lqq \Wp(e^{-\qQ t} x + \e \oO_\infty, e^{-\qQ t} x+\e \oO_t)+\Wp(e^{-\qQ t} x + \e \oO_t,
\e \oO_\infty)\nonumber\\
& =\e\Wp(\oO_\infty, \oO_t)+\Wp(X^\e_t(x), \mu^\e).
\end{align}
Combining the preceding inequalities we obtain
\begin{align}\label{eq:coupling}
\frac{|e^{-\qQ t}x|}{\e} - \Wp(\oO_t, \oO_\infty)\lqq 
\frac{\Wp(X_t^\e(x), \mu^\e)}{\e} \lqq \frac{|e^{-\qQ t}x|}{\e} + \Wp(\oO_t, \oO_\infty).
\end{align}
Since the $\Wp(\oO_t, \oO_\infty)\to 0$ as $t\to \infty$, for any $t_\e\to \infty$ as $\e\to 0$ we have $\Wp(\oO_{t_\e}, \oO_\infty)\to 0$ as $\e\to 0$.
It remains to show abrupt convergence of
 $|\frac{e^{-\qQ {\ft^x_\e}} x}{\e}|$ for the correct choice of $\ft^x_\e$. Therefore, the refined analysis of  the  linear system $e^{-\qQ t}x$ carried out in Lemma~\ref{jara} is necessary.

 \begin{rem}
 Note that the preceding formula (\ref{eq:coupling}) is valid for any $p$ of Hypothesis \ref{hyp:moments}. 
If $L$ has finite moments of all orders, that is, formally $p = \infty$, 
we may pass to the limit in (\ref{eq:coupling}) and obtain 
\begin{align}\label{eq:coupling infinity}
\frac{|e^{-\qQ t}x|}{\e} - \wW_\infty(\oO_t, \oO_\infty)\lqq 
\frac{\wW_\infty(X_t^\e(x), \mu^\e)}{\e} \lqq \frac{|e^{-\qQ t}x|}{\e} + \wW_\infty(\oO_t, \oO_\infty).
\end{align}
 This is satisfied for instance in the case of pure Brownian motion or uniformly bounded jumps. Moreover 
 \begin{align*}
 \wW_\infty(\oO_t, \oO_\infty) &= \lim_{p \ra \infty} \wW_p(\oO_t, \oO_\infty) 
 \lqq \lim_{p \ra \infty} \int_{\RR^d} \wW_p(\oO_t, \oO_t(z)) \PP(\oO_\infty \in \ud z) \\
 &\lqq |e^{-\qQ t}|\cdot \EE[ |\oO_\infty|] \ra 0, \quad\mbox{ as }t\ra \infty.
 \end{align*}
 \end{rem}

\subsubsection{\textbf{The key estimates for $p\in (0,1)$}}\label{sss:keyestimate2}
We point out that for $0<p'\lqq p$
the distance $\Wp$ satisfies all properties of Lemma \ref{lem:invariance}, however, with modified versions of c) and d).
Therefore, the upper bound (\ref{eq:keyestp1u}) has the shape
\begin{align*}
\Wp(X_t^\e(x), \mu^\e) 
&\lqq \Wp(e^{-\qQ t}x + \e \oO_t, e^{-\qQ t}x +\e \oO_\infty) + \Wp(e^{-\qQ t}x + \e \oO_\infty, \e \oO_\infty) \\
&= \e^{p'}\Wp(\oO_t, \oO_\infty) + \e^{p'}\Wp\big(\frac{e^{-\qQ t}x}{\e} + \oO_\infty, \oO_\infty\big)
\end{align*}
and the lower bound (\ref{eq:keyestp1l}) reads
\begin{align*}
\e^{p'}\Wp\big(\frac{e^{-\qQ t}x}{\e} + \oO_\infty, \oO_\infty\big)&=
\Wp(e^{-\qQ t} x + \e \oO_\infty, \e \oO_\infty) \\
&\lqq \Wp(e^{-\qQ t} x + \e \oO_\infty, e^{-\qQ t} x+\e \oO_t)+\Wp(e^{-\qQ t} x + \e \oO_t,
\e \oO_\infty)\\
& =\e^{p'}\Wp(\oO_\infty, \oO_t)+\Wp(X^\e_t(x), \mu^\e).
\end{align*}
The combination of the preceding inequalities yields
\begin{align}\label{eq:couplingp}
\Wp\big(\frac{e^{-\qQ t} x}{\e} +  \oO_\infty,  \oO_\infty\big) - \Wp(\oO_t, \oO_\infty)&\lqq 
\frac{\Wp(X_t^\e(x), \mu^\e)}{\e^{p'}} \nonumber\\
&
\lqq \Wp\big(\frac{e^{-\qQ t} x}{\e} +  \oO_\infty,  \oO_\infty\big) + \Wp(\oO_t, \oO_\infty).
\end{align}
\begin{rem} For $p'\in (0,1)$, property d) in Lemma \ref{lem:invariance} yields
\begin{equation}
\frac{|e^{-\qQ t} x|^{p'}}{\e^{p'}}-2\EE[|\oO_\infty|^{p'}]\lqq 
\Wp\big(\frac{e^{-\qQ t} x}{\e} +  \oO_\infty,  \oO_\infty\big)\lqq \frac{|e^{-\qQ t} x|^{p'}}{\e^{p'}}. 
\end{equation}
\end{rem}

\bigskip 

\subsection{\textbf{The first main result: Characterizations of profile cutoff thermalization}}\label{ss:firstmainresult}\hfill\\
This subsection presents the first cutoff thermalization results 
of in the sense of (\ref{e: introprofile}) for the system (\ref{eq: linear SDE}) with $x\neq 0$. 

\begin{rem}
Note that for initial value $x=0$ there is no cutoff thermalization. 
Indeed, by property c) in Lemma \ref{lem:invariance} we have
\[
\Wp(X^\e_t(0),\mu^\e)=\Wp(\e \oO_t,\e \oO_\infty)=\e^{\min\{1,p'\}}\Wp( \oO_t, \oO_\infty).
\]
Hence for any $t_\e\to \infty$ as 
$\e\to 0$ we have
\[
\lim\limits_{\e\to 0}
\frac{\Wp(X^\e_{t_\e}(0),\mu^\e)}{\e^{\min\{1,p'\}}}=\lim\limits_{\e\to 0}\Wp( \oO_{t_\e}, \oO_\infty)=0
\]
excluding a cutoff time scale separation.
\end{rem}

\subsubsection{\textbf{Explicit cutoff thermalization profile in case of first moments $p\gqq 1$}}
\label{sss:explicitcutoff}
The first main result characterizes the convergence of 
$
{
\Wp(X^\e_t(x),\mu^\e)}/{\e}
$ to a profile function
for $x\neq 0$ and $1\lqq p'\lqq p$.

\begin{thm}[\textbf{Cutoff thermalization profile}]\label{th:profile}
Let $\nu$ satisfy Hypothesis \ref{hyp:moments} for some $1 \lqq p \lqq \infty$. 
Let $\qQ$ satisfy Hypothesis \ref{hyp:statibility} and $x\in \RR^d$, $x\neq 0$, with the spectral representation
$\fq>0$,
$\ell , m \in  \{1,\ldots, d\}$, $\theta_1,\dots,\theta_{m} \in [0,2\pi)$ and $v_1,\dots,v_{m} \in  \mathbb{C}^d$ 
of Lemma \ref{jara}.

Then the following statements are equivalent.
\begin{enumerate}
\item[i)] The $\omega$-limit set
\begin{equation}\label{eq:omegalimit}
\omega(x):=\left\{\textrm{accumulation points of $\sum_{k=1}^{m} e^{i  t\theta_k} v_k$ as 
$t\to \infty$ }\right\}
\end{equation}
is contained in a sphere, that is,
the function 
\begin{equation}
\label{eq:jaka}
\omega(x)\ni u\mapsto
|u|\quad \textrm{ is constant}.
\end{equation}
\item[ii)] 
For the time scale 
\begin{equation}\label{eq:thermalization profile}
\ft^x_\e=\frac{1}{\fq}|\ln(\e)|+\frac{\ell-1}{\fq}\ln(|\ln(\e)|)
\end{equation}
the system $(X^{\e}_t(x))_{t\gqq 0}$ exhibits
for all  asymptotically constant window sizes $w_\e\to w>0$ 
 the abrupt thermalization profile
 for any $1\lqq p'\lqq p$ in the following sense
\[
\lim_{\e\to 0}
\frac{\Wp(X^\e_{\ft^x_\e+r\cdot w_\e}(x),\mu^\e)}{\e}=
\pP_x(r) \quad \textrm{ for any }
r\in \RR,
\]
where 
\begin{equation}\label{eq:profileshape}
\pP_x(r):=\frac{e^{-r  \fq w}}{\fq^{\ell-1}}|v| \qquad \mbox{ for any representative }v\in \om(x).
\end{equation} 
\end{enumerate}
Under either of the conditions, for $\e$ sufficiently small, we have the error estimate
\begin{equation}\label{eq:rate}
\left|
\frac{\Wp(X^\e_{\ft^x_\e+r\cdot w_\e}(x),\mu^\e)}{\e}-\pP_x(r) \right|\lqq \Wp(\oO_{\ft^x_\e},\oO_\infty)+\Big| \frac{|e^{-\qQ (\ft^x_\e + r \cdot w_\e)}x|}{\e}-\pP_x(r)\Big|
\end{equation}
which for generic $x$ yields a constant $C_x$ such that 
\begin{equation}\label{eq:rateprecise}
\left|
\frac{\Wp(X^\e_{\ft^x_\e+r\cdot w_\e}(x),\mu^\e)}{\e}-\pP_x(r) \right|\lqq C_x~\e^{1 \wedge \frac{\mathfrak{g}}{\fq}}.
\end{equation}
\end{thm}

\noindent {The proof of Theorem \ref{th:profile} is given at the end of Appendix \ref{ap:B}.} In the sequel, we essentially characterize when the function
\[
\omega(x)\ni u\mapsto |u|
\]
is constant. We enumerate $v_1, \dots, v_m$ as follows. Without loss of generality we assume that $\theta_1=0$, that is, $v_1 \in \RR^d$. Otherwise we take $v_1=0$ and eliminate it from the sum $\sum_{k=1}^{m}e^{i\theta_k t}v_k$. Without loss of generality we assume that $m=2n+1$ for some $n\in \NN$.
We assume that 
$v_k$ and $v_{k+1}=\bar v_k$ are complex conjugate for all even number $k\in \{2,\ldots,m\}$.
For $k\in \{2,\ldots,m\}$ we write $
v_k=\hat v_k+i\check v_k$ where 
$\hat v_k,\check v_k\in \RR^d$. 
{Since 
\[
e^{i\theta_k t}v_k=[\cos(\theta_k t)\hat v_k-\sin(\theta_k t)\check v_k]+i[\sin(\theta_k t)\hat v_k+\cos(\theta_k t)\check v_k],
\]
the decomplexification given in Lemma \ref{lem:base} yields the representation 
\begin{equation}\label{eq:wthetarepre1}
\sum_{k=1}^{m}e^{i\theta_k t}v_k={v_1}+2\sum_{k=1}^{n}
\Big(\cos(\theta_{2k} t)\hat v_{2k}-\sin(\theta_{2k} t)\check v_{2k}\Big),
\end{equation}
where 
${v_1}\in \RR^d$.}
\begin{rem}\label{rem:rationallyindependent1}
Note that the angles $\theta_{2},\ldots,\theta_{2n}$ in (\ref{eq:wthetarepre1}) 
coming from Lemma \ref{jara} are rationally $2\pi$-independent for generic matrices $\qQ$ and  initial values $x$. In other words, they satisfy the non-resonance condition
\begin{equation}\label{eq:rationallyindependent1}
h_1\theta_{2}+\cdots+h_n \theta_{2n}\not\in  2\pi \cdot \mathbb{Z}
\end{equation}
for all $(h_1,\ldots,h_n)\in \mathbb{Z}^n\setminus \{0\}$.
\end{rem}

\begin{thm}\label{th:normalgrowth}
Let the assumptions of Theorem \ref{th:profile} be satisfied.
In addition, we assume that the angles 
$\theta_2,\ldots,\theta_{2n}$ are rationally $2\pi$-independent according to  (\ref{eq:rationallyindependent1}) in Remark \ref{rem:rationallyindependent1}. 
\hfill\\
Then
i) and ii) in Theorem \ref{th:profile}
are equivalent to the
following normal growth condition: 
the family of  $\RR^d$-valued vectors  
\begin{equation}\label{e:non-normal growth}
(\cblue{v_1},\hat v_2,\check v_2,\ldots,\hat v_{2n},\check v_{2n}) ~\textrm{ is orthogonal and satisfies }\quad
 |\hat v_{2k}|=|\check v_{2k}|\quad \textrm{ for } \quad k=1,\ldots,n.
\end{equation}
In this case the profile function has the following shape 
\begin{equation}\label{eq:simplifiedprofile}
\pP_x(r)=\frac{e^{-r  \fq w}}{\fq^{\ell-1}}|\sum_{j=1}^{m}v_k|, \qquad r\in \RR. 
\end{equation}
\end{thm}
\noindent The proof is given in Appendix \ref{Ap:normal}. 
It consists of a characterization of Theorem \ref{th:profile} item i), that is, the property of $\om(x)$ being contained in a sphere.  
This characterization is carried out in two consecutive steps in Appendix \ref{Ap:normal} under the non-resonance condition (\ref{eq:rationallyindependent1}) given in Remark~\ref{rem:rationallyindependent1}. 
Lemma \ref{lem:ida} yields the necessary implication, 
while Lemma~\ref{lem:vuelta} states the sufficiency. 

\begin{rem}\label{rem:explanationformula}
\hfill

\begin{enumerate}
\item It is clear that under item i) of Theorem \ref{th:profile} the profile can be defined as 
\[
\pP_x(r)=\frac{e^{-r  \fq w}}{\fq^{\ell-1}}|u|
\]
for any representative $u\in \omega(x)$. 
Under the assumption of non-resonance of Remark \ref{rem:rationallyindependent1}  
we have that $u=\sum_{j=1}^{m}v_k\in \omega(x)$.
Indeed, since the $\theta_2, \dots, \theta_{2n}$ are rationally independent 
$((e^{i \theta_2 t}, \dots, e^{i\theta_{2n} t}))_{t\gqq 0}$ is dense in the torus 
$S_1^n$, see Corollary 4.2.3 in \cite{VIANA}. Hence we approximate the point $(1, \dots, 1)$ for a subsequence $t_k\ra \infty$ 
as $k \ra\infty$ and hence $u=\sum_{j=1}^{m}v_k\in \omega(x)$ and (\ref{eq:simplifiedprofile}) 
is valid. 
\item 
The existence of a thermalization profile boils down to the precise geometric structure of the complicated limit set $\omega(x)$.
However, it is not difficult to cover several cases of interest. In particular, 
in the case $\qQ$ being symmetric $\omega(x)=\{\sum_{j=1}^{m}v_k\}$ since all the rotation angles vanish, the function (\ref{eq:jaka}) is trivially constant.
\item  The shape of the thermalization profile given in (\ref{eq:profileshape}) is suprisingly universal: \begin{align*}
\pP_x(r)=\frac{e^{-r  \fq w}}{\fq^{\ell-1}}|u|=\Wp\left(\frac{e^{-r  \fq w}}{\fq^{\ell-1}}u+\oO_\infty,\oO_\infty\right).
\end{align*}
It does not depend on the parameters $1\lqq p' \lqq p \in [1, \infty]$ (beyond finite moments of order $p$) nor on the statistical properties of the driving noise $\nu$ due to the shift linearity (\ref{eq:shitflinearity}), item d) of Lemma \ref{lem:invariance}. 
For $p=2$ item d) of Lemma \ref{lem:invariance} is well-known and a direct consequence of Pythagoras' theorem, see for instance \cite{Panaretos19}, Section 2, p.412. 
We give the proof for general $p\gqq 1$ in Appendix \ref{ap:A}. 
\item The statistical information of $L$ enters in the rate of convergence 
on the right-hand side of (\ref{eq:rate}).
Indeed, by (\ref{eq:disintegration}) we have generically
\begin{align}\label{eq:constanteslimite}
\Wp(\oO_{\ft^x_\e},\oO_\infty)\lqq 
\EE[|\oO_\infty|]\cdot |e^{-\qQ \ft^x_\e}|
\lqq C_0 \EE[|\oO_\infty|] \e,
\end{align}
where $C_0$ is given in (\ref{eq:initialgrowth}). Moreover, 
$\EE[|\oO_\infty|]$ is bounded explicitly in terms of the characteristic of the noise and the matrix $\qQ$, see (\ref{eq:invariantmoment}).
\item The  order of the asymptotic error  $|\pP_x(r) - \frac{|e^{-\qQ (\ft^x_\e + r \cdot w)}x|}{\e}|$ depends 
inherently on the spectral structure of $\qQ$.
In the worst case its rates of convergence are of logarithmic order $1/\ft^x_\e$, see formula (\ref{eq:tasamult}) in the example of Subsection \ref{ex:multiplicity}. 
In Subsection \ref{ex:complejo} we see the optimal rate of convergence where this error is zero. 
However, this is not the generic picture. 
Generically all eigenvalues $\la_1, \la_2, \dots, \la_d\in\CC$ have different real parts (up to pairs of complex conjugate eigenvalues) with multiplicity $1$.
Without loss of generality we label $\la_1, \la_2, \dots, \la_d\in\CC$ by ascending (positive) real parts. 
Moreover, $\fq= \mathsf{Re}(\la_1)$ in the generic case.
Under the assumption of a thermalization profile 
we count with the speed of convergence of order $e^{-\mathfrak{g} \ft^x_\e} = K(x) \e^{\nicefrac{\mathfrak{g}}{\fq}}$,  where 
\begin{equation}\label{eq:ordendeconvergencia}
\mathfrak{g} =\begin{cases} \mathsf{Re}(\la_2)-\fq, & \mbox{if }\la_2 \neq \bar \la_1, \\  \mathsf{Re}(\la_3)-\fq, & \mbox{if }\la_2 = \bar\la_1, \end{cases}
\end{equation}
and  
since any initial datum  $x$ has the unique representation 
$
x=\sum_{j=1}^{d}c_j(x)v_j
$
and hence $K(x)$ can be taken as 
\begin{equation}\label{eq:krepresentation}
K(x)=\max_{j=1,\ldots,d} |c_j(x)v_j|,
\end{equation}
where $v_j$ are the eigenvectors associated to the eigenvalue $\lambda_j$.
\item  By (\ref{eq:rate}), item (3) and item (5) we obtain the generic order of magnitude of $\e$ 
such that the asymptotic approximation holds
for concrete systems in terms of the noises characteristics, $\qQ$, the long-term dynamics of $|e^{-\qQ t}|$, and the initial value $x$.
\end{enumerate}
\end{rem}

\subsubsection{\textbf{Abstract cutoff thermalization profile in case of $p \in (0,1)$}}

This result is stated in order to cover perturbations of the Cauchy process, where $p<1$ 
and other stable processes such as the Holtsmark process $p<\frac{1}{2}$. 
Here the profile function does exist but remains abstract. 

\begin{thm}[\textbf{Abstract cutoff thermalization profile for any $p>0$}]\label{th:profileabstract} \hfill\\
Let the assumptions (and the notation) of Theorem \ref{th:profile}
be valid for some $0<p\lqq \infty$.
Then for any $0<p'\lqq p$ the following statements are equivalent.
\begin{enumerate}
\item[i)] For any $\lambda>0$, the function
$\omega(x)\ni u\mapsto
\Wp(\lambda u+\oO_\infty,\oO_\infty)$
is constant, where $\omega(x)$ is given in (\ref{eq:omegalimit}).
\item[ii)] 
For the time scale 
$\ft^x_\e$ given in (\ref{eq:thermalization profile})
the system $(X^{\e}_t(x))_{t\gqq 0}$ exhibits
for all  asymptotically constant window sizes $w_\e\to w>0$ 
 the abrupt thermalization profile
 for any $0< p'\lqq p$ in the following sense
\[
\lim_{\e\to 0}
\frac{\Wp(X^\e_{\ft^x_\e+r\cdot w_\e}(x),\mu^\e)}{\e^{\min\{1,p'\}}}=
\pP_{x, p'}(r) \quad \textrm{ for any }
r\in \RR,
\]
where
\begin{equation}\label{eq:abstractperfil}
\pP_{x, p'}(r):=\Wp\Big(\frac{e^{-r  \fq w}}{\fq^{\ell-1}}v+
\oO_\infty,\oO_\infty\Big) \qquad \mbox{ for any representative }v\in \om(x).
\end{equation}
\end{enumerate}
\end{thm}

\noindent {The proof is given in Appendix \ref{ap:B}. }

\begin{rem}
\hfill
\begin{enumerate}
\item For $1\lqq p'\lqq p$ Theorem \ref{th:profileabstract} with the help of property d) in Lemma \ref{lem:invariance} recovers an abstract  version of Theorem \ref{th:profile} which also extends to $p<1$.
\item The (asymptotic) error estimates of Theorem \ref{th:profile} are harder to obtain for $0<p'<1$.

\item
For (nondegenerate) pure Brownian motion, the existence of a cutoff thermalization profile in total variation distance is equivalent to  the set $\Sigma^{-1/2}\omega(x)$ being contained in a sphere, where $\Sigma$ is the covariance matrix of the invariant distribution, see Corollary~2.11 in \cite{BJ1}.
In Corollary~4.14 of \cite{BHP} 
it is shown that 
a corresponding geometric condition  is at least sufficient.
For further unexpected properties in the pure $\alpha$ stable case, see \cite{sokolov}.
\end{enumerate}
\end{rem}

\subsection{\textbf{The second main result: generic window cutoff thermalization}}\label{ss:secondmainresult}

Roughly speaking, condition (\ref{eq:jaka}) in item i)  of Theorem \ref{th:profile} (as well as item i) in Theorem \ref{th:profileabstract}) fails to hold 
if the rotational part of $\qQ$ is too strong.
However, for the general case we still have abrupt thermalization in the following  weaker sense. 

\begin{thm}\label{th:linearwhite}
Let the assumptions (and the notation) of Theorem \ref{th:profile}
be valid for some $0<p\lqq \infty$.
Then the system $(X^{\e}_t(x))_{t\gqq 0}$ exhibits
 window cutoff thermalization
 on the time scale 
\[
\ft^x_\e=\frac{1}{\fq}|\ln(\e)|+\frac{\ell-1}{\fq}\ln(|\ln(\e)|)
\]
and in the sense that for all asymptotically constant window sizes $w_\e\to w>0$ it follows  
\[
\lim_{r\to -\infty}\liminf_{\e\to 0}
\frac{\Wp(X^\e_{\ft^x_\e+r\cdot w_\e}(x),\mu^\e)}{\e^{\min\{1,p'\}}}=\infty\quad \mathrm{ and } \quad\lim_{r\to \infty}\limsup_{\e\to 0}
\frac{\Wp(X^\e_{\ft^x_\e+r\cdot w_\e}(x),\mu^\e)}{\e^{\min\{1,p'\}}}=0.
\]
for all $0< p'\lqq p$.
 \end{thm}
\noindent {The proof is given Appendix \ref{ap:C}.  }
In contrast to other distances, 
the Wasserstein distance
also implies the cutoff thermalization for 
the physical observables as follows. 
\hfill 
\begin{cor}\label{cor:momentscutoffOUP}
Let the assumptions (and the notation) of Theorem \ref{th:profile}
be valid for some $p>0$.
Then we have for any $0< p'\lqq  p<\infty$ and $x\neq 0$ 
\begin{align*}
\lim_{r\to \infty}
\liminf_{\e\to 0}\frac{1}{\e^{p'}}\EE[|X^{\e}_{\ft_\e^x+r\cdot w_\e}(x)|^{p'}]=\lim_{r\to \infty}
\limsup_{\e\to 0}\frac{1}{\e^{p'}}\EE[|X^{\e}_{\ft_\e^x+r\cdot w_\e}(x)|^{p'}]=
\EE[|\oO_{\infty}|^{p'}]
\end{align*}
and \begin{align*}
\lim_{r\to -\infty}
\liminf_{\e\to 0}\frac{1}{\e^{p'}}\EE[|X^{\e}_{\ft_\e^x+r\cdot w_\e}(x)|^{p'}]=\lim_{r\to -\infty}
\limsup_{\e\to 0}\frac{1}{\e^{p'}}\EE[|X^{\e}_{\ft_\e^x+r\cdot w_\e}(x)|^{p'}]=\infty.
\end{align*}
\end{cor}
\noindent {For the proof we refer to Appendix \ref{ap:D}. }\\
\hfill
\begin{cor}\label{cor: delta-window}
Let the assumptions (and the notation) of Theorem \ref{th:profile}
be valid for some $p>0$.
Then we have 
\begin{equation*} 
\lim_{\e \ra 0} \wW_p(X^\e_{\delta\cdot \ft_\e}(x), \mu^\e) \cdot \e^{-1} 
 = ~\left.
\begin{cases} 
\infty &\mbox{ for }   \delta \in (0,1)\\
0 &\mbox{ for }  \delta>1
\end{cases}\right\}. 
\end{equation*}
\end{cor}
\noindent {This corollary justifies formula (\ref{eq: delta-window}) in the introduction. 
For the proof we refer to Appendix \ref{ap:D}. }\\

\begin{rem}\label{rem:windows}~
\begin{enumerate}
\item
In general
for $1\lqq p'\lqq p$
 there is no thermalization profile in the sense of Theorem \ref{th:profile} (and Theorem \ref{th:profileabstract}).
However, it is easy to see that a cutoff thermalization profile implies window cutoff thermalization.
The contrary not always holds. For instance, for different values of $u\in \omega(x)$ there is no unique candidate for the profile.
To be more precise, 
\[
\limsup_{\e\to 0}
\frac{\Wp(X^\e_{\ft^x_\e+r\cdot w_\e}(x),\mu^\e)}{\e}=
\frac{e^{-r  \fq w}}{\fq^{\ell-1}}|\hat{u}|
\]
and
\[
\liminf_{\e\to 0}
\frac{\Wp(X^\e_{\ft^x_\e+r\cdot w_\e}(x),\mu^\e)}{\e}=
\frac{e^{-r  \fq w}}{\fq^{\ell-1}}|\check{u}|,
\]
where $\check{u},\hat{u}\in \omega(x)$.
The discussion of the linear oscillator given in Subsection \ref{ex:osc} 
yields an example where $\om(z)$ is not contained in a sphere for any $z\neq 0$. 
The case of subcritical damping always exhibits complex eigenvalues which together 
with the precise structure of the dynamics excludes a thermalization profile 
and only window thermalization remains valid. 
\item The error estimate in Remark \ref{rem:explanationformula} item (4) remains untouched and item (5) is slightly adapted as follows.
Here we consider the error term
\begin{equation}\label{eq:speedm}
\left |\frac{e^{- \qQ \ft^x_\e}x}{\e} - \sum_{k=1}^{m} e^{i  \ft^x_\e\theta_k} v_k\right |,
\end{equation}
which analogously depends on the spectral structure of $\qQ$. 
Generically all eigenvalues $\la_1, \la_2, \dots, \la_d\in\CC$ have different real parts (up to pairs of complex conjugate eigenvalues) with multiplicity $1$.
Without loss of generality we label $\la_1, \la_2, \dots, \la_d\in\CC$ by ascending (positive) real parts. 
Moreover, $\fq= \mathsf{Re}(\la_1)$ in the generic case.
The speed of convergence of (\ref{eq:speedm}) is of
order $e^{-\mathfrak{g} \ft^x_\e} = K(x) \e^{\nicefrac{\mathfrak{g}}{\fq}}$,  where 
$\mathfrak{g}$ is given in (\ref{eq:ordendeconvergencia}) and $K(x)$  given in (\ref{eq:krepresentation}) is estimated identically.

\item In Subsection \ref{ex:complejo} we give an (ad hoc) linear $(2\times 2)$-system 
showing a thermalization profile in the presence of complex (conjugate) eigenvalues 
for all initial values. 
\item The example in Subsection \ref{ex:suspension} represents a system where the presence of a thermalization profile depends {on} the initial value $x$.
\item We finally construct in Subsection \ref{ex:multiplicity} a system with arbitrary high logarithmic corrections terms. 
\end{enumerate}
\end{rem}
 
 \section{\textbf{Physics examples}}\label{s:physicalexamples}
 \subsection{\textbf{Gradient systems}}\label{ss:gradientsystems}
For a symmetric $\qQ$, Theorem \ref{th:profile} applies to the gradient case 
\[
\ud X^{\e}_t=-\nabla \uU(X^{\e}_t)\ud t +\e  \ud L_t\qquad \textrm{ with } X^{\e}_0=x\neq 0
\]
for the quadratic potential form $\uU(z)=(1/2)z^*\qQ z$. 
Indeed, by the spectral decomposition we have an orthonormal basis $v_1,v_2,\ldots, v_d\in \RR^d$ with corresponding eigenvalues $0<\la_1\lqq \cdots \lqq \la_d$ such that 
\[
e^{-\qQ t}x=\sum_{j=1}^{d}e^{-\la_j t}\<v_j,x\>v_j.
\]
Let $
\tau(x)=\min\{j\in \{1,\ldots,d\}: \<v_j,x\>\neq 0\}$ and $J=\{j\in \{\tau(x),\ldots,d\}: \la_j=\la_{\tau(x)}\}$. 
Hence 
\[
\lim\limits_{t\to \infty}
e^{\la_{\tau(x)} t}e^{-\qQ t}x=\sum_{j\in J}\<v_j,x\>v_j\neq 0.
\]
That is, for $p\gqq 1$ the cutoff thermalization profile is  $\pP_x(r)=e^{-r\la_{\tau(x)}w}|\sum_{j\in J}\<v_j,x\>v_j|$.
More generally,
\cite{BP}, Proposition A.4.ii) yields a complete description of the spectral 
decomposition of non-symmetric $\qQ$ with real spectrum. 

\subsection{\textbf{The linear oscillator}}\label{ex:osc}
In this subsection we provide a complete discussion of the cutoff thermalization 
of the damped linear oscillator driven by different noises at small temperature. We consider 
\begin{align}\label{eq:linearoscillator}
\left\{
\begin{array}{r@{\;=\;}l}
\ud X^{\e}_t & Y^{\e}_t\ud t,\\
\ud Y^{\e}_t& - \kappa X^{\e}_t \ud t-\gamma Y^\e_t \ud t+ \e \ud L_t
\end{array}
\right.
\end{align}
with initial conditions $X^{\e}_0=x$,
$Y^{\e}_0=y$ and L\'evy noise $L=(L_t)_{\gqq 0}$ satisfying Hypothesis \ref{hyp:moments} for some $p>0$. Examples of interest are the following:\\
For $p\lqq \infty$ we cover 
\begin{enumerate}
\item standard Brownian motion, 
\item deterministic (linear) drift,  
\item discontinuous compound Poisson process with finitely many point increments. 
\end{enumerate}
For $p< \alpha$ for some $\alpha>0$ 
\begin{enumerate}
\item[(4)] $\alpha$-stable L\'evy flight with finite first moment for index $\alpha\in (1,2)$,
\item[(5)] $\alpha$-stable L\'evy flight with index $\alpha\in (0,1]$ 
 including the Cauchy flight when $\alpha=1$. 
 See \cite{garbaol} and \cite{sokolov} for a thorough discussion. 
\end{enumerate}
We rewrite the system (\ref{eq:linearoscillator}) as a vector valued Ornstein-Uhlenbeck process
\begin{align*}
\ud \left( \begin{array}{c} X_t^\e \\ Y_t^\e \end{array}\right) 
&= -\qQ \left( \begin{array}{c} X_t^\e \\ Y_t^\e \end{array}\right) \ud t + \e \ud \lL_t,
\end{align*}
where 
\begin{equation}\label{eq:Qlinearoscillator}
\qQ:=\left( \begin{array}{cc} 0 & -1\\ \kappa & \gamma\end{array} \right)\quad \textrm{ and } \quad 
\lL_t:=\left( \begin{array}{c} 0 \\  L_t \end{array}\right).
\end{equation}
Let $z=(x,y)^*\not =(0,0)^*$. In the sequel, we compute $e^{-\qQ t}z$.
The eigenvalues of $-\qQ$ are given by 
\begin{align*}
\la_{\pm}=\frac{-\gamma \pm \sqrt{\gamma^2 - 4\kappa}}{2} .
\end{align*}
Note that for any $\gamma, \kappa>0$ 
the respective real parts are strictly negative.

\bigskip 
\subsubsection{{\bf{Overdamped linear oscillator}}: 
$\Delta = \gamma^2 - 4\kappa>0$.} In this case, $-\qQ$ has two real different eigenvalues
 \[
 \la_{-}:= \frac{1}{2} (-\gamma - \sqrt{\Delta})<\frac{1}{2} (-\gamma + \sqrt{\Delta})=:\la_{+} <0
 \]
with the respective eigenvectors $v_-$ and $v_+$.
The exponential matrix is given by 
\begin{displaymath}
e^{-\qQ t}=(v_- \; v_+)\diag\left(e^{-\la_- t}, e^{-\la_+ t}\right)(v_-\; v_+)^{*}\quad
\textrm{ for } t\gqq 0.
\end{displaymath}
Recall $z=(x,y)^*\neq (0,0)^*$.
We denote by $\tilde{z}:=(v_-\; v_+)^{*}z=(\tilde{z}_{1},\tilde{z}_{2})^*$ the  coordinate change of $z$.
Note that 
\begin{displaymath}
e^{-\qQ t}z=(v_- \; v_+)(\tilde{z}_{1}e^{-\la_- t},\tilde{z}_{2}e^{-\la_+ t})^{*}.
\end{displaymath}
This formula yields that for any $z\not={0}$ there exist an explicit $\fq>0$ and $u \not={0}$ such that
\begin{eqnarray}\label{q1}
\lim\limits_{t\rightarrow \infty}
e^{\fq t}e^{-\qQ t }z=u_z.
\end{eqnarray}
Indeed, if $\tilde{z}_{1}=0$ then $\tilde{z}_{2}\not=0$ and we have
$\lim\limits_{t\rightarrow \infty}e^{\la_+ t}e^{-\qQ t }z=(v_- \; v_+)(0,\tilde{z}_{2})^{*}=:u_z\not={0}.$ If $\tilde{z}_{2}=0$ then $\tilde{z}_{1}\not=0$ and we have $\lim\limits_{t\rightarrow \infty}e^{\la_- t}e^{-\qQ t }z=
(v_- \; v_+)(\tilde{z}_{1},0)^{*}=:u_z\not={0}$. Finally, if $\tilde{z}_{1}\not=0$ and $\tilde{z}_{2}\not=0$. Then
$\lim\limits_{t\rightarrow \infty}e^{\la_- t}e^{-\qQ t}z=(v_- \; v_+)(\tilde{z}_{1},0)^{*}=:u_z\not={0}.$
In particular, for any $z\neq 0$, the omega limit set $\omega(z)$ defined in (\ref{eq:omegalimit}) consists of a single point $u_z$.

Hence for the noises (1)-(4), Theorem \ref{th:profile} applies for $1\lqq p'\lqq p$ and thermalization profile holds at the time scale
\[
\ft^x_\e=\frac{1}{\fq}|\ln(\e)|,
\]
with profile
\[
\pP_z(r)=e^{-r  \fq w}|u_z|
\] for all  window sizes $w>0$.
Roughly speaking, for any $1\lqq p'\lqq p$ Theorem \ref{th:profile} yields for some positive $K_{p, p'}$ 
\[
\Wp(X^\e_{\ft^x_\e+r\cdot w}(x),\mu^\e)\approx_{\e} \e\, e^{-r  \fq w}|u|+\e^{2}K_{p, p'}\EE[|\oO_\infty|].
\]
For the noise (5), Theorem \ref{th:profileabstract} still implies profile thermalization, however, the profile is  given by the abstract formula
\[
\pP_{x, p'}(r)=\Wp\Big({e^{-r  \fq w}}u_z+
\oO_\infty,\oO_\infty\Big)
\]
for any $p'\lqq p<\alpha$.
\bigskip 
\subsubsection{{\bf{Critically damped linear oscillator}}: $\Delta = \gamma^2 - 4\kappa=0$} 
In this case,
$-\qQ$ has a repeated real  eigenvalue $\la:=\la_-=\la_+=-\gamma/2<0$ and the matrix exponential is given by
\[
e^{-\qQ t}=e^{\la t}I_2 + e^{\la t} t(-\qQ-\la I_2)\quad \textrm{ for } t\gqq 0.
\]
Let $z=(x,y)^*\neq (0,0)^*$. 
On the one hand, if $z\in \mathrm{Ker}(-\qQ-\la I_2)$, that is, in Lemma \ref{jara} we have
$\ell=1$, and 
$e^{-\la t}e^{-\qQ t}z=u_z$ for $u_z=z$.
On the other hand, $z\neq \mathrm{Ker}(-\qQ-\la I_2)$ which corresponds to $\ell=2$ yields
\[
\lim_{t\to \infty}\frac{e^{-\la t}}{t}e^{-\qQ t}z=(-\qQ-\la I_2)z=:u_z \neq 0.
\]
In particular, for any $z\neq 0$, the omega limit set $\omega(z)$ defined in (\ref{eq:omegalimit}) consists of the a single point $u_z$. Hence for the noises (1)-(4)
Theorem \ref{th:profile} still applies for $1\lqq p'\lqq p$ and profile thermalization holds true at the modified time scale
\[
\ft^x_\e=\frac{1}{\fq}|\ln(\e)|+\frac{\ell-1}{\fq}\ln(|\ln(\e)|)
\]
with the modified profile
\[
\pP_z(r)=\frac{{e^{-rw(\gamma /2)} }}{(\gamma/2)^{\ell-1}}|u|
\] for all  window sizes $w>0$.
Roughly speaking, for any $1\lqq p'\lqq p$ Theorem \ref{th:profile} yields for some positive constant $K_{p, p'}$ 
\[
\Wp(X^\e_{\ft^x_\e+r\cdot w}(x),\mu^\e)\approx_{\e} ~\e\cdot \frac{{e^{-rw (\gamma/2)} }}{(\gamma/2)^{\ell-1}}|u|
+\e^{2} |\ln(\e)|^{\ell-1}\cdot K_{p, p'}\, \EE[|\oO_\infty|].
\]
For the noise (5), Theorem \ref{th:profileabstract} still implies profile thermalization, however, the modified profile is also given by the abstract formula
\[
\pP_{x, p'}(r)=\Wp\Big({e^{-r  \fq w}}u_z+
\oO_\infty,\oO_\infty\Big)
\]
for any $p'\lqq p<\alpha$.

\bigskip 
In the sequel, we discuss the general case of complex conjugate eigenvalues in order to treat the subcritical case.
\subsubsection{{\bf{Non-normal growth of the linear oscillator for complex eigenvalues}}} 

Recall that the eigenvalues of $\qQ$ are given by
\[
\la_{-}=\hat\la-i\check\la\quad\textrm{ and }\quad \la_{+}=\hat \la+i\check\la, \quad \check \la \neq 0. 
\]
By the Jacobi formula, see for instance Theorem 1 in \cite{MN99}, Part Three, Sec 8.3, we have
\begin{align}\label{eq:jacobi}
\det(e^{\hat\la t} e^{-\qQ t}) &=  e^{2\hat\la t}  e^{-\trace(\qQ t)} = 1.
\end{align}
By the Lagrange interpolation theorem (see Theorem 7.11, p.209, in \cite{APOSTOL}) we have
\begin{align}
e^{-\qQ t}&=\frac{e^{-t \la_-}}{-2i \check\la} (-\qQ + \la_+ I_2)+ \frac{e^{-t \la_+}}{2i\check\la} (-\qQ + \la_- I_2)\nonumber\\
&=
\frac{e^{-t\hat\la}e^{it \check\la}}{2i \check\la} (\qQ - \la_+ I_2)- \frac{e^{-t\hat\la}e^{-it \check\la}}{2i\check\la} (\qQ - \la_- I_2)\nonumber\\
&= -\frac{e^{-t\hat\la}}{\check\la}\mathsf{Re}\Big(e^{it \check\la} (\qQ - \la_+ I_2)\Big).
\label{eq:matrixexponentialrepresentation}
\end{align}
Let $z=(x,y)^*\neq (0,0)^*$.
 The preceding equality yields
\begin{equation*}
|e^{t \hat\la}
 e^{-\qQ t}z|=\frac{1}{|\check\la|}|\mathsf{Re}\left(e^{i t \check\la} (\qQ - \la_+ I_2)z\right)|.
 \end{equation*}
Moreover, by (\ref{eq:jacobi}) we deduce 
\begin{equation}\label{eq:liminfpos}\liminf\limits_{t\ra \infty} | e^{\hat\la t} e^{-\qQ t} z | >0.\end{equation}
Additionally by the periodicity we have that 
\begin{equation}\label{eq:limsupfin}\limsup\limits_{t\ra \infty} | e^{\hat\la t} e^{-\qQ t} z | <\infty.\end{equation}
Note that $|e^{t \hat\la}
 e^{-\qQ t}z|$ is a constant function if and only if 
 $|\mathsf{Re}\left(e^{i t \check\la} (-\qQ + \la_+ I_2)z\right)|$ is so, too.
In the sequel, we characterize when the function
\[
t\mapsto 
|\mathsf{Re}\left(e^{i t \check\la} (\qQ - \la_+ I_2)z\right)|\quad\textrm{ is constant}.
\]
Let 
\begin{equation}\label{eq:base}
a(z):=\mathsf{Re}((\qQ - \la_+ I_2)z)\qquad \mbox{ and }\qquad b(z):=\mathsf{Im}((\qQ - \la_+ I_2)z).
\end{equation}
Note that
\begin{equation}\label{eq:coordenadas}
\mathsf{Re}(e^{i t \check\la} (\qQ - \la_+ I_2)z)=\cos(\check\la t)a(z)-\sin(\check\la t)b(z).
\end{equation}
Combining (\ref{eq:matrixexponentialrepresentation}) with (\ref{eq:coordenadas}) yields 
\begin{equation}\label{eq:omegalimitrepresentation}
e^{\hat\la t} e^{-\qQ t}z  = -\frac{1}{\check\la} \Big(\cos(\check\la t)a(z)-\sin(\check\la t)b(z)\Big). 
\end{equation}
As a consequence, the Pythagoras theorem yields
\begin{align}
|\mathsf{Re}(e^{i t \check\la} (\qQ - \la_+ I_2)z)|^2&=|\cos(\check\la t)a(z)-\sin(\check\la t)b(z)|^2
\nonumber\\
&=\cos^2(\check\la t)|a(z)|^2+
\sin^2(\check\la t)|b(z)|^2-
2\cos(\check\la t)\sin(\check\la t)\<a(z),b(z)\>.
\label{eq:norma}
\end{align}
\begin{rem}
Note that equation (\ref{eq:norma}) does not require any specific structure of $\qQ$. 
It only uses that $d=2$, $\qQ \in \RR^{2\otimes 2}$ and the existence 
of conjugate complex, non-real eigenvalues. 
For this (more general) case we state the following lemma. 
\end{rem}

\begin{lem}[Profile cutoff characterization by the absence of non-normal growth]\label{lem:equivale}
For $d=2$ the following statements are equivalent.
 \begin{enumerate}
\item[i)] The function 
$t\mapsto 
|\mathsf{Re}(e^{i t \check\la} (\qQ - \la_+ I_2)z)|$ is constant.
\item[ii)] $|a(z)|^2=|b(z)|^2$ and $\<a(z),b(z)\>=0$, where $a(z)$ and $b(z)$ are given in 
(\ref{eq:base}). 
\item[iii)] For some $R>0$
\[
\omega(z)\subset \{|u| = R\},
\]
where 
\[\om(z) = \{u\in \RR^2~|~e^{\hat\la t_n} e^{-\qQ t_n} z \ra u ~\mbox{ for some }(t_n)_{n\in \NN}, ~t_n \ra \infty\}.\]
\end{enumerate}
\end{lem}

\begin{proof}
The proofs of ii)$\Longrightarrow$  i) and ii) $\Longrightarrow$ iii) 
are immediately from (\ref{eq:norma}). 

\noindent We continue with i)$\Longrightarrow$ ii) and assume that $t\mapsto 
|\mathsf{Re}(e^{i t \check\la} (\qQ - \la_+ I_2)z)|$ is constant. 
Evaluating in $t_n=\frac{\pi n}{\check\la}$, $n\in \mathbb{N}$ we obtain
\[
|\mathsf{Re}(e^{i t_n \check\la} (\qQ - \la_+ I_2)z)|^2=|a(z)|^2.
\]
Now, we evaluate in $s_n=\frac{\pi+2n \pi}{2\check\la}$, $n\in \mathbb{N}$ and deduce
\[
|\mathsf{Re}(e^{i s_n \check\la} (\qQ - \la_+ I_2)z)|^2=|b(z)|^2.
\]
Hence $|a(z)|^2 = |b(z)|^2$ as in ii). 
Inserting the preceding equalities in (\ref{eq:norma}), we have for any $t\gqq 0$ 
\[
|a(z)|^2=
|\mathsf{Re}(e^{i t \check\la} (\qQ - \la_+ I_2)z)|^2=|a(z)|^2-2
\cos(\check\la t)\sin(\check\la t)\<a(z),b(z)\>.
\]
Since $\check\la \neq 0$ the latter implies $\<a(z),b(z)\>=0$.

We continue with iii) $\Longrightarrow$ ii). 
For the sequence $t_n = \frac{2\pi n}{\check\la}$, $n\in \NN$,
applied to (\ref{eq:omegalimitrepresentation}) 
yields $-\frac{a(z)}{\check\la} \in \om(z)$. This implies $R = |\frac{a(z)}{\check\la}|$. 
For the sequence $t_n = \frac{1}{\check\la}(2\pi n+ \frac{\pi}{2})$, $n\in \NN$, we have 
$\frac{b(z)}{\check\la}\in \om(z)$, and hence also $R = |\frac{b(z)}{\check\la}|$. 
This and $\check\la \neq 0$ gives $|a(z)| = |b(z)|$. 
For the inner product we use that (\ref{eq:matrixexponentialrepresentation}) 
and (\ref{eq:norma}) imply 
\begin{align*}
|e^{t \hat\la}
 e^{-\qQ t}z|^2&=\frac{1}{\check\la^2}|\cos(\check\la t)a(z)-\sin(\check\la t)b(z)|^2
\nonumber\\
&\hspace{-1.5cm}=\frac{1}{\check\la^2}\cos^2(\check\la t)|a(z)|^2
+\frac{1}{\check\la^2}\sin^2(\check\la t)|b(z)|^2-\frac{2}{\check\la^2}
\cos(\check\la t)\sin(\check\la t)\<a(z),b(z)\> \\
&\hspace{-1.5cm}=\cos^2(\check\la t)R^2
+\sin^2(\check\la t)R^2-\frac{2}{\check\la^2}
\cos(\check\la t)\sin(\check\la t)\<a(z),b(z)\> \\
&\hspace{-1.5cm}=R^2-\frac{2}{\check\la^2}
\cos(\check\la t)\sin(\check\la t)\<a(z),b(z)\>. 
\label{eq:norma2}
\end{align*}
For $t_n = \frac{\pi/4+ 2\pi n}{\check\la}$ in (\ref{eq:omegalimitrepresentation}) we have 
$\frac{\sqrt{2}}{2 \check\la} (b(z) - a(z)) \in \om(z).$ In addition, for this $t_n$ we obtain
\begin{align*}
R^2 = \lim_{n\ra \infty} |e^{t_n \hat\la}e^{-\qQ t}z|^2= R^2-\frac{4}{\check\la}\<a(z),b(z)\>,
\end{align*}
which yields that $\<a(z),b(z)\>= 0$. 
\end{proof}

With this result at hand we complete the discussion of the linear oscillator in the sequel.  

\subsubsection{{\bf{Subcritically damped linear oscillator}}: $\Delta = \gamma^2 - 4\kappa<0$}

Recall that the eigenvalues of $\qQ$ in the case of (\ref{eq:linearoscillator}) are given by
\[
\la_{-}=\hat\la-i\check\la\quad\textrm{ and }\quad \la_{+}=\hat\la+i\check\la,
\]
where 
\[
\hat\la=\frac{\gamma}{2}\quad \textrm{ and } \quad\check\la=\frac{\sqrt{4\kappa-\gamma^2}}{2}\neq 0.
\]
By (\ref{eq:liminfpos}) and (\ref{eq:limsupfin}) for the noises (1)-(5) Theorem \ref{th:linearwhite} implies window thermalization for any $0<p'\lqq p$ at time scale 
\[
\ft^x_\e = \frac{2}{\gamma} |\ln(\e)|\quad \mbox{ for any initial condition }\quad (x,y) \in \RR^2,\; (x, y) \neq (0, 0).
\]
In the sequel, by using the shift linearity property d) in Lemma \ref{lem:invariance} we exclude
the existence of a cutoff thermalization profile for any $1\lqq p'\lqq p$ and noises (1)-(4).
\begin{lem}\label{lem: subcritical oscillator} 
Let $1\lqq p'\lqq p$.
For any $\gamma>0$ and  $\kappa>0$ such that $\gamma^2-4\kappa<0$, there is no cutoff thermalization profile for any $(x,y)\neq (0,0)$.
\end{lem}
\begin{proof}
We apply Lemma \ref{lem:equivale} to (\ref{eq:Qlinearoscillator}).  
Recall $z = (x, y)^* \neq 0$. A straightforward calculation yields
\[
a(z) = \Big(-\frac{\gamma}{2}x - y, \kappa x + \frac{\gamma}{2} y\Big)^* 
\qquad \textrm{ and } \qquad
b(z) = - \frac{\sqrt{4 \kappa - \gamma^2}}{2} (x, y)^*.
\]
The condition $\<a(z),b(z)\>=0$ reads as 
\begin{align*}
0 =  -\Big(\frac{\gamma}{2}x + y\Big)x + \Big(\kappa x + \frac{\gamma}{2} y\Big) y 
= - \frac{\gamma}{2} x^2 - xy + \kappa x y + \frac{\gamma}{2} y^2 
= - \frac{\gamma}{2} x^2 + (\kappa-1) xy + \frac{\gamma}{2} y^2,
\end{align*}
that is, 
\begin{equation}\label{eq: orth}
\gamma x^2 - \gamma y^2 - 2 (\kappa-1) xy = 0. 
\end{equation}
Since $(x,y)\neq (0,0)$,
the preceding equality yields $x\neq 0$ and $y\neq 0$.
The condition $|a(z)|^2 = |b(z)|^2$ is equivalent to 
\begin{align*}
\frac{\gamma^2}{4} x^2 + y^2 + \gamma xy + \kappa^2 x^2 + \frac{\gamma^2}{4}y^2 + \kappa \gamma xy
&= \Big(\kappa - \frac{\gamma^2}{4}\Big) (x^2 +y^2). 
\end{align*}
Simplifying we obtain 
\begin{equation}\label{eq: iso}
(\gamma^2 + 2 \kappa (\kappa-1)) x^2 + (\gamma^2 + 2 (1-\kappa)) y^2 + 2 \gamma (\kappa+1) xy = 0. 
\end{equation}
For $\kappa = 1$ we have that (\ref{eq: orth}) yields $x^2 = y^2$. 
Substituting in (\ref{eq: iso}) we obtain 
\[
0 = \gamma^2 x^2 \pm 2 \gamma x^2 = (\gamma^2 \pm 2\gamma ) x^2.  
\]
Since $x^2>0$ and $\gamma >0$, the unique solution is $\gamma =2$, 
which implies $\gamma^2 - 4 \kappa = 4 - 4= 0$ and gives a contradiction to the subcritical damping $\gamma^2 - 4\kappa <0$ of this case. As a consequence $\kappa=1$ excludes profile thermalization for any parameters $\gamma>0$ and $(x, y)\neq (0,0)$. 
In the sequel, we assume $\kappa \neq 1$. Multiplying (\ref{eq: iso}) by $(\kappa -1)$ we have 
\begin{equation}\label{eq: iso2}
((\kappa -1) \gamma^2 + 2 \kappa (\kappa-1)^2) x^2 + ((\kappa -1) \gamma^2 - 2 (\kappa-1)^2) y^2 + 2 \gamma (\kappa-1)(\kappa+1) xy = 0. 
\end{equation}
Inserting the expression for $ 2 (\kappa-1) xy $ being given in (\ref{eq: orth}) into (\ref{eq: iso2}) we obtain 
\begin{align*}\label{eq: iso3}
0 &= ((\kappa -1) \gamma^2 + 2 \kappa (\kappa-1)^2) x^2 + ((\kappa -1) \gamma^2 - 2 (\kappa-1)^2) y^2 + (\kappa+1) \gamma^2 (x^2 - y^2)\nonumber\\
&= ((\kappa -1) \gamma^2 + 2 \kappa (\kappa-1)^2 + (\kappa +1) \gamma^2) x^2 + ((\kappa -1) \gamma^2 - 2 (\kappa-1)^2-  (\kappa + 1)\gamma^2)y^2\\
&= 2(\kappa( \gamma^2 + (\kappa-1)^2) x^2 -2 ( \gamma^2 + (\kappa-1)^2)y^2.
\end{align*}
Since $\gamma^2+(\kappa-1)^2>0$, we obtain
$\kappa x^2 -y^2=0$.
In other words, $y=\pm \sqrt{k}x$. Substituting in (\ref{eq: orth}) we have
\begin{align*}
0&=\gamma x^2-\gamma^2 y^2-2(\kappa-1)xy=\Big(\gamma-\gamma \kappa\mp2(\kappa-1)\sqrt{\kappa}\Big)x^2
\end{align*}
which implies $\gamma-\gamma \kappa\mp2(k-1)\sqrt{\kappa}=0$.
Hence $\gamma(1-\kappa)=\pm 2(\kappa-1)\sqrt{\kappa}$.
Since $\kappa \neq 1$ we have 
$\gamma=\mp 2\sqrt{\kappa}$ which implies $\gamma^2-4\kappa=0$ and gives a contradiction to the subcritical damping $\gamma^2-4\kappa<0$.
As a consequence for any $\kappa>0$ and $\gamma>0$ such that $\gamma^2-4\kappa<0$, there is no profile thermalization for any initial condition $(x,y)\neq (0,0)$.  
\end{proof}
This concludes the complete analysis of the linear oscillator (\ref{eq:linearoscillator}). 

\subsection{\textbf{Linear chain of oscillators in a thermal bath at low temperature}}\label{ss:linearchain}
\subsubsection{\textbf{Window cutoff thermalization for the linear chain of oscillators}}
Our results cover the setting of Jacobi chains of $n$ oscillators with nearest neighbor interactions coupled to heat baths at its two ends, as discussed in Section 4.1 
in \cite{RAQUEPAS} and Section 4.2 in \cite{JPS17}. For the sake of simplicity we show  window cutoff thermalization for $n$ oscillators with the Hamiltonian 
\begin{align*}
\hH:\RR^n & \times  \RR^n \ra \RR\\ 
&(p, q) \mapsto \hH(p, q) := \frac{1}{2} \sum_{i=1}^n p_i^2 + \frac{1}{2} \sum_{i=1}^{n} \gamma q_i^2  
+ \frac{1}{2} \sum_{i=1}^{n-1} \kappa (q_{i+1}- q_i)^2.
\end{align*}
Coupling the first and the $n$-th oscillator to a Langevin heat bath each 
with positive temperature $\e^2$ and positive coupling constants $\varsigma_1$ and $\varsigma_n$ yields 
for $X^\e = (X^{1, \e}, \dots, X^{2n, \e}) = (p^\e, q^\e) = (p^\e_1, \dots, p^\e_n, q^\e_1, \dots, q^\e_n)$ 
the system 
\begin{align*}
\ud X^\e_t = -\qQ X^\e_t \ud t + \e\ud L_t,
\end{align*}
where $\qQ$ is a $2n\times 2n$-dimensional real matrix of the following shape 
\begin{align*}
\qQ = \left(
\begin{array}{ccccc|ccccccc}
\varsigma_1 &0& \dots & 0&&\kappa + \gamma &-\kappa&0&0& \dots  & 0\\             
 0& 0 &\dots&0 &&-\kappa & 2\kappa + \gamma & - \kappa& 0&0 \dots &0\\             
&&\ddots&&&\dots&\ddots&&\ddots&\ddots&\\
&&\ddots&&&\dots&\ddots&&\ddots&\ddots&\\
 0&&\dots&0 &0 & \dots &&0 &-\kappa & 2\kappa + \gamma & - \kappa\\             
0&0& \dots & 0& \varsigma_n& 0& \dots &&0&-\kappa & \kappa + \gamma\\
\hline
-1&0&\dots &0&0& 0 &&& \dots & &0\\             
0&-1&\dots &0&0& &&&  &&\\             
\vdots&&\ddots &\ddots&&\vdots& && \ddots &&\vdots \\             
&&\ddots &-1&0& &&& &&\\             
0&&\dots &0&-1& 0 &&& \dots &&0\\             
\end{array}
\right),  
\end{align*}
and $L_t = (L^1_t, 0, \dots, L^n_t, 0, \dots, 0)^*$. 
Here $L^1, L^n$ are one dimensional independent L\'evy processes 
satisfying Hypothesis \ref{hyp:moments} for some $p>0$. 
By Section 4.1 in \cite{RAQUEPAS} $\qQ$ satisfies Hypothesis \ref{hyp:statibility}. 
Consequently by Theorem \ref{th:linearwhite} the system exhibits window cutoff thermalization 
for any initial condition $x\neq 0$. The presence of a thermalization profile depends highly on the choice 
of the parameters $\kappa$, $\varsigma_1$, $\varsigma_2$, $\gamma$. 
\hfill

\subsubsection{\textbf{Numerical example of a linear chain of oscillators}}
In the sequel, we set 
$\varsigma_1=\varsigma_n=\kappa=1$, $\gamma=0.01$ and $n=5$. 
The following computations are carried out in Wolfram Mathematica 12.1.
The interaction matrix $\qQ$ is given by 
\[\left[ \begin {array}{cccccccccc} 1&0&0&0&0&1.01&-1&0&0&0\\ \noalign{\medskip}0&0&0&0&0&-1&2.01&-1&0&0\\ \noalign{\medskip}0&0&0&0&0&0&-1&2.01&-1&0\\ \noalign{\medskip}0&0&0&0&0&0&0&-1&2.01&-1\\ \noalign{\medskip}0&0&0&0&1&0&0&0&-1&1.01\\ \noalign{\medskip}-1&0&0&0&0&0&0&0&0&0\\ \noalign{\medskip}0&-1&0&0&0&0&0&0&0&0\\ \noalign{\medskip}0&0&-1&0&0&0&0&0&0&0\\ \noalign{\medskip}0&0&0&-1&0&0&0&0&0&0\\ \noalign{\medskip}0&0&0&0&-1&0&0&0&0&0\end {array} \right] \]
with the following eigenvalues
\[
\left[
\begin{array}{l}
\la_1\\
\bar \la_1\\
\la_2\\
\bar \la_2\\
\la_3\\
\bar \la_3\\
\la_4\\
\bar \la_4\\
\la_5\\
\la_6
\end{array}
\right]
=
\left[
\begin{array}{l}
0.0263377 + 1.88656\cdot i\\ 0.0263377 - 1.88656\cdot i\\
0.104782 + 1.55549\cdot i\\
0.104782 - 1.55549\cdot i\\
0.234099 + 1.06262\cdot i\\ 
0.234099 - 1.06262\cdot i\\ 
0.395218 + 0.517319\cdot i\\ 
0.395218 - 0.517319\cdot i\\ 
0.452655 + 0.\cdot i\\ 
0.0264706 + 0.\cdot i
\end{array}
\right].
\]
Since we have $10$ complex eigenvalues, we obtain a base of $10$ eigenvectors $v_1,\bar v_1,v_2,$ $\bar v_2,v_3,\bar v_3,$ $v_4,\bar v_4, v_5,v_6$ where 
$v_1,v_2,v_3,v_4\in \CC^{10}\setminus\RR^{10}$  and $v_5,v_6\in \RR^{10}$, maintaining the natural ordering.
Hence for the initial condition $x$ we have the unique representation
\[
x=\sum\limits_{j=1}^{4}(c_j(x)v_j+\bar c_j(x)\bar v_j)+c_5(x)v_5+c_6(x)v_6,
\]
where $c_1(x),c_2(x),c_3(x),c_4(x)\in \CC$ and $c_5(x),c_6(x)\in \RR$.
We note that the minimum of real parts of the eigenvalues is taken by the eigenvalues
$\la_1,\bar \la_1$.
Let 
$\fq=\mathsf{Re}(\lambda_1)=0.0263377$ and 
$\theta=\arg(\lambda_1)=1.55684$.
Hence, for generic $x$ (not properly contained in any eigenspace) we have 
\begin{align*}
e^{\fq t}e^{-\qQ t}x &\approx 
e^{i\theta t}c_1(x) v_1+ \bar e^{i\theta t}
\bar c_1(x) \bar v_1+2e^{(\fq-0.0264706)t}\mathsf{Re}(c_6(x)v_6)\\
&=
2\mathsf{Re}(e^{i\theta t}c_1(x) v_1)+2e^{(\fq-0.0264706)t}\mathsf{Re}(c_6(x)v_6),
\end{align*}
where $c_6(x)$ is a constant depending on $x$ and  
\[
v_1=
\left[
\begin{array}{llllllllll}
 0.112319 - 0.0891416 \cdot i\\
 -0.448508 + 0.0287844 \cdot i\\
 0.579305 + 0. \cdot i\\
 -0.448508 + 0.0287844 \cdot i\\
 0.112319 - 0.0891416 \cdot i\\
 0.0464105 + 0.060184 \cdot i\\
 -0.0119363 - 0.237904 \cdot i\\
 -0.00428606 + 0.307009 \cdot i\\
 -0.0119363 - 0.237904 \cdot i\\
 0.0464105 + 0.060184 \cdot i
\end{array}
\right].
\]
The vector of $c_j=c_j(x)$ is given by  
\[
c(x)=(c_1,\bar c_1,c_2,\bar c_2,c_3,\bar c_3(x), c_4,\bar c_4, c_5,c_6)^*=
[v_1 | \bar v_1 | v_2 |\bar v_2 | v_3 | \bar v_3 | v_4 | \bar v_4 | v_5 | v_6]^{-1}x.
\]
For instance for $x=e_1=(1,0,\ldots,0)$ we obtain
\[
c(e_1)=
\left[
\begin{array}{llllllllll}
0.0800993 - 0.0495081\cdot i\\
0.0800993 + 0.0495081\cdot i\\
0.186213 - 0.0967681\cdot i\\
0.186213 + 0.0967681\cdot i\\
0.371378 - 0.158507\cdot i\\
0.371378 + 0.158507\cdot i\\
-0.0619613 + 0.787062\cdot i\\
-0.0619613 - 0.787062\cdot i\\
1.62062 + 5.543\cdot 10^{-16}\cdot i\\
1.2715 + 2.33866\cdot 10^{-16}\cdot i
\end{array}
\right].
\]
Hence
\[
c_1(e_1)v_1=\hat w+i\cdot \check w =\left[
\begin{array}{l}
0.00396485\\
0.00886695\\
0.0101247\\
0.0197063\\
0.0218997\\
0.0375945\\
0.0340029\\
-0.043929\\
0.12981\\
0.101846
\end{array}
\right]+i\cdot \left[
\begin{array}{l}
- 0.00793113\\
- 2.28066\cdot 10^{-9}\\
- 0.0169701\\
- 0.001468\\
- 0.0310825\\
- 0.00568992\\
0.0661107\\
- 0.0599756\\
- 0.0802337\\
- 0.0629493
\end{array}
\right]
\]
and consequently
\[
e^{\fq t}e^{-\qQ t}e_1\approx 
2\cos(\theta t)\hat w
-2\sin(\theta t)\check w.
\]
In the sequel, we check the existence of a thermalization profile for $p\gqq 1$.
A straightforward computations show that
 $|\hat w|=0.181073\neq 
0.140425=|\hat w|$ and $\<\hat w, \check{w}\>=-0.0130705\neq 0$. Hence Theorem \ref{th:normalgrowth} yields the absence of a thermalization profile.

Recall that
$e^{-\fq \ft_{\e}}=\e$. Hence 
\begin{align*}
\frac{e^{-\qQ  \ft_{\e}}x}{\e}&=e^{\fq \ft_{\e}}e^{-\qQ \ft_{\e}}\approx
(e^{i\theta  \ft_{\e}}c_1(x) v_1+ \bar e^{i\theta  \ft_{\e}}
\bar c_1(x) \bar v_1)+2e^{(\fq-0.0264706)   \ft_{\e}}\mathsf{Re}(c_6(x)v_6)\\
&=2\mathsf{Re}(e^{i\theta \ft_{\e}}c_1(x) v_1)+2\e^{\nicefrac{(0.0264706-\fq)}{\fq} }\mathsf{Re}(c_6(x)v_6)\\
&=2\mathsf{Re}(e^{i\theta \ft_{\e}}c_1(x) v_1)+2\e^{0.0001329 }\mathsf{Re}(c_6(x)v_6).
\end{align*}
The low order of the error is essentially due to the relative spectral gap $\nicefrac{(0.0264706-\fq)}{\fq}$.

\section{\textbf{Conceptual examples}}\label{s:conceptual}
In the sequel, we give mathematical examples 
illustrating typical features of linear systems. 
We start with a non-symmetric linear system with complex (conjugate) eigenvalues exhibiting always 
a thermalization profile. This is followed by an ad hoc example illustrating 
the sensitive dependence of a thermalization profile {on} the initial condition. 
Finally we provide an example of repeated eigenvalues, where a log-log correction 
in the thermalization time scale appears. 

\subsection{\textbf{Example: Leading complex eigenvalues do not exclude profile}}\label{ex:complejo}
Let
\[
\qQ:=\left( \begin{array}{cc} \lambda & \theta\\ -\theta & \lambda\end{array} \right) \qquad \textrm{ with }\quad \la>0\quad \textrm{ and }\quad \theta\not=0.
\]
The eigenvalues of $-\qQ$ are given by $-\la\pm i \theta$.
A straightforward computation yields
\[
e^{-\qQ t}=e^{-\la t}\left( \begin{array}{cc} \cos(\theta t) & -\sin(\theta t)\\ \sin(\theta t) & \cos(\theta t)\end{array} \right).
\]
Hence for any $z=(x,y)^*$ we have 
\[
|e^{\la t}e^{-\qQ t}z|=\left|\left( \begin{array}{cc} \cos(\theta t) & -\sin(\theta t)\\ \sin(\theta t) & \cos(\theta t)\end{array} \right)z\right|=|z|.
\]
As a consequence, Theorem \ref{th:profile} implies 
a thermalization profile for any initial value $z\neq 0$.

\subsection{\textbf{Example: The initial value strongly determines the cutoff}}\label{ex:suspension}

We consider an embedding of the linear oscillator (\ref{eq:linearoscillator}) in $\RR^3$. 
Assume the case of subcritical damping $\gamma^2<4\kappa$ for positive parameters $\gamma, \kappa, \la$, 
\[
\qQ:=\left( \begin{array}{cc}
\la & 0 \\
0 &  \qQ_1 \\ 
\end{array} \right) \qquad \mbox{ and }\quad \qQ_1 = \left(\begin{array}{cc}0 & -1\\ \kappa & \gamma  \end{array}\right).
\]
The matrix $\qQ_1$ is precisely the one for the linear oscillator (\ref{eq:Qlinearoscillator}) analyzed in Example~\ref{ex:osc}. A straightforward computation shows 
\[
e^{-\qQ t}=\left( \begin{array}{cc}
e^{-\la t} & 0 \\
0 &  e^{-\qQ_1 t}  \\ 
\end{array} \right).
\]
For any initial value $z=(z_1,0,0)$ with $z_1\neq 0$ we have $|e^{\la t}e^{-\qQ t}z|=|z_1|$ and therefore 
a thermalization profile is valid due to Theorem \ref{th:profile}. However, for any $z=(0,z_2,z_3)$ with $(z_2,z_3)\neq (0,0)$ we have $|e^{\frac{\gamma}{2}t}e^{-\qQ t}z|=|e^{-\qQ_1 t}(z_2,z_3)^*|$ which by the case of subcritical damping ($\gamma^2<4\kappa$) discussed in Subsection \ref{ex:osc} does not have a cutoff thermalization profile. Instead, by Theorem \ref{th:linearwhite} only window cutoff thermalization is valid. 
That is, the presence of a thermalization profile is sensitive with respect to the initial condition. 

In the sequel, we emphasize the presence of a threshold effect for the {existence} of a thermalization profile with respect to the parameters due to competing real parts of the eigenvalues. 
Let $z=(z_1,z_2,z_3)$ with $z_1\neq 0$, $(z_2, z_3) \neq (0, 0)$. 
If in addition $\frac{\gamma}{2}<\la$ we have
\[
\lim_{t \ra \infty} e^{\la t}|e^{-\qQ t}z| = |z_1|,
\]
which implies that $\om(z) \subset \{|u|= |z_1|\}$ and therefore by Theorem \ref{th:profile} a thermalization profile. However, if $\frac{\gamma}{2}\gqq\la$ we have
\[
\lim_{t\to \infty}e^{\frac{\gamma}{2}t}|e^{-\qQ t}z|=\lim_{t\ra\infty} e^{\frac{\gamma}{2} t} e^{-\la t}  |z_1|+ \lim_{t\to \infty}e^{\frac{\gamma}{2}t}|e^{-\qQ_1 t}(z_2,z_3)^*|,
\]
which is not constant, as discussed in Subsection \ref{ex:osc}, and has only window thermalization (by 
Theorem \ref{th:linearwhite}), but no profile due the negative result in Lemma \ref{lem: subcritical oscillator}. 

\subsection{\textbf{Example: Multiplicities in the Jordan block yield logarithmic corrections}}\label{ex:multiplicity}

Let $\qQ$ be a $d$-squared matrix with all its eigenvalues equal to $\la>0$.
Theorem 7.10 p.209 in \cite{APOSTOL} yields
\[
e^{-\qQ t}=e^{-\la t}\sum_{j=0}^{d-1}\frac{t^k}{k!}(-\qQ+\la I_d)^k.
\]
For $z\in \RR^d$, $z\neq 0$ let 
\[
l(z)=\max\{k\in\{0,\ldots,d-1\}: (-\qQ+\la I_d)^{k}z\neq 0\}.
\]
Then 
\begin{align*}
e^{\la t}e^{-\qQ t}z&=\sum_{k=0}^{d-1}\frac{t^k}{k!}(-\qQ+\la I_d)^kz=\sum_{k=0}^{l(z)}\frac{t^k}{k!}(-\qQ+\la I_d)^k z.
\end{align*}
On the one hand,
if $l(z)=0$ we have $e^{\la t}e^{-\qQ t}z=z$.
On the other hand, if $l(z)\gqq 1$ we obtain
\begin{align}\label{eq:tasamult}
\frac{e^{\la t}}{t^{l(z)}}e^{-\qQ t}z=\sum_{k=0}^{l(z)-1}\frac{t^{k-l(z)}}{k!}(-\qQ+\la I_d)^k z+ \frac{1}{l(z)!}(-\qQ+\la I_d)^{l(z)} z.
\end{align}
Hence
\[
\lim\limits_{t\to \infty}
\frac{e^{\la t}}{t^{l(z)}}e^{-\qQ t}z= \frac{1}{l(z)!}(-\qQ+\la I_d)^{l(z)} z\neq 0,
\]
and in this case $\ell(z)=l(z)+1 \gqq 2$, where $\ell(z)$ is the constant given in Lemma \ref{jara}. 
Due to Theorem \ref{th:profile} there is always a thermalization profile. 
However, if $\ell(z)\gqq 2$, the log-log correction in (\ref{eq:thermalization profile}) appears.  
Note that a log-log correction and the presence of a thermalization profile are independent properties. 
It is not complicated to construct an example with no thermalization profile and log-log correction. 

\section{\textbf{Extensions and applications}} \label{s:ext}

{This section contains the cutoff phenomenon for the relative entropy in the Brownian case, for the Wasserstein distance with stationary red noises and comments about the computational observation of the cutoff phenomenon. }

\subsection{\textbf{Cutoff thermalization in the relative entropy}}\label{ss:ext1}

{In this subsection we discuss the asymptotics in the explicit formula  (\ref{eq:explicitentropy}) of the relative entropy for general exponentially asymptotically stable $-\qQ$. }

The strongest notion of thermalization of interest is given in terms of the Kullback-Leibler divergence also called relative entropy.
For pure Brownian perturbations the marginals of 
\[
\ud X^{\e}_t(x)=-\qQ X^{\e}_t(x)\ud t+\e \sigma\ud B_t,\qquad X^\e_0=x,
\]
{for some $\si \in \RR^{d\times d}$} are known to be 
\[
X^{\e}_t(x)\stackrel{d}=N(e^{-\qQ t}x,\e^2 \Sigma_t),
\]
where 
\[
\Sigma_t=e^{-\qQ t}\Big(\int_{0}^{t}e^{\qQ s}
\sigma \sigma^* e^{\qQ^* s}\ud s \Big)e^{-\qQ^* t}
\] 
is a symmetric and non-negative definite square matrix, see
Proposition 3.5 in \cite{Pavli}.
Since $\qQ$ satisfies Hypothesis \ref{hyp:statibility}, we have 
$e^{-\qQ t}x\to 0$ and $\e^2\Sigma_t\to \e^2\Sigma_\infty$ as $t\to \infty$ which implies the existence of a unique limiting distribution $\mu^\e=N(0,\e^2\Sigma_\infty)$.
A priori, $\Sigma_t$ and $\Sigma_\infty$ may degenerate, however, Theorem \ref{th:linearwhite} applies for $p=\infty$ and Theorem \ref{th:profileabstract} is valid under condition (\ref{eq:jaka}).
If additionally we assume that $-\qQ$ and $\sigma$ are controllable, i.e.
$\textrm{Rank}[-\qQ,\sigma]=d$,
the matrices $\Sigma_t$  and  $\Sigma_\infty$ turn out to be non-singular{. Moreover, $\Sigma_\infty$ is the unique symmetric positive definite solution of the Lyapunov matrix equation
$
\qQ\Sigma_\infty+\Sigma_\infty \qQ^* =\sigma \sigma^*$.}
The relative entropy is given explicitly by formula (\ref{eq:explicitentropy}),
which we rewrite as
\[
H(X^\e_t(x)|\mu^\e)-\frac{1}{2\e^2}(e^{-\qQ t}x)^*\Sigma^{-1}_\infty e^{-\qQ t}x=\frac{1}{2}\left(
\mathrm{Tr}(\Sigma^{-1}_\infty \Sigma_t)-d+\ln\Big(\frac{\mathrm{det}(\Sigma_\infty)}{\mathrm{det}(\Sigma_t)}\Big)\right).
\]
For any $t_\e\to \infty$ as $\e \to 0$ we have that the error term in the right-hand side of the preceding equality tends to zero as $\e \to 0$. In the sequel, we analyze the asymptotic quadratic form 
\[
\frac{1}{2\e^2}(e^{-\qQ t}x)^*\Sigma^{-1}_\infty \big(e^{-\qQ t}x\big)=\frac{1}{2}\Big|\Sigma^{-1/2}_\infty \frac{e^{-\qQ t}x}{\e}\Big|^2.
\]
By Lemma \ref{jara} it has the spectral decomposition
\begin{equation}\label{eq:2newhartmandecomposition}
\lim_{t \to \infty} \left |\frac{e^{\fq t}}{t^{\ell-1}} e^{- \qQ t}x - \sum_{k=1}^{m} e^{i  t\theta_k} v_k\right | =0.
\end{equation}
For the scale $\ft^x_\e$ and $w_\e$ given in Theorem \ref{th:profile} we have 
\begin{equation}\label{eq:2newscaling}
\lim\limits_{\e \to 0} 
\frac{(\ft^x_\e+r\cdot w_\e)^{\ell-1} e^{-\fq (\ft^x_\e+r\cdot w_\e)}}{\e}=\fq^{1-\ell}e^{-\fq w r}.
\end{equation}
{Lemma \ref{lem:precisehartmandecomp} implies 
\begin{align*}
&\left|
\Big|\Sigma^{-1/2}_\infty \frac{e^{-\qQ (\ft_\e^x+r\cdot w_\e)}x}{\e}\Big|-\Big|\Sigma^{-1/2}_\infty 
\frac{(\ft^x_\e+r\cdot w_\e)^{\ell-1}}{\e e^{\fq (\ft^x_\e+r\cdot w_\e)}} \sum_{k=1}^{m} e^{i  (\ft^x_\e+r\cdot w_\e)\theta_k} v_k\Big|\right|\\
&\lqq |\Sigma^{-1/2}_\infty|\frac{(\ft^x_\e+r\cdot w_\e)^{\ell-1}}{\e e^{\fq (\ft^x_\e+r\cdot w_\e)}}
\left| \Big(e^{\fq (\ft^x_\e+r\cdot w_\e)}\frac{e^{-\qQ (\ft_\e^x+r\cdot w_\e)}x}{(\ft^x_\e+r\cdot w_\e)^{\ell-1}}-
 \sum_{k=1}^{m} e^{i  (\ft^x_\e+r\cdot w_\e)\theta_k} v_k\Big)\right|.
\end{align*}
Combining the preceding inequality }
with (\ref{eq:2newhartmandecomposition}) and (\ref{eq:2newscaling}), yields
\begin{align}\label{eq:superior}
\limsup_{\e\to 0}\Big|\Sigma^{-1/2}_\infty \frac{e^{-\qQ (\ft_\e^x+r\cdot w_\e)}x}{\e}\Big|
&=\fq^{1-\ell}e^{-\fq w r}\limsup_{\e \to 0}
\Big|\Sigma^{-1/2}_\infty 
\sum_{k=1}^{m} e^{i  (\ft^x_\e+r\cdot w_\e)\theta_k} v_k\Big|\\
&\lqq \fq^{1-\ell}e^{-\fq w r}
\Big|\Sigma^{-1/2}_\infty \Big|
\sum_{k=1}^{m} |v_k|\nonumber
\end{align}
and 
\begin{align}\label{eq:inferior}
\liminf_{\e\to 0}\Big|\Sigma^{-1/2}_\infty \frac{e^{-\qQ (\ft_\e^x+r\cdot w_\e)}x}{\e}\Big|
&=\fq^{1-\ell}e^{-\fq w r}\liminf_{\e \to 0}
\Big|\Sigma^{-1/2}_\infty 
 \sum_{k=1}^{m} e^{i  (\ft^x_\e+r\cdot w_\e)\theta_k} v_k\Big|\\
 &\gqq \fq^{1-\ell}e^{-\fq w r}
 |\Sigma^{1/2}_\infty|^{-1}
\liminf_{\e \to 0}
\Big| 
 \sum_{k=1}^{m} e^{i  (\ft^x_\e+r\cdot w_\e)\theta_k} v_k\Big|\nonumber\\
  &\gqq \fq^{1-\ell}e^{-\fq w r}
 |\Sigma^{1/2}_\infty|^{-1}
\liminf_{t \to \infty}
\Big| 
 \sum_{k=1}^{m} e^{i t\theta_k} v_k\Big|.\nonumber
\end{align}
{Hence the analogue of Theorem \ref{th:linearwhite} is valid for the relative entropy, that is,   
\[
\lim_{r\to \infty}
\limsup_{\e\to 0}\Big|\Sigma^{-1/2}_\infty \frac{e^{-\qQ (\ft_\e^x+r\cdot w_\e)}x}{\e}\Big|=0
\]
and  by (\ref{eq:belowabove}) in Lemma \ref{jara} 
\[
\lim_{r\to -\infty}
\liminf_{\e\to 0}\Big|\Sigma^{-1/2}_\infty \frac{e^{-\qQ (\ft_\e^x+r\cdot w_\e)}x}{\e}\Big|=\infty.
\]
}
Moreover, we have the analogue of Theorem \ref{th:profile} with the following modification. 
By (\ref{eq:superior}) and (\ref{eq:inferior})
the existence of a cutoff thermalization profile holds if and only if the geometric condition of $|\Sigma^{-1/2}_\infty\omega(x)|$
being contained in a sphere is satisfied, 
where $\omega(x)$ is given in (\ref{eq:omegalimit}). 
{Recall that the normal growth condition (\ref{e:non-normal growth}) in Theorem \ref{th:normalgrowth} under the non-resonance hypothesis (\ref{eq:rationallyindependent1}) in Remark \ref{rem:rationallyindependent1} 
is given by 
\[
({v_1},\hat v_2,\check v_2,\ldots,\hat v_{2n},\check v_{2n}) \quad \textrm{ being orthogonal and satisfying }\quad
 |\hat v_{2k}|=|\check v_{2k}|\quad \textrm{ for } \quad k=1,\ldots,n.
\]
This characterization of the thermalization profile changes for the relative entropy to the following \textit{weighted normal growth condition}:  
\begin{equation}\label{eq:weightednonnormal} 
\begin{split}
&(\Sigma^{-1/2}_\infty{v_1},\Sigma^{-1/2}_\infty\hat v_2,\Sigma^{-1/2}_\infty\check v_2,\ldots,\Sigma^{-1/2}_\infty\hat v_{2n},\Sigma^{-1/2}_\infty\check v_{2n}) \quad \textrm{ is orthogonal and satisfies }\\
& |\Sigma^{-1/2}_\infty\hat v_{2k}|=|\Sigma^{-1/2}_\infty\check v_{2k}|\quad \textrm{ for } \quad k=1,\ldots,n.
\end{split}
\end{equation}
In case of \eqref{eq:weightednonnormal} the thermalization profile is given by
\begin{equation}\label{eq:weightednonnormalprofile}
\ti \pP_x(r)=\fq^{1-\ell}e^{-\fq w r}|\Sigma^{-1/2}_\infty u|,
\end{equation}
where $u$ is any representative of  $\omega(x)$. To our knowledge, this result is original and not known in the literature due to the lack of Lemma \ref{jara}. 
}

\subsection{\textbf{Cutoff thermalization for small red and more general ergodic noises}}\label{ss:ext2}
In this subsection we show that our results remain intact if we replace the driving L\'evy noise by red noise or more general ergodic noises. 

In the sequel, we consider the generalized Ornstein-Uhlenbeck process $(X^\e_t(x))_{t\gqq 0}$ 
\begin{align}\label{eq:ouds}
\ud X^\e_t=-\qQ X^\e_t \ud t+ \e \ud U_t, \qquad X^\e_0=x,
\end{align}
where the matrix
$\qQ$ satisfies Hypothesis \ref{hyp:statibility}.
Equation \eqref{eq:ouds}
is
driven by (i) a stationary multidimensional Ornstein-Uhlenbeck process 
 $(U_t)_{t\gqq 0}$  given by
\begin{equation}\label{eq:subordinated}
\ud U_t=-\Lambda U_t \ud t+ \ud L_t, \qquad U_0\stackrel{d}{=}\ti\mu,
\end{equation}
where the matrix
$\Lambda$ satisfies Hypothesis \ref{hyp:statibility}, $L=(L_t)_{t\gqq 0}$ fulfills Hypothesis \ref{hyp:moments} for some $p>0$ and $U_0$ is independent of $L$.
We point out that $U_t\stackrel{d}{=}\tilde{\mu}$ for all $t\gqq 0$.
We stress that with more technical effort the subordinated linear process  $(U_t)_{t\gqq 0}$ can be replaced by virtually any ergodic (Feller-) Markov process which is  sufficiently integrable.
For illustration of the ideas we focus on the stationary driving noise given by (\ref{eq:subordinated}).
By the variation of constant formula we have
\[
X^\e_t=e^{-\qQ t}x+\e e^{-\qQ t}\int_{0}^{t}e^{\qQ s}\ud U_s=:e^{-\qQ t}x+\e \uU_t.
\]
Note that
\begin{align*}
\ud 
\left(
\begin{array}{c}
X^\e_t\\
U_t
\end{array}
\right)=
\Gamma_\e \left(
\begin{array}{c}
X^\e_t\\
U_t
\end{array}
\right) \ud t+\left(
\begin{array}{c}
\e \\
1
\end{array}
\right)\ud L_t.
\end{align*}
where 
\begin{equation}
\Gamma_\e:=\left(\begin{matrix}
-\qQ & -\e\Lambda\\
0    &  -\Lambda
\end{matrix}\right).
\end{equation}
Since $\qQ$ and $\Lambda$ satisfy Hypothesis \ref{hyp:statibility}
and $\Gamma_\e$ is an upper block matrix, we have that $\Gamma_\e$ also satisfies Hypothesis \ref{hyp:statibility}.
In particular, the vector process
$((X^\e_t, U_t))_{t\gqq 0}$ is an Ornstein-Uhlenbeck process and hence Markovian. 
As a consequence, Theorem 4.1 in \cite{SaYa} yields the existence and uniqueness of an
 invariant and limiting  distribution (for the weak convergence) $(X^\e_\infty,U_0)$ 
of $((X^\e_t, U_t))_{t\gqq 0}$.
Hence
$X^{\e}_\infty\stackrel{d}{=}\e \uU_\infty$.
We continue with the estimate of $\Wp(X^\e_t,X^{\e}_\infty)$.
For any $1\lqq p'\lqq p$ the analogous computations used in (\ref{eq:coupling}) yield
\begin{align*}
|\frac{\Wp(X^\e_t,X^{\e}_\infty)}{\e}- \frac{|e^{-\qQ t}x|}{\e}|\lqq \Wp(\uU_t,\uU_\infty).
\end{align*}
Hence cutoff thermalization occurs whenever $\Wp(\uU_t,\uU_\infty)\to 0$ as $t\to \infty$.
Properties a), b) and d) of Lemma \ref{lem:invariance} imply
\begin{align}\label{eq:ut}
\Wp( \uU_t, \uU_\infty)&=
\Wp(X^{1}_t-e^{-\qQ t}x,X^1_\infty)\nonumber\\
&\lqq \Wp(X^1_t-e^{-\qQ t}x,X^1_\infty-e^{-\qQ t}x)+\Wp(X^1_\infty-e^{-\qQ t}x,X^1_\infty)
\nonumber\\
&=
\Wp(X^1_t,X^1_\infty)+|e^{-\qQ t}x|.
\end{align} 
We point out that the vector process $(X^1_t,U_t)_{t\gqq 0}$ is a Markov process. 
Since in general projections of Markovian processes are not Markovian, we study the process $X_t$ in more detail.
Due to the triangle structure we have the dependences $X^1_t(x,U_0)$ and $U_t(U_0)$. In case of initial data $(x,u)$ instead of 
$(x,U_0)$ in \eqref{eq:ouds} and 
\eqref{eq:subordinated}  we write (for $\e=1)$ $X^1_t(x,u)$
and $U_t(u)$.
Analogously to the total variation distance, Theorem~5.2 in \cite{devro}
the Wasserstein exhibits a contraction property which for completeness is shown here.
By the contraction property g) in Lemma \ref{lem:invariance} for the projection $T(x,u)=x$ 
 we have
\begin{equation}\label{e:ineq27}
\Wp(X^1_t(x,U_0),X^1_\infty)\lqq 
\Wp((X^1_t(x,U_0),U_t(U_0)),(X^1_\infty,U_t(U_0)))
\end{equation}
for any $1\lqq p'\lqq p$.
We note that
\begin{equation}
\Wp((X^1_t(x,U_0),U_t(U_0)),(X^1_\infty,U_t(U_0)))=\Wp((X^1_t(x,U_0),U_t(U_0)),(X^1_\infty,U_0)).
\end{equation}
Lemma \ref{lem:erg5} applied to the vector-valued Ornstein-Uhlenbeck process
$((X^1_t(x,U_0),U_t(U_0)))_{t\gqq 0}$
instead of 
$(X^\e_t(x))_{t\gqq 0}$ where $\Gamma_1$ replaces $\qQ$ yields the limit
\[
\lim\limits_{t\to \infty}
\Wp((X^1_t(x,U_0),U_t(U_0)),(X^1_\infty,U_0))=0.
\]
The preceding limit with the help of \eqref{e:ineq27} implies 
\[
\lim\limits_{t\to \infty}
\Wp(X^1_t(x,U_0),X^1_\infty)=0.
\]
Therefore the cutoff thermalization behavior of the
Ornstein-Uhlenbeck driven system (\ref{eq:ouds}) is the same as the white noise driven system (\ref{eq:linear}) given in Theorem \ref{th:profile} and Theorem \ref{th:linearwhite}. This is not surprising since the shift-linearity property of the Wasserstein distance for $p\gqq 1$ cancels out the specific invariant distribution.

\subsection{\textbf{Conditions on $\e$ for the observation of the cutoff on a fixed interval $[0, T]$}}\label{ss:ext3}

This subsection provides bounds on the size of $\e$ in order to observe cutoff on a fixed (large) interval $[0, T]$. Similar observations 
have been made in Section IV of \cite{BOK10} in order illustrate the optimal tuning of the parameter $\e$. 

Our main results contain the time scale $\ft_\e\to \infty$ as $\e\to 0$ at which thermalization occurs. However, the computational resources can only cover up to finite time horizon $T>0$. In the sequel, we line out estimates on the smallest of $\e$ in order to  observe the cutoff thermalization  before time $\nicefrac{T}{2}\in [0,T]$,
that is, $T\gqq 2\ft^x_\e$ for $\e \ll 1$. 
In other words, we have the lower bound 
\begin{equation}\label{eq:lowerbound2020}
\e\gqq e^{-(\nicefrac{\fq T}{2})}.
\end{equation}
Since our results are asymptotic, it is required that 
$\e<\e_0$, where $\e_0$ typically depends of 
$\qQ,x$ and $\EE[|\oO_\infty|]$.

Given an error $\eta>0$. 
In the light of estimates (\ref{eq:coupling}) ($p\gqq 1$) and
(\ref{eq:couplingp}) ($p\in (0,1)$) we carry out the following error analysis. 
In the sequel, we always consider a generic initial condition $x$.
Formula (\ref{eq:constanteslimite}) in item (4) of Remark \ref{rem:explanationformula} 
yields the following upper bound of the error 
\begin{align}\label{eq:eta1}
C_0 \EE[|\oO_\infty|]\e\lqq \nicefrac{\eta}{2},
\end{align}
where the constant $C_0$ is given in terms of the spectral gap in (\ref{eq:initialgrowth}) and an upper bound
of $\EE[|\oO_\infty|]$ is expressed in terms of the noise parameters in the estimate (\ref{eq:invariantmoment}). 
By Remark \ref{rem:windows} item (2) we have an error of order
\begin{equation}\label{eq:eta2}
K(x)\e^{\mathfrak{g}/\fq}\lqq \nicefrac{\eta}{2},
\end{equation}
where the spectral gap $\mathfrak{g}$ is given in (\ref{eq:ordendeconvergencia}) and the constant $K(x)$ is estimated in (\ref{eq:krepresentation}).
Combining 
(\ref{eq:lowerbound2020})
(\ref{eq:eta1}) and (\ref{eq:eta2}) and solving for $\e$ yields
\begin{equation}\label{eq:ebound}
e^{-(\nicefrac{\fq T}{2})}\lqq 
\e\lqq \min\left\{\frac{\eta}{2C_0 \EE[|\oO_\infty|]},
\left(\frac{\eta}{2K(x)}\right)^{{\mathfrak{q}/\mathfrak{g}}}\right\}.
\end{equation}

\appendix
{
\section{Proof of Lemma \ref{lem:invariance} (Properties of the Wasserstein distance)}\label{ap:A}
{
\begin{proof}[Proof of Lemma \ref{lem:invariance}]
Property a) is shown for $p\gqq 1$ in \cite{Vi09} p. 94. The proof for $p\in (0, 1)$ follows by the same reasoning with the help of the subadditivity of the map $\RR_+\ni r \ra r^p\in \RR_+$. 
Item b) is straightforward for any $p>0$ due to the translation invariance in formula (\ref{def:wasserstein}). 
The homogeneity property of item c) follows directly from (\ref{def:wasserstein}) for any $p>0$. \\
In the sequel we show item d). Since we are not aware of a proof in the literature the statement is shown here. Synchronous replica $(U_1,U_1)$ with joint law
$\Pi(\ud u, \ud u)$ (natural coupling) yields the upper bound for any $p>0$ as follows
\begin{align}\label{eq:cotasuperior1}
\W(u_1+U_1,U_1)\lqq \left(\int_{\RR^d \times \RR^d} |u_1+u-u|^p\Pi(\ud u,\ud u)\right)^{\min\{1,1/p\}}=|u_1|^{\min\{1,p\}}.
\end{align}
We continue with the lower bound for $p\gqq 1$. Let $\pi$ any coupling (joint law) between $u_1+U_1$ and $U_1$.
Note that
\[
\int_{\RR^d \times \RR^d}(u-v)\pi(\ud u,\ud v)=\int_{\RR^d \times \RR^d}u\pi(\ud u,\ud v)-
\int_{\RR^d \times \RR^d}v\pi(\ud u,\ud v)
=\EE[u_1+U_1]-\EE[U_1]=u_1.
\]
Then the triangular inequality yields 
\[
|u_1|=\Big|\int_{\RR^d \times \RR^d}(u-v)\pi(\ud u,\ud v)\Big|\lqq 
\int_{\RR^d \times \RR^d}|u-v|\pi(\ud u,\ud v).
\]
Minimizing over all possible coupling  between $u_1+U_1$ and $U_1$ we obtain
\begin{equation}\label{eq:cotainferior1}
|u_1|\lqq \mathcal{W}_1(u_1+U_1,U_1).
\end{equation}
For $p\gqq 1$, Jensen's inequality with the help of (\ref{eq:cotasuperior1}) and (\ref{eq:cotainferior1})
yields 
\[ 
|u_1|\lqq \mathcal{W}_1(u_1+U_1,U_1)\lqq \W(u_1+U_1,U_1)\lqq |u_1|,
\]
such that $\W(u_1+U_1,U_1)=|u_1|$.\\
For $p\in(0,1)$, the triangle inequality  and the translation invariance b) imply
 \begin{align*}
|x|^p=\W(x,0)&\lqq \W(x,x+U)+\W(x+U,U)+\W(U,0)\\
&=\W(x+U,U)+2\EE[|U|^p],
\end{align*}
and hence
\begin{equation}\label{eq:desigualdadinferior}
\W(x+U,U)\gqq |x|^p-2\EE[|U|^p].
\end{equation}
Combining (\ref{eq:cotasuperior1}) and (\ref{eq:desigualdadinferior}) we obtain (\ref{ec:cotaabajop01}). This finishes the proof of item d). \\
Property e) is straightforward. The characterization in item f) is proven Theorem~6.9 in \cite{Vi09} for $p\gqq 1$. For $p\in (0,1)$ we refer to Remark~1.4 in \cite{FJR20}. \\
For completeness we give a proof of item g). We apply the Kantorovich duality (Theorem~5.10 p. 57-58 in \cite{Vi09}) for the cost function $c(x,y)=|x-y|^p$ for any $x,y\in \RR^k$ and some $p>0$.
Let $\widetilde{X}=T(X)$ and $\widetilde{Y}=T(Y)$.
By item iii) of Theorem~5.10 in
\cite{Vi09} we have
\begin{equation*}
\begin{split}
\W(\widetilde{X},\widetilde{Y})=\max\limits_{(\psi,\phi)}\bigg(\mathbb{E}[\phi(\widetilde{Y})-\psi(\widetilde{X})]\bigg)^{\min\{1, \frac{1}{p}\}},
\end{split}
\end{equation*}
where the maximum is running over all integrable functions $\psi$ and $\phi$ such that  
\begin{equation}\label{eq: cost-inequality}\phi(x)-\psi(y)\lqq |x-y|^p\end{equation} for all $x,y\in \RR^k$. 
In addition, item iii) of Theorem~5.10 in
\cite{Vi09} states the existence of a respective maximizer $(\phi_*, \psi_*)$. 
The preceding equality yields
\begin{equation}\label{eq:kanto1}
\begin{split}
\W(\widetilde{X},\widetilde{Y})
&=\bigg(\mathbb{E}[\phi_*(\widetilde{Y})-\psi_*(\widetilde{X})]\bigg)^{\min\{1, \frac{1}{p}\}}
=\bigg(\mathbb{E}[\phi_*(T(Y))-\psi_*(T(X))]\bigg)^{\min\{1, \frac{1}{p}\}}\\
&=\bigg(\mathbb{E}[\phi_*\circ T(Y)-\psi_*\circ T(X)]\bigg)^{\min\{1, \frac{1}{p}\}}.
\end{split}
\end{equation}
Using (\ref{eq: cost-inequality}) and the fact that $T$ is Lipschitz continuous with Lipschitz constant $1$, we have for any $u,v\in \RR^d$ 
\[
\phi_*\circ T(u)-\psi_*\circ T(v)=\phi_*(T(u))-\psi_*(T(v))\lqq |T(u)-T(v)|^p\lqq |u-v|^p
\]
and hence
\begin{equation}\label{eq:kanto2}
\W({X},{Y})\gqq \bigg(\mathbb{E}[\phi_*\circ T(Y)-\psi_*\circ T(X)]\bigg)^{\min\{1, \frac{1}{p}\}}.
\end{equation}
The statement of item g) is a direct consequence of \eqref{eq:kanto1} and \eqref{eq:kanto2}.
\end{proof} 
}

\begin{proof}[Proof of Lemma   
 \ref{lem:totalvariation}]
Let $p'\in (0,p]$. Given $U_1,U_2$ let
 $\pi$ be some joint distribution of $(U_1,U_2)$.
By definition
\begin{align*}
\Wp(U_1,U_2)&=\inf_{\pi}\int_{\RR^d\times \RR^d} |u_1-u_2|^{p'} \pi(\ud u_1,\ud u_2)\\
&=\inf_{\pi}\int_{\RR^d\times \RR^d} |u_1-u_2|^{p'} \ind\{u_1\neq u_2\}\pi(\ud u_1,\ud u_2)\\
&\lqq \int_{\RR^d\times \RR^d} |u_1-u_2|^{p'} \ind\{u_1\neq u_2\}\pi(\ud u_1,\ud u_2).
\end{align*}
Note that for any $u_1,u_2\in \RR^d$
\begin{align*}
|u_1-u_2|^{p'}\lqq \ind\{|u_1-u_2|\lqq 1\}+|u_1-u_2|^{p}1(|u_1-u_2|>1)\lqq 1+|u_1-u_2|^{p}
\end{align*}
and 
$|u_1-u_2|^{p'} \ind\{u_1\neq u_2\}\to \ind\{u_1\neq u_2\}$ as $p'\to 0$ for any $u_1,u_2\in \RR^d$.
By the dominated convergence theorem we obtain
\begin{align*}
\lim\limits_{p'\to 0} \Wp(U_1,U_2)
&\lqq \lim\limits_{p'\to 0}
\int_{\RR^d \times \RR^d}  
|u_1-u_2|^{p'} 
\ind\{u_1\neq u_2\}\pi(\ud u_1,\ud u_2)\\
&= \int_{\RR^d \times \RR^d}  \ind\{u_1\neq u_2\}\pi(\ud u_1,\ud u_2).
\end{align*}
Minimizing over all joint distributions $\pi$ of $U_1,U_2$ we obtain
\[
\lim\limits_{p'\to 0} \Wp(U_1,U_2)\lqq d_{\mathrm{TV}}(U_1,U_2).
\]
Moreover, the dominated convergence theorem also yields the lower bound
\begin{align*}
 \int_{\RR^d \times \RR^d}  \ind\{u_1\neq u_2\}\pi(\ud u_1,\ud u_2)
 &=
\lim\limits_{p'\to 0}
\int_{\RR^d \times \RR^d}  
|u_1-u_2|^{p'} 
\ind\{u_1\neq u_2\}\pi(\ud u_1,\ud u_2)\\
&\gqq 
\lim\limits_{p'\to 0}
\Wp(U_1,U_2).
\end{align*}
Minimizing $\pi$ as above we deduce
\[
d_{\mathrm{TV}}(U_1,U_2)\gqq 
\lim\limits_{p'\to 0}
\Wp(U_1,U_2)
\]
and consequently 
\[
d_{\mathrm{TV}}(U_1,U_2)
=\lim\limits_{p'\to 0}
\Wp(U_1,U_2).
\]
\end{proof}

\section{Proof of Theorem \ref{th:profile}
and Theorem \ref{th:profileabstract} (Cutoff thermalization profile).
}\label{ap:B}
The following proposition presents the core arguments of the subsequent proofs of 
Theorem \ref{th:profile} and Theorem \ref{th:profileabstract}.
\begin{prop}\label{prop:replacement}
For any $0<p'\lqq p$ it follows
\begin{align}\label{eq:limitesuperiorformula}
\limsup_{\e\to 0}\frac{\Wp(X_{\ft^x_\e+r\cdot w_\e}^\e(x), \mu^\e)}{\e^{\min\{p',1\}}}=\limsup_{\e\to 0}\Wp\big(
\frac{(\ft^x_\e+r\cdot w_\e)^{\ell-1}}{\e e^{\fq (\ft^x_\e+r\cdot w_\e)}} \sum_{k=1}^{m} e^{i  (\ft^x_\e+r\cdot w_\e)\theta_k} v_k
 +  \oO_\infty,  \oO_\infty\big)
\end{align}
and
\begin{align}
\label{eq:limiteinferiorformula}
\liminf_{\e\to 0}\frac{\Wp(X_{\ft^x_\e+r\cdot w_\e}^\e(x), \mu^\e)}{\e^{\min\{p',1\}}}=\liminf_{\e\to 0}\Wp\big(
\frac{(\ft^x_\e+r\cdot w_\e)^{\ell-1}}{\e e^{\fq (\ft^x_\e+r\cdot w_\e)}} \sum_{k=1}^{m} e^{i  (\ft^x_\e+r\cdot w_\e)\theta_k} v_k
 +  \oO_\infty,  \oO_\infty\big).
\end{align}
In particular, the limit
\begin{align}\label{eq:distancia}
\lim\limits_{\e\to 0}\frac{\Wp(X_{\ft^x_\e+r\cdot w_\e}^\e(x), \mu^\e)}{\e^{\min\{p',1\}}}\qquad \textrm{ exists}
\end{align}
if and only if
the limit
\begin{align}\label{eq:equivalentedistancia}
\lim_{\e\to 0}\Wp\big(
\frac{(\ft^x_\e+r\cdot w_\e)^{\ell-1}}{\e e^{\fq (\ft^x_\e+r\cdot w_\e)}} \sum_{k=1}^{m} e^{i  (\ft^x_\e+r\cdot w_\e)\theta_k} v_k
 +  \oO_\infty,  \oO_\infty\big)\qquad \textrm{exists}.
\end{align}
\end{prop}

%

\begin{proof}[Proof of Proposition \ref{prop:replacement}]
Let $0<p'\lqq p$.
We first treat the case $0< p'\lqq 1 $.
By (\ref{eq:couplingp}),
for any $0< p'\lqq 1 $, we have 
\begin{align*}
\left|
\frac{\Wp(X_t^\e(x), \mu^\e)}{\e^{p'}} -\Wp\big(\frac{e^{-\qQ t} x}{\e} +  \oO_\infty,  \oO_\infty\big)\right|
\lqq \Wp(\oO_t, \oO_\infty).
\end{align*}
We continue with  the case $
p'\gqq 1     $.
By (\ref{eq:coupling}) and property d) in Lemma  \ref{lem:invariance} for any $p'\gqq 1$, we obtain
\begin{align*}
\left|
\frac{\Wp(X_t^\e(x), \mu^\e)}{\e} -\Wp\big(\frac{e^{-\qQ t} x}{\e} +  \oO_\infty,  \oO_\infty\big)\right|
\lqq \Wp(\oO_t, \oO_\infty).
\end{align*}
Combining the preceding inequalities we obtain for any $0<p'\lqq p$
\begin{align}\label{eq: continuidad666}
\left|
\frac{\Wp(X_t^\e(x), \mu^\e)}{\e^{\min\{p',1\}}} -\Wp\big(\frac{e^{-\qQ t} x}{\e} +  \oO_\infty,  \oO_\infty\big)\right|
\lqq \Wp(\oO_t, \oO_\infty).
\end{align}
Let $\ft^x_\e$ be the time scale given in Theorem \ref{th:profile} and $w_\e\to w>0$, as  
 $\e \to 0$.
By (\ref{eq: OUthermalizationW}) we have $\Wp(\oO_t, \oO_\infty)\to 0$ whenever $t\to \infty$.
Consequently,
\[
\limsup\limits_{\e\to 0}\frac{\Wp(X^\e_{\ft^x_\e+r\cdot w_\e}(x), X^\e_\infty)}{\e^{\min\{p',1\}}}=
\limsup_{\e\to 0}\Wp\big(\frac{e^{-\qQ (\ft^x_\e+r\cdot w_\e)} x}{\e} +  \oO_\infty,  \oO_\infty\big)
\]
and 
\[
\liminf\limits_{\e\to 0}\frac{\Wp(X^\e_{\ft^x_\e+r\cdot w_\e}(x), X^\e_\infty)}{\e^{\min\{p',1\}}}=
\liminf_{\e\to 0}\Wp\big(\frac{e^{-\qQ (\ft^x_\e+r\cdot w_\e)} x}{\e} +  \oO_\infty,  \oO_\infty\big).
\]
In the sequel, we study the asymptotics of the  drift term $\frac{e^{-\qQ t} x}{\e}$.
By Lemma \ref{jara}  it has the spectral decomposition
\begin{equation}\label{eq:newhartmandecomposition}
\lim_{t \to \infty} \left |\frac{e^{\fq t}}{t^{\ell-1}} e^{- \qQ t}x - \sum_{k=1}^{m} e^{i  t\theta_k} v_k\right | =0.
\end{equation}
A straightforward calculation shows
\begin{equation}\label{eq:newscaling}
\lim\limits_{\e \to 0} 
\frac{(\ft^x_\e+r\cdot w_\e)^{\ell-1} e^{-\fq (\ft^x_\e+r\cdot w_\e)}}{\e}=\fq^{1-\ell}e^{-\fq w r}.
\end{equation}
With the help of the spectral decomposition (\ref{eq:newhartmandecomposition})
and the triangle inequality we have
\begin{align}\label{eq:replacement}
\Wp&\big(\frac{e^{-\qQ (\ft^x_\e+r\cdot w_\e)}x}{\e} +  \oO_\infty,  \oO_\infty\big)
\nonumber\\
&\lqq 
\Wp\big(
\frac{(\ft^x_\e+r\cdot w_\e)^{\ell-1}}{\e e^{\fq (\ft^x_\e+r\cdot w_\e)}} \sum_{k=1}^{m} e^{i  (\ft^x_\e+r\cdot w_\e)\theta_k} v_k
 +  \oO_\infty,  \oO_\infty\big)
 \nonumber\\
 &\qquad +
 \Wp(\big(
\frac{(\ft^x_\e+r\cdot w_\e)^{\ell-1}}{\e e^{\fq (\ft^x_\e+r\cdot w_\e)}} \sum_{k=1}^{m} e^{i  (\ft^x_\e+r\cdot w_\e)\theta_k} v_k
-\frac{e^{-\qQ (\ft^x_\e+r\cdot w_\e)}x}{\e}\big)
 +  \oO_\infty,  \oO_\infty\big)
\nonumber\\
&=\Wp\big(
\frac{(\ft^x_\e+r\cdot w_\e)^{\ell-1}}{\e e^{\fq (\ft^x_\e+r\cdot w_\e)}} \sum_{k=1}^{m} e^{i  (\ft^x_\e+r\cdot w_\e)\theta_k} v_k
 +  \oO_\infty,  \oO_\infty\big)+R^x_\e,
\end{align}
where 
\[
R^x_\e:=\Wp(\big(
\frac{(\ft^x_\e+r\cdot w_\e)^{\ell-1}}{\e e^{\fq (\ft^x_\e+r\cdot w_\e)}} \sum_{k=1}^{m} e^{i  (\ft^x_\e+r\cdot w_\e)\theta_k} v_k
-\frac{e^{-\qQ (\ft^x_\e+r\cdot w_\e)}x}{\e}\big)
 +  \oO_\infty,  \oO_\infty\big).
\]
On the other  hand, analogous reasoning yields
\begin{align}\label{eq:Repsilon}
\Wp\big(
\frac{(\ft^x_\e+r\cdot w_\e)^{\ell-1}}{\e e^{\fq (\ft^x_\e+r\cdot w_\e)}} \sum_{k=1}^{m} e^{i  (\ft^x_\e+r\cdot w_\e)\theta_k} v_k
 +  \oO_\infty,  \oO_\infty\big)
 \lqq \Wp\big(\frac{e^{-\qQ (\ft^x_\e+r\cdot w_\e)}x}{\e} +  \oO_\infty,  \oO_\infty\big)+R^x_\e.
\end{align}
Combining (\ref{eq:replacement}) and (\ref{eq:Repsilon}) we have
\begin{align}\label{eq:ecuacioncomparadas}
\bigg|
\Wp\big(\frac{e^{-\qQ (\ft^x_\e+r\cdot w_\e)}x}{\e} +  \oO_\infty,  \oO_\infty\big)-
\Wp\big(
\frac{(\ft^x_\e+r\cdot w_\e)^{\ell-1}}{\e e^{\fq (\ft^x_\e+r\cdot w_\e)}} \sum_{k=1}^{m} e^{i  (\ft^x_\e+r\cdot w_\e)\theta_k} v_k
 +  \oO_\infty,  \oO_\infty\big)
\bigg|
\lqq R^x_\e.
\end{align}
In the sequel, we show $R^x_\e\to 0$ as $\e \to 0$. By continuity of $\Wp$ it 
is enough to prove 
\[
\big|\frac{(\ft^x_\e+r\cdot w_\e)^{\ell-1}}{\e e^{\fq (\ft^x_\e+r\cdot w_\e)}} \sum_{k=1}^{m} e^{i  (\ft^x_\e+r\cdot w_\e)\theta_k} v_k
-\frac{e^{-\qQ (\ft^x_\e+r\cdot w_\e)}x}{\e}\big|\to 0, \quad \e\to 0.
\]
By limit (\ref{eq:newhartmandecomposition}) and  limit (\ref{eq:newscaling}) we obtain for $\e \ra 0$ 
\begin{align*}
\big|\frac{(\ft^x_\e+r\cdot w_\e)^{\ell-1}}{\e e^{\fq (\ft^x_\e+r\cdot w_\e)}} \sum_{k=1}^{m} e^{i  (\ft^x_\e+r\cdot w_\e)\theta_k} v_k
-\frac{e^{-\qQ (\ft^x_\e+r\cdot w_\e)}x}{\e}\big|\\
&\hspace{-6cm}=
\frac{(\ft^x_\e+r\cdot w_\e)^{\ell-1}}{\e e^{\fq (\ft^x_\e+r\cdot w_\e)}}
\big|\sum_{k=1}^{m} e^{i  (\ft^x_\e+r\cdot w_\e)\theta_k} v_k
-\frac{e^{\fq (\ft^x_\e+r\cdot w_\e)}e^{-\qQ (\ft^x_\e+r\cdot w_\e)}x}{(\ft^x_\e+r\cdot w_\e)^{\ell-1}}\big|\to 0.
\end{align*}
Inequality (\ref{eq:ecuacioncomparadas}) with the help of the preceding limit yields
\begin{align*}
&\limsup_{\e\to 0}\Wp\big(\frac{e^{-\qQ (\ft^x_\e+r\cdot w_\e)}x}{\e} +  \oO_\infty,  \oO_\infty\big)\\
&=
\limsup_{\e\to 0}\Wp\big(
\frac{(\ft^x_\e+r\cdot w_\e)^{\ell-1}}{\e e^{\fq (\ft^x_\e+r\cdot w_\e)}} \sum_{k=1}^{m} e^{i  (\ft^x_\e+r\cdot w_\e)\theta_k} v_k
 +  \oO_\infty,  \oO_\infty\big).
\end{align*}
The analogous result holds for the lower limit. 
\end{proof}
\begin{proof}[Proof of Theorem \ref{th:profileabstract}]
We start with the proof of i) $\Longrightarrow$ ii). Statement 
i) implies that
for any $\lambda\gqq 0$
the map 
\[
\omega(x) \ni u\mapsto \Wp(\lambda u+\oO_\infty,\oO_\infty)
\]
defined on
\[
\omega(x)=\left\{\textrm{accumulation points of $\sum_{k=1}^{m} e^{i  t\theta_k} v_k$ as 
$t\to \infty$ }\right\}.
\]
is constant.
Hence
\begin{align}\label{eq:contlimitesuperior}
\limsup\limits_{\e\to 0}\Wp\Big(
\frac{(\ft^x_\e+r\cdot w_\e)^{\ell-1}}{\e e^{\fq (\ft^x_\e+r\cdot w_\e)}} 
\sum_{k=1}^{m} e^{i  (\ft^x_\e+r\cdot w_\e)\theta_k} v_k
 +  \oO_\infty,  \oO_\infty\Big)
 =\Wp\big(\fq^{1-\ell}e^{-\fq w r}\hat u+\oO_\infty,\oO_\infty\big).
\end{align}
Indeed, by definition of upper limits,
there is a sequence $(\e_n)_{n\in \NN}$, $\e_n\to 0$, $n\to \infty$ for which the upper limit is the true limit, i.e.,
\begin{align*}
\limsup\limits_{\e\to 0}&\Wp\big(
\frac{(\ft^x_\e+r\cdot w_\e)^{\ell-1}}{\e e^{\fq (\ft^x_\e+r\cdot w_\e)}} 
\sum_{k=1}^{m} e^{i  (\ft^x_\e+r\cdot w_\e)\theta_k} v_k
 +  \oO_\infty,  \oO_\infty\big)
\\
&=
\lim\limits_{n\to \infty}\Wp\big(
\frac{(\ft^x_{\e_n}+r\cdot w_{\e_n})^{\ell-1}}{\e_{n} e^{\fq (\ft^x_{\e_{n}}+r\cdot w_{\e_{n}})}} 
\sum_{k=1}^{m} e^{i  (\ft^x_{\e_{n}}+r\cdot w_{\e_{n}})\theta_k} v_k
 +  \oO_\infty,  \oO_\infty\big).
\end{align*}
By (\ref{eq:newscaling}) we have that the sequence
\[
\left(
\frac{(\ft^x_{\e_n}+r\cdot w_{\e_n})^{\ell-1}}{\e_{n} e^{\fq (\ft^x_{\e_{n}}+r\cdot w_{\e_{n}})}} 
\sum_{k=1}^{m} e^{i  (\ft^x_{\e_{n}}+r\cdot w_{\e_{n}})\theta_k} v_k
\right)_{n\in \NN}
\]
is uniformly bounded. Then the Bolzano-Weierstrass theorem 
yields a subsequence $({\e_{n_j}})_{j\in \NN}$ of $(\e_{n})_{n\in \NN}$ such that
\[
\lim_{j\to \infty}\frac{(\ft^x_{\e_{n_j}}+r\cdot w_{\e_{n_j}})^{\ell-1}}{{\e_{n_j}} e^{\fq (\ft^x_{{\e_{n_j}}}+r\cdot w_{\e_{n_j}})}} 
\sum_{k=1}^{m} e^{i  (\ft^x_{{\e_{n_j}}}+r\cdot w_{{\e_{n_j}}})\theta_k} v_k
=\fq^{1-\ell}e^{-\fq w r}\hat u 
\]
for some  $\hat u\in \omega(x)$.
By continuity we deduce
\begin{align*}
\limsup\limits_{\e\to 0}&\Wp\big(
\frac{(\ft^x_\e+r\cdot w_\e)^{\ell-1}}{\e e^{\fq (\ft^x_\e+r\cdot w_\e)}} 
\sum_{k=1}^{m} e^{i  (\ft^x_\e+r\cdot w_\e)\theta_k} v_k
 +  \oO_\infty,  \oO_\infty\big)
\\
&=
\lim\limits_{n\to \infty}\Wp\big(
\frac{(\ft^x_{\e_n}+r\cdot w_{\e_n})^{\ell-1}}{\e_{n} e^{\fq (\ft^x_{\e_{n}}+r\cdot w_{\e_{n}})}} 
\sum_{k=1}^{m} e^{i  (\ft^x_{\e_{n}}+r\cdot w_{\e_{n}})\theta_k} v_k
 +  \oO_\infty,  \oO_\infty\big)\\
&= \lim\limits_{j\to \infty}\Wp\big(
\frac{(\ft^x_{\e_{n_j}}+r\cdot w_{\e_{n_j}})^{\ell-1}}{{\e_{n_j}} e^{\fq (\ft^x_{{\e_{n_j}}}+r\cdot w_{{\e_{n_j}}})}} 
\sum_{k=1}^{m} e^{i  (\ft^x_{{\e_{n_j}}}+r\cdot w_{{\e_{n_j}}})\theta_k} v_k
 +  \oO_\infty,  \oO_\infty\big)\\
&=\Wp(\fq^{1-\ell}e^{-\fq w r}\hat u+ \oO_\infty, \oO_\infty).
\end{align*}
An analogous reasoning also justifies
\begin{align}
\label{eq:contlimiteinferior}
\liminf\limits_{\e\to 0}\Wp\big(
\frac{(\ft^x_\e+r\cdot w_\e)^{\ell-1}}{\e e^{\fq (\ft^x_\e+r\cdot w_\e)}} 
\sum_{k=1}^{m} e^{i  (\ft^x_\e+r\cdot w_\e)\theta_k} v_k
 +  \oO_\infty,  \oO_\infty\big)
 =\Wp\big(\fq^{1-\ell}e^{-\fq w r}\check u+\oO_\infty,\oO_\infty\big),
\end{align}
where $\check u\in \omega(x)$. For $\lambda=\fq^{1-\ell}e^{-\fq w r}$ we obtain 
\begin{equation}\label{eq:profileformula66}
\lim\limits_{\e\to 0}\Wp\big(
\frac{(\ft^x_\e+r\cdot w_\e)^{\ell-1}}{\e e^{\fq (\ft^x_\e+r\cdot w_\e)}} 
\sum_{k=1}^{m} e^{i  (\ft^x_\e+r\cdot w_\e)\theta_k} v_k
 +  \oO_\infty,  \oO_\infty\big)
 =\Wp\big(\fq^{1-\ell}e^{-\fq w r}u+\oO_\infty,\oO_\infty\big),
\end{equation}
where $u$ is any element in $\omega(x)$. 
It follows the proof of ii) $\Longrightarrow$ i). 
Statement ii) yields the existence of the limit 
(\ref{eq:distancia}).
By Proposition \ref{prop:replacement} we have that the limit (\ref{eq:equivalentedistancia}) also exists.
Pick an arbitrary element $u\in \omega(x)$. Then there exists a sequence of $(\e_n)_{n\in \NN}$, $\e_n\to 0$, $n\to \infty$ such that 
\[
\lim_{n\to \infty}
\sum_{k=1}^{m} e^{i  (\ft^x_{{\e_{n}}}+r\cdot w_{{\e_{n}}})\theta_k} v_k
=u.
\]
By (\ref{eq:newscaling}) we have 
\begin{align}\label{eq:limiteene}
\lim_{n\to \infty}
\frac{(\ft^x_{\e_{n}}+r\cdot w_{\e_{n}})^{\ell-1}}{{\e_{n}} e^{\fq (\ft^x_{{\e_{n}}}+r\cdot w_{\e_{n}})}} 
\sum_{k=1}^{m} e^{i  (\ft^x_{{\e_{n}}}+r\cdot w_{{\e_{n}}})\theta_k} v_k
=\fq^{1-\ell}e^{-\fq w r} u.
\end{align}
On the other hand, the existence of the limit
(\ref{eq:equivalentedistancia}) implies for any $({\ti {\e}_n})_{n\in \NN}$
 such that $\ti {\e}_n\to 0$, as $n\to \infty$ the limit
\[
\lim\limits_{n\to \infty}\Wp\big(
\frac{(\ft^x_{\ti {\e}_n}+r\cdot w_{\ti {\e}_n})^{\ell-1}}{{\ti {\e}_n} e^{\fq (\ft^x_{\ti {\e}_n}+r\cdot w_{\ti {\e}_n})}} 
\sum_{k=1}^{m} e^{i  (\ft^x_{\ti {\e}_n}+r\cdot w_{\ti {\e}_n})\theta_k} v_k
 +  \oO_\infty,  \oO_\infty\big)
 =\Wp\big(\fq^{1-\ell}e^{-\fq w r}\ti u+\oO_\infty,\oO_\infty\big)
\]
for some $\ti u\in \omega(x)$.
In particular for ${\ti {\e}_n}=\e_n$ we have
\begin{align*}
\Wp\big(\fq^{1-\ell}e^{-\fq w r}\ti u+\oO_\infty,\oO_\infty\big)&=
\lim\limits_{n\to \infty}\Wp\big(
\frac{(\ft^x_{\e_n}+r\cdot w_{\e_n})^{\ell-1}}{{\e_n} e^{\fq (\ft^x_{\e_n}+r\cdot w_{\e_n})}} 
\sum_{k=1}^{m} e^{i  (\ft^x_{\e_n}+r\cdot w_{\e_n})\theta_k} v_k
 +  \oO_\infty,  \oO_\infty\big)
\\
& =\Wp\big(\fq^{1-\ell}e^{-\fq w r} u+\oO_\infty,\oO_\infty\big),
\end{align*}
where the last equality follows by (\ref{eq:limiteene}) and continuity.
As a consequence, the function
\[
\omega(x)\ni u\mapsto \Wp\big(\fq^{1-\ell}e^{-\fq w r} u+\oO_\infty,\oO_\infty\big)=\Wp\big(\fq^{1-\ell}e^{-\fq w r}\ti u+\oO_\infty,\oO_\infty\big)
\]
is constant.
Since $r\in \RR$ is arbitrary, we obtain the statement i) for arbitrary $\lambda>0$.
\end{proof}

\begin{proof}[Proof of Theorem \ref{th:profile}]
By Theorem \ref{th:profileabstract} we have that profile thermalization holds if and only if for any $\lambda\gqq 0$
 the shift map $\sS_{\lambda}^\infty:\omega(x)\rightarrow [0,\infty)$ defined by $\sS_{\lambda}^\infty(u)=\Wp(\lambda u+\oO_\infty,\oO_\infty)$ is constant.
Since $1\lqq p'\lqq p$, property d) of Lemma \ref{lem:invariance} yields
\[
\Wp(v+\oO_\infty,\oO_\infty)=|v|
\] 
for any $v\in \RR^d$.
Therefore, $\sS_{\lambda}^\infty(u)=|\lambda u|$ for any $u\in \omega(x)$.
Consequently, the map $\sS^\infty_\lambda$ is constant if and only if $\omega(x)$ is a contained in a sphere.
In this case,  (\ref{eq:profileformula66}) yields the cutoff thermalization profile
\[
\pP_x(r)=\fq^{1-\ell}e^{-\fq w r}|u|,
\] 
where $u$ is any
element in $ \omega(x)$.
\end{proof}

\section{Proof of Theorem \ref{th:linearwhite} (window cutoff thermalization)}\label{ap:C}
\begin{proof}[Proof of Theorem \ref{th:linearwhite}]
Let $0<p'\lqq p$.
By Proposition \ref{prop:replacement}  we have
\[
\limsup_{\e\to 0}\frac{\Wp(X_{\ft^x_\e+r\cdot w_\e}^\e(x), \mu^\e)}{\e^{\min\{p',1\}}}=\limsup_{\e\to 0}\Wp\big(
\frac{(\ft^x_\e+r\cdot w_\e)^{\ell-1}}{\e e^{\fq (\ft^x_\e+r\cdot w_\e)}} \sum_{k=1}^{m} e^{i  (\ft^x_\e+r\cdot w_\e)\theta_k} v_k
 +  \oO_\infty,  \oO_\infty\big)
\]
and 
\[
\liminf_{\e\to 0}\frac{\Wp(X_{\ft^x_\e+r\cdot w_\e}^\e(x), \mu^\e)}{\e^{\min\{p',1\}}}=\liminf_{\e\to 0}\Wp\big(
\frac{(\ft^x_\e+r\cdot w_\e)^{\ell-1}}{\e e^{\fq (\ft^x_\e+r\cdot w_\e)}} \sum_{k=1}^{m} e^{i  (\ft^x_\e+r\cdot w_\e)\theta_k} v_k
 +  \oO_\infty,  \oO_\infty\big).
\]
In particular, Property d) in Lemma \ref{lem:invariance} yields 
\begin{equation}\label{eq:cotasuperiornferior}
|u_\e|^{\min\{p',1\}}-2\EE[\oO_\infty|^{p'}]\lqq 
\Wp(u_\e+\oO_\infty,\oO_\infty)\lqq |u_\e|^{\min\{p',1\}},
\end{equation}
where
\[
u_\e=\frac{(\ft^x_\e+r\cdot w_\e)^{\ell-1}}{\e e^{\fq (\ft^x_\e+r\cdot w_\e)}} \sum_{k=1}^{m} e^{i  (\ft^x_\e+r\cdot w_\e)\theta_k} v_k.
\]
We start with the upper limit.
Therefore by (\ref{eq:newscaling}) and (\ref{eq:cotasuperiornferior}) we have
\begin{align*}
\limsup_{\e\to 0}\frac{\Wp(X_{\ft^x_\e+r\cdot w_\e}^\e(x), \mu^\e)}{\e^{\min\{p',1\}}}&\lqq 
\limsup_{\e\to 0}
|u_\e|^{\min\{p',1\}}\\
&
\lqq \fq^{1-\ell}e^{-\fq w r}\limsup_{\e\to 0}|\sum_{k=1}^{m} e^{i  (\ft^x_\e+r\cdot w_\e)\theta_k} v_k|\lqq e^{-\fq w r}\fq^{1-\ell}
\sum_{k=1}^{m} |v_k|.
\end{align*}
Hence, 
\[
\lim_{r\to \infty}
\limsup_{\e\to 0}\frac{\Wp(X_{\ft^x_\e+r\cdot w_\e}^\e(x), \mu^\e)}{\e^{\min\{p',1\}}}=0.
\]
We continue with the lower limit.
By (\ref{eq:newscaling}) and (\ref{eq:cotasuperiornferior}) we have
\begin{align*}
\liminf_{\e\to 0}\frac{\Wp(X_{\ft^x_\e+r\cdot w_\e}^\e(x), \mu^\e)}{\e^{\min\{p',1\}}}&\gqq 
\liminf_{\e\to 0}
|u_\e|^{\min\{p',1\}}-2\EE[|\oO_\infty|^{p'}]\\
&
=  \fq^{1-\ell}e^{-\fq w r}\liminf_{\e\to 0}|\sum_{k=1}^{m} e^{i  (\ft^x_\e+r\cdot w_\e)\theta_k} v_k|-2\EE[|\oO_\infty|^{p'}]\\
& \gqq
  \fq^{1-\ell}e^{-\fq w r}\liminf_{t\to \infty}|\sum_{k=1}^{m} e^{i \theta_k t} v_k|-2\EE[|\oO_\infty|^{p'}]
\end{align*}
By (\ref{eq:belowabove}) in Lemma \ref{jara} we have $\liminf_{t\to \infty}|\sum_{k=1}^{m} e^{i \theta_k t} v_k|>0$.
Hence, 
\[
\lim_{r\to -\infty}
\liminf_{\e\to 0}\frac{\Wp(X_{\ft^x_\e+r\cdot w_\e}^\e(x), \mu^\e)}{\e^{\min\{p',1\}}}=\infty.
\]
\end{proof}

\section{Proof of Corollary \ref{cor:momentscutoffOUP} and Corollary~\ref{cor: delta-window} (Moment thermalization cutoff)}\label{ap:D}

The following lemma is used in the proof of Corollary \ref{cor:momentscutoffOUP} and in Subsection \ref{ss:ext1}.

\begin{lem}\label{lem:precisehartmandecomp}
For any $d\times d$ matrix $A$ it follows
\begin{align*}
&\left|
\Big|A \frac{e^{-\qQ (\ft_\e^x+r\cdot w_\e)}x}{\e}\Big|-\Big|A 
\frac{(\ft^x_\e+r\cdot w_\e)^{\ell-1}}{\e e^{\fq (\ft^x_\e+r\cdot w_\e)}} \sum_{k=1}^{m} e^{i  (\ft^x_\e+r\cdot w_\e)\theta_k} v_k\Big|\right|\\
&\lqq |A|\frac{(\ft^x_\e+r\cdot w_\e)^{\ell-1}}{\e e^{\fq (\ft^x_\e+r\cdot w_\e)}}
\left| \Big(e^{\fq (\ft^x_\e+r\cdot w_\e)}\frac{e^{-\qQ (\ft_\e^x+r\cdot w_\e)}x}{(\ft^x_\e+r\cdot w_\e)^{\ell-1}}-
 \sum_{k=1}^{m} e^{i  (\ft^x_\e+r\cdot w_\e)\theta_k} v_k\Big)\right|.
\end{align*}
\end{lem}

\begin{proof} The inverse triangle inequality and the submultiplicativity of matrix norm imply 
\begin{align*}
\left|
\Big|A e^{-\qQ t}x\Big|-\Big|A
\frac{t^{\ell-1}}{ e^{\fq t}} \sum_{k=1}^{m} e^{i t\theta_k} v_k\Big|\right|& \lqq \Big|A e^{-\qQ t}x-A
\frac{t^{\ell-1}}{ e^{\fq t}} \sum_{k=1}^{m} e^{i t\theta_k} v_k\Big| \\
&\lqq |A|\Big| e^{-\qQ t}x-
\frac{t^{\ell-1}}{ e^{\fq t}} \sum_{k=1}^{m} e^{i t\theta_k} v_k\Big|.
\end{align*}
In particular, we have 
\begin{align*}
&\left|
\Big|A \frac{e^{-\qQ (\ft_\e^x+r\cdot w_\e)}x}{\e}\Big|-\Big|A
\frac{(\ft^x_\e+r\cdot w_\e)^{\ell-1}}{\e e^{\fq (\ft^x_\e+r\cdot w_\e)}} \sum_{k=1}^{m} e^{i (\ft^x_\e+r\cdot w_\e)\theta_k} v_k\Big|\right|\\
&\lqq |A|\frac{(\ft^x_\e+r\cdot w_\e)^{\ell-1}}{\e e^{\fq (\ft^x_\e+r\cdot w_\e)}}
\left| \Big(e^{\fq (\ft^x_\e+r\cdot w_\e)}\frac{e^{-\qQ (\ft_\e^x+r\cdot w_\e)}x}{(\ft^x_\e+r\cdot w_\e)^{\ell-1}}-
\sum_{k=1}^{m} e^{i (\ft^x_\e+r\cdot w_\e)\theta_k} v_k\Big)\right|.
\end{align*}

\end{proof}

We show the moment thermalization of Corollary \ref{cor:momentscutoffOUP}. 

\begin{proof}[Proof of Corollary \ref{cor:momentscutoffOUP}]
For convenience we start with the case $p'\gqq 1$ and we write $\|X\|_{p'}=(\EE[|X|^{p'}])^{1/p'}$.
We start with the proof of 
\begin{align*}
\lim_{r\to \infty}
\limsup_{\e\to 0}\frac{1}{\e^{p'}}\EE[|X^{\e}_{\ft_\e^x+r\cdot w_\e}(x)|^{p'}]=\EE[|\oO_{\infty}|^{p'}].
\end{align*}
Note that
\begin{align*}
\frac{1}{\e}\|X^\e_t(x)\|_{p'}&=\frac{1}{\e}\|e^{-\qQ t}x+\e \oO_t\|_{p'}\lqq 
\frac{1}{\e}|e^{-\qQ t}x|+\|\oO_t\|_{p'}.
\end{align*}
Hence
\begin{align*}
\frac{1}{\e^{p'}}\|X^\e_t(x)\|^{p'}_{p'}\lqq \Big(
\frac{1}{\e}|e^{-\qQ t}x|+\|\oO_t\|_{p'}\Big)^{p'}.
\end{align*}
Since $\W(\oO_t,\oO_\infty)\to 0$ as $t\to \infty$, we have $\|\oO_t\|_{p'}\to \|\oO_\infty\|_{p'} <\infty$ as $t\to \infty$.
The preceding inequality yields
\begin{align*}
\limsup_{\e\to 0}\frac{1}{\e^{p'}}\|X^\e_{\ft_\e^x+r\cdot w_\e}(x)\|^{p'}_{p'} &\lqq 
\limsup_{\e\to 0}
\Big(
\frac{1}{\e}|e^{-\qQ (\ft_\e^x+r\cdot w_\e)}x|+\|\oO_{\ft_\e^x+r\cdot w_\e}\|_{p'}\Big)^{p'}\\
&=
\Big(
\limsup_{\e\to 0}\big(
\frac{1}{\e}|e^{-\qQ (\ft_\e^x+r\cdot w_\e)}x|+\|\oO_{\ft_\e^x+r\cdot w_\e}\|_{p'}\big)\Big)^{p'}\\
&=
\Big(
\limsup_{\e\to 0}\big(
\frac{1}{\e}|e^{-\qQ (\ft_\e^x+r\cdot w_\e)}x|\big)+\|\oO_{\infty}\|_{p'}\Big)^{p'}.
\end{align*}
By Lemma \ref{lem:precisehartmandecomp}
and the  continuity of $y\mapsto |y|^{p'}$, sending $r\to \infty$ we obtain
\begin{align}\label{ec:limitesuperior2}
\lim_{r\to \infty}\limsup_{\e\to 0}\frac{1}{\e^{p'}}\|X^\e_{\ft_\e^x+r\cdot w_\e}(x)\|^{p'}_{p'} &\lqq 
\Big(\lim_{r\to \infty}
\limsup_{\e\to 0}\big(
\frac{1}{\e}|e^{-\qQ (\ft_\e^x+r\cdot w_\e)}x|\big)+\|\oO_{\infty}\|_{p'}\Big)^{p'}=
\|\oO_{\infty}\|_{p'}^{p'}.
\end{align}
We continue with the proof of
\[
\lim_{r\to \infty}\liminf_{\e\to 0}\frac{1}{\e^{p'}}\|X^\e_{\ft_\e^x+r\cdot w_\e}(x)\|^{p'}_{p'}\gqq \|\oO_{\infty}\|_{p'}^{p'}.
\]
Note that
\begin{align*}
\|\oO_t\|_{p'}=\frac{1}{\e}\|\e\oO_t\|_{p'}\lqq \frac{1}{\e}\|X^{\e}_t(x)\|_{p'}+ 
\frac{1}{\e}|e^{-\qQ t}x|.
\end{align*}
Hence
\[
\liminf_{\e \to 0}\|\oO_{\ft^x_\e+r\cdot w_\e}\|_{p'}+\liminf_{\e \to 0}\big(-\frac{1}{\e}|e^{-\qQ (\ft^x_\e+r\cdot w_\e)}x|\big)\lqq 
\liminf_{\e \to 0}\big(
\frac{1}{\e}\|X^{\e}_{\ft^x_\e+r\cdot w_\e}(x)\|_{p'}\big).
\]
Since for a general family $(a_\e)_{\e>0}$ we have $\liminf_{\e\to 0} (-a_\e)=-\limsup_{\e \to 0} a_\e$ it follows
\[
\|\oO_{\infty}\|_{p'}-\limsup_{\e \to 0}\frac{1}{\e}|e^{-\qQ (\ft^x_\e+r\cdot w_\e)}x|
\lqq 
\liminf_{\e \to 0}
\frac{1}{\e}\|X^{\e}_{\ft^x_\e+r\cdot w_\e}(x)\|_{p'}.
\]
By continuity we have 
\[
\left(\|\oO_{\infty}\|_{p'}-\limsup_{\e \to 0}\frac{1}{\e}|e^{-\qQ (\ft^x_\e+r\cdot w_\e)}x|\right)^{p'}\lqq 
\liminf_{\e \to 0}\big(
\frac{1}{\e^{p'}}\|X^{\e}_{\ft_\e^x+r\cdot w_\e}(x)\|^{p'}_{p'}\big).
\]
By Lemma \ref{lem:precisehartmandecomp}, 
and sending $r\to \infty$ we have
\begin{align}\label{ec:limsuperior}
\|\oO_{\infty}\|^{p'}_{p'}\lqq 
\lim_{r\to \infty}\liminf_{\e \to 0}
\frac{1}{\e^{p'}}\|X^{\e}_{\ft_\e^x+r\cdot w_\e}(x)\|^{p'}_{p'}.
\end{align}
Combining (\ref{ec:limitesuperior2}) and (\ref{ec:limsuperior}) we obtain
\begin{align*}
\lim_{r\to \infty}
\liminf_{\e\to 0}\frac{1}{\e^{p'}}\|X^{\e}_{\ft_\e^x+r\cdot w_\e}(x)\|^{p'}_{p'}=
\lim_{r\to \infty}
\limsup_{\e\to 0}\frac{1}{\e^{p'}}\|X^{\e}_{\ft_\e^x+r\cdot w_\e}(x)\|^{p'}_{p'}=\|\oO_{\infty}\|^{p'}_{p'}.
\end{align*}
In the sequel, we show
\begin{align*}
\lim_{r\to -\infty}
\liminf_{\e\to 0}\frac{1}{\e^{p'}}\|X^{\e}_{\ft_\e^x+r\cdot w_\e}(x)\|^{p'}_{p'}=\infty.
\end{align*}
Note that
\begin{align*}
\frac{1}{\e}|e^{-\qQ t}x|=
\frac{1}{\e}\|X^{\e}_t(x)-\e \oO_t\|_{p'}\lqq \frac{1}{\e}\|X^{\e}_t(x)\|_{p'}+\| \oO_t\|_{p'}
\end{align*}
and hence
\begin{align*}
\frac{1}{\e^{p'}}|e^{-\qQ t}x|^{p'}\lqq
\Big( \frac{1}{\e}\|X^{\e}_t(x)\|_{p'}+\| \oO_t\|_{p'}\Big)^{p'}.
\end{align*} 
\begin{align*}
\liminf_{\e\to 0}
\frac{1}{\e^{p'}}|e^{-\qQ (\ft_\e^x+r\cdot w_\e)}x|^{p'}&\lqq
\liminf_{\e\to 0}
\Big( \frac{1}{\e}\|X^{\e}_{\ft_\e^x+r\cdot w_\e}(x)\|_{p'}+\| \oO_{\ft_\e^x+r\cdot w_\e}\|_{p'}\Big)^{p'}\\
&=
\Big(\liminf_{\e\to 0}\big(\frac{1}{\e}\|X^{\e}_{\ft^x_\e+r\cdot w_\e}(x)\|_{p'}+\| \oO_{\ft_\e^x+r\cdot w_\e}\|_{p'}\big)\Big)^{p'}\\
&=
\Big(\liminf_{\e\to 0}\big(\frac{1}{\e}\|X^{\e}_{\ft^x_\e+r\cdot w_\e}(x)\|_{p'}\big)+\| \oO_{\infty}\|_{p'}\Big)^{p'}.
\end{align*}
Sending $r\to -\infty$,
Lemma \ref{lem:precisehartmandecomp} yields
\begin{align*}
\infty=\lim_{r\to -\infty}\liminf_{\e\to 0}
\frac{1}{\e^{p'}}|e^{-\qQ (\ft_\e^x+r\cdot w^\e)}x|^{p'}&\lqq
\Big(\lim_{r\to -\infty}\liminf_{\e\to 0}\big(\frac{1}{\e}\|X^{\e}_{\ft_\e^x+r\cdot w_\e}(x)\|_{p'}\big)+\| \oO_{\infty}\|_{p'}\Big)^{p'}.
\end{align*}
Hence 
\[
\lim_{r\to -\infty}\liminf_{\e\to 0}\frac{1}{\e^{p'}}\|X^{\e}_{\ft_\e^x+r\cdot w_\e}(x)\|^{p'}_{p'}=\infty.
\]
We continue with the case $p'\in (0,1)$ and we write $\|X\|_{p'}=\EE[|X|^{p'}]$.
We start with the proof of 
\begin{align*}
\lim_{r\to \infty}
\limsup_{\e\to 0}\frac{1}{\e^{p'}}\EE[|X^{\e}_{\ft_\e^x+r\cdot w_\e}(x)|^{p'}]=\EE[|\oO_{\infty}|^{p'}].
\end{align*}
Note that
\begin{align*}
\frac{1}{\e^{p'}}\|X^\e_t(x)\|_{p'}&=\frac{1}{\e^{p'}}\|e^{-\qQ t}x+\e \oO_t\|_{p'}\lqq 
\frac{1}{\e^{p'}}|e^{-\qQ t}x|^{p'}+\|\oO_t\|_{p'}.
\end{align*}
Hence
\begin{align*}
\frac{1}{\e^{p'}}\|X^\e_t(x)\|_{p'}\lqq
\frac{1}{\e^{p'}}|e^{-\qQ t}x|+\|\oO_t\|_{p'}.
\end{align*}
Since $\W(\oO_t,\oO_\infty)\to 0$ as $t\to \infty$, we have $\|\oO_t\|_{p'}\to \|\oO_\infty\|_{p'} <\infty$ as $t\to \infty$.
The preceding inequality yields
\begin{align*}
\limsup_{\e\to 0}\frac{1}{\e^{p'}}\|X^\e_{\ft_\e^x+r\cdot w_\e}(x)\|_{p'} &\lqq 
\limsup_{\e\to 0}
\frac{1}{\e^{p'}}|e^{-\qQ (\ft_\e^x+r\cdot w_\e)}x|^{p'}+\|\oO_{\ft_\e^x+r\cdot w_\e}\|_{p'}\\
&=
\limsup_{\e\to 0}\big(
\frac{1}{\e}|e^{-\qQ (\ft_\e^x+r\cdot w_\e)}x|\big)^{p'}+\|\oO_{\infty}\|_{p'}\\
&=
\Big(
\limsup_{\e\to 0}
\frac{1}{\e}|e^{-\qQ (\ft_\e^x+r\cdot w_\e)}x|\Big)^{p'}+\|\oO_{\infty}\|_{p'}.
\end{align*}
By Lemma \ref{lem:precisehartmandecomp}
and the  continuity of $y\mapsto |y|^{p'}$, sending $r\to \infty$ we obtain
\begin{align}\label{ec:limitesuperior222}
\lim_{r\to \infty}\limsup_{\e\to 0}\frac{1}{\e^{p'}}\|X^\e_{\ft_\e^x+r\cdot w_\e}(x)\|_{p'} &\lqq 
\Big(\lim_{r\to \infty}
\limsup_{\e\to 0}
\frac{1}{\e}|e^{-\qQ (\ft_\e^x+r\cdot w_\e)}x|\Big)^{p'}+\|\oO_{\infty}\|_{p'}=
\|\oO_{\infty}\|_{p'}^{p'}.
\end{align}
We continue with the proof of
\[
\lim_{r\to \infty}\liminf_{\e\to 0}\frac{1}{\e^{p'}}\|X^\e_{\ft_\e^x+r\cdot w_\e}(x)\|_{p'}\gqq \|\oO_{\infty}\|_{p'}.
\]
Note that
\begin{align*}
\|\oO_t\|_{p'}=\frac{1}{\e^{p'}}\|\e\oO_t\|_{p'}\lqq \frac{1}{\e^{p'}}\|X^{\e}_t(x)\|_{p'}+ 
\frac{1}{\e^{p'}}|e^{-\qQ t}x|^{p'}.
\end{align*}
Hence
\[
\liminf_{\e \to 0}\|\oO_{\ft_\e^x+r\cdot w_\e}\|_{p'}+\liminf_{\e \to 0}\big(-\frac{1}{\e^{p'}}|e^{-\qQ (\ft_\e^x+r\cdot w_\e)}x|^{p'}\big)\lqq 
\liminf_{\e \to 0}\big(
\frac{1}{\e^{p'}}\|X^{\e}_{\ft_\e^x+r\cdot w_\e}(x)\|_{p'}\big).
\]
Since for a general sequence $(a_\e)_{\e>0}$ we have $\liminf_{\e\to 0} (-a_\e)=-\limsup_{\e \to 0} a_\e$ it follows
\[
\|\oO_{\infty}\|_{p'}-\limsup_{\e \to 0}\frac{1}{\e^{p'}}|e^{-\qQ (\ft_\e^x+r\cdot w_\e)}x|^{p'}
\lqq 
\liminf_{\e \to 0}
\frac{1}{\e^{p'}}\|X^{\e}_{\ft_\e^x+r\cdot w_\e}(x)\|_{p'}.
\]
By Lemma \ref{lem:precisehartmandecomp},
 and sending $r\to \infty$ we obtain
\begin{align}\label{ec:limsuperiornew}
\|\oO_{\infty}\|_{p'}\lqq 
\lim_{r\to \infty}\liminf_{\e \to 0}
\frac{1}{\e^{p'}}\|X^{\e}_{\ft_\e^x+r\cdot w_\e}(x)\|_{p'}.
\end{align}
Combining (\ref{ec:limitesuperior222}) and (\ref{ec:limsuperiornew}) we obtain
\begin{align*}
\lim_{r\to \infty}
\liminf_{\e\to 0}\frac{1}{\e^{p'}}\|X^{\e}_{\ft_\e^x+r\cdot w_\e}(x)\|_{p'}=
\lim_{r\to \infty}
\limsup_{\e\to 0}\frac{1}{\e^{p'}}\|X^{\e}_{\ft_\e^x+r\cdot w_\e}(x)\|_{p'}=\|\oO_{\infty}\|_{p'}.
\end{align*}
In the sequel, we show
\begin{align*}
\lim_{r\to -\infty}
\liminf_{\e\to 0}\frac{1}{\e^{p'}}\|X^{\e}_{\ft_\e^x+r\cdot w_\e}(x)\|_{p'}=\infty.
\end{align*}
Note that
\begin{align*}
\frac{1}{\e^{p'}}|e^{-\qQ t}x|^{p'}=
\frac{1}{\e^{p'}}\|X^{\e}_t(x)-\e \oO_t\|_{p'}\lqq \frac{1}{\e^{p'}}\|X^{\e}_t(x)\|_{p'}+\| \oO_t\|_{p'}
\end{align*}
and hence 
\begin{align*}
\liminf_{\e\to 0}
\frac{1}{\e^{p'}}|e^{-\qQ (\ft_\e^x+r\cdot w_\e)}x|^{p'}&\lqq
\liminf_{\e\to 0}\Big(
\frac{1}{\e^{p'}}\|X^{\e}_{\ft_\e^x+r\cdot w_\e}(x)\|_{p'}+\| \oO_{\ft_\e^x+r\cdot w_\e}\|_{p'}\Big)\\
&=
\liminf_{\e\to 0}\big(\frac{1}{\e^{p'}}\|X^{\e}_{\ft_\e^x+r\cdot w_\e}(x)\|_{p'}\big)+\| \oO_{\infty}\|_{p'}.
\end{align*}
Sending $r\to -\infty$,
Lemma \ref{lem:precisehartmandecomp} yields
\begin{align*}
\infty=\lim_{r\to -\infty}\liminf_{\e\to 0}
\frac{1}{\e^{p'}}|e^{-\qQ (\ft_\e^x+r\cdot w_\e)}x|^{p'}&\lqq
\lim_{r\to -\infty}\liminf_{\e\to 0}\big(\frac{1}{\e^{p'}}\|X^{\e}_{\ft_\e^x+r\cdot w_\e}(x)\|_{p'}\big)+\| \oO_{\infty}\|_{p'}.
\end{align*}
Hence 
\begin{equation*}
\lim_{r\to -\infty}\liminf_{\e\to 0}\frac{1}{\e^{p'}}\|X^{\e}_{\ft_\e^x+r\cdot w_\e}(x)\|_{p'}=\infty.
\end{equation*}
\end{proof}

In the sequel we prove Corollary~\ref{cor: delta-window}.

\begin{proof}[Proof of Corollary~\ref{cor: delta-window}]
Let $\ft^x_\e$ be the time scale given in Theorem \ref{th:profile}.
By (\ref{eq: continuidad666}) we have for any $\delta>0$ 
\begin{equation}\label{eq:arribasup}
\limsup\limits_{\e\to 0}\frac{\Wp(X^\e_{\delta \ft^x_\e}(x), X^\e_\infty)}{\e^{\min\{p',1\}}}=
\limsup_{\e\to 0}\Wp\big(\frac{e^{-\qQ (\delta\ft^x_\e)} x}{\e} + \oO_\infty, \oO_\infty\big)
\end{equation}
and
\begin{equation}\label{eq:abajoinf}
\liminf\limits_{\e\to 0}\frac{\Wp(X^\e_{\delta\ft^x_\e}(x), X^\e_\infty)}{\e^{\min\{p',1\}}}=
\liminf_{\e\to 0}\Wp\big(\frac{e^{-\qQ (\delta\ft^x_\e)} x}{\e} + \oO_\infty, \oO_\infty\big).
\end{equation}
By \eqref {eq:belowabove} given in Lemma \ref{jara} we have
\begin{equation}\label{eq:deltaabovebelow}
0<\liminf_{\e \to 0} \left |\frac{e^{\fq (\delta\ft^x_\e)}}{(\delta\ft^x_\e)^{\ell-1}} e^{- \qQ (\delta\ft^x_\e)}x \right |\lqq \limsup_{\e \to 0} \left |\frac{e^{\fq (\delta\ft^x_\e)}}{(\delta\ft^x_\e)^{\ell-1}} e^{- \qQ (\delta\ft^x_\e)}x \right|<\infty
\end{equation}
for any $\delta>0$. A straightforward calculation shows
\begin{equation}\label{eq:deltalimit}
\lim\limits_{\e \to 0}
\frac{(\delta\ft^x_\e)^{\ell-1} e^{-\fq (\delta\ft^x_\e)}}{\e}=
\left.\begin{cases}
\infty & \textrm{ for } \delta\in (0,1),\\
0 & \textrm{ for } \delta>1.
\end{cases}\right\}
\end{equation}
Since
\[
\frac{e^{- \qQ (\delta\ft^x_\e)}x}{\e}=
\frac{(\delta\ft^x_\e)^{\ell-1} e^{-\fq (\delta\ft^x_\e)}}{\e}\frac{e^{\fq (\delta\ft^x_\e)}}{(\delta\ft^x_\e)^{\ell-1}}e^{- \qQ (\delta\ft^x_\e)}x,
\]
the relations \eqref{eq:arribasup}, \eqref{eq:abajoinf}, \eqref{eq:deltaabovebelow} and the limit \eqref{eq:deltalimit} imply  with the help of the continuity of $\Wp$ the desired result: 
\[
\lim\limits_{\e\to 0}\frac{\Wp(X^\e_{\delta \ft^x_\e}(x), X^\e_\infty)}{\e^{\min\{p',1\}}}=
\left.\begin{cases}
\infty & \textrm{ for } \delta\in (0,1),\\
0 & \textrm{ for } \delta>1.
\end{cases}\right\}
\] 
 \end{proof}

\section{Proof of Theorem \ref{th:normalgrowth} (Normal growth characterization)}\label{Ap:normal}

We start with the following lemma, which shows that the $\CC$-linear independence of a family of pairs complex conjugate vectors implies the $\RR$-linear independence of the family of real and imaginary parts, in which the characterization in Theorem \ref{th:normalgrowth} is stated. The lemma is used in the representation (\ref{eq:wthetarepre1}). The proof is given for completeness. \\
\begin{lem}[Decomplexificacion]\label{lem:base}
For any $d\in \NN$ let $d\gqq 2n+r$ for some $n, r\in \NN_0$. 
Consider an arbitrary family of linearly independent vectors $(w_1,\ldots, w_r, v_1,\bar v_1, \ldots, v_n, \bar v_n)$ in $\CC^{d}$, where $\bar v_j$ denotes the complex conjugate of $v_j$.
Assume that $w_1,\ldots, w_r\in \RR^d$.
For short we write $v_j=\hat v_j+i\check v_j$, with $\hat v_j,\check v_j\in \RR^d$.
Then the family of vectors
\[
(w_1,\ldots, w_r,\hat v_1, \check v_1,\ldots,\hat v_n, \check v_n)
\]
is linear independent in $\RR^d$.
\end{lem}
\begin{proof}
Let $\alpha_1,\beta_1,\ldots,\alpha_n,\beta_n, \gamma_1,\ldots,\gamma_r\in \RR$ such that 
\[
\gamma_1w_1+\cdots +\gamma_r w_r+\alpha_1\hat v_1+\beta_1\check v_1+\cdots+\alpha_n \hat v_n+\beta_n \check v_n=0.
\]
Let $\hat g_j=\gamma_j$ and $\check g_j=0$, and for $j\in \{1,\ldots,n\}$
$\hat c_j=\hat d_j=\al_j/2$ and $\check c_j=-\check d_j=-\beta_j/2$. Then for $c_j=\hat c_j+i\check c_j$ 
and $d_j=\hat d_j+i\check d_j$ 
we have
\[
g_1w_1+\cdots+ g_rw_r+c_1  v_1+d_1 \bar v_1
+\ldots+c_n v_n+d_n \bar v_n=0.
\]
By Hypothesis $w_1,\ldots,w_r,v_1,\bar v_1, \ldots, v_n, \bar v_n$ is a family of linearly independent vectors in $\CC^{2d}$ and hence 
$g_1=\cdots=g_r=c_1=d_1=\cdots=c_n=d_n=0$ which is equivalent to $\gamma_1=\cdots=\gamma_r=\al_1=\beta_1=\cdots=\alpha_n=\beta_n=0$ as desired.
\end{proof}
In the sequel, we characterize in Lemmas \ref{lem:ida} and \ref{lem:vuelta} under the non-resonance hypothesis (\ref{eq:rationallyindependent1}) given in Remark~\ref{rem:rationallyindependent1} when the function
$\omega(x)\ni u\mapsto |u|$ is constant, which is the statement of Theorem 3.1 item i), 
and compute the norm of (\ref{eq:wthetarepre1}). 
Using representation (\ref{eq:wthetarepre1}) of $\om(x)$ defined by  (\ref{eq:omegalimit}) we prepare the statement of Lemmas \ref{lem:ida} and \ref{lem:vuelta}. Lemma \ref{lem:ida} yields the necessity, while Lemma \ref{lem:vuelta} states the sufficiency of the normal growth condition (\ref{e:non-normal growth}) of Theorem~3.3. 
The  Pythagoras theorem yields
\begin{align*}
&|\sum_{k=1}^{m}e^{i\theta_k t}v_k|^2\\
&=
|{v_1}|^2+4\<{v_1},\sum_{k=1}^{n}
\Big(\cos(\theta_{2k} t)\hat v_{2k}-\sin(\theta_{2k} t)\check v_{2k}\Big)\>
+4|\sum_{k=1}^{n}
\Big(\cos(\theta_{2k} t)\hat v_{2k}-\sin(\theta_{2k} t)\check v_{2k}\Big)|^2\\
&=
|{v_1}|^2+4
\sum_{k=1}^{n}
\big(\cos(\theta_{2k} t)
\<{v_1},\hat v_{2k}\>
-
\sin(\theta_{2k} t)
\<{v_1},\check v_{2k}\>\big)
+4|\sum_{k=1}^{n}
\Big(\cos(\theta_{2k} t)\hat v_{2k}-\sin(\theta_{2k} t)\check v_{2k}\Big)|^2.
\end{align*}
We continue with the last term on the right-hand side of the preceding 
expression omitting the prefactor 
\begin{align*}
&|\sum_{k=1}^{n}
\Big(\cos(\theta_{2k} t)\hat v_{2k}-\sin(\theta_{2k} t)\check v_{2k}\Big)|^2\\
&\qquad=
\sum_{k=1}^{n}
|\cos(\theta_{2k} t)\hat v_{2k}-\sin(\theta_{2k} t)\check v_{2k}|^2\\
&\qquad\qquad+
\sum_{k\not =k'}
\<\cos(\theta_{2k} t)\hat v_{2k}-\sin(\theta_{2k} t)\check v_{2k},
\cos(\theta_{2k'} t)\hat v_{2k'}-\sin(\theta_{2k'} t)\check v_{2k'}\>\\
&\qquad=\sum_{k=1}^{n}
\Big(\cos^2(\theta_{2k} t)
|\hat v_{2k}|^2+\sin^2(\theta_{2k} t)|\check v_{2k}|^2-2
\cos(\theta_{2k} t)\sin(\theta_{2k} t)
\<\hat v_{2k},
\check v_{2k}
\>\Big)\\
&\qquad\qquad+
\sum_{k\not =k'}
\cos(\theta_{2k} t)\cos(\theta_{2k'} t)
\<\hat v_{2k},\hat v_{2k'}\>
+
\sum_{k\not =k'}
\sin(\theta_{2k} t)\sin(\theta_{2k'} t)
\<\check v_{2k},\check v_{2k'}\>\\
&\qquad\qquad
-\sum_{k\not =k'}
\sin(\theta_{2k} t)\cos(\theta_{2k'} t)
\<\check v_{2k},\hat v_{2k'}\>
-
\sum_{k\not =k'}
\cos(\theta_{2k} t)\sin(\theta_{2k'} t)
\<\hat v_{2k},\check v_{2k'}\>.
\end{align*}
Combining the preceding equalities  we deduce
\begin{align}\label{eq:sumacuadrado}
&|\sum_{k=1}^{m}e^{i\theta_k t}v_k|^2 =
|{v_1}|^2\nonumber\\
&+4
\sum_{k=1}^{n}
\cos(\theta_{2k} t)
\<{v_1},\hat v_{2k}\>
-4\sum_{k=1}^{n}
\sin(\theta_{2k} t)
\<{v_1},\check v_{2k}\>
+4|\sum_{k=1}^{n}
\Big(\cos(\theta_{2k} t)\hat v_{2k}-\sin(\theta_{2k} t)\check v_{2k}\Big)|^2\nonumber\\
&=
|{v_1}|^2+4
\sum_{k=1}^{n}
\cos(\theta_{2k} t)
\<{v_1},\hat v_{2k}\>
-4\sum_{k=1}^{n}
\sin(\theta_{2k} t)
\<{v_1},\check v_{2k}\>\nonumber\\
&\qquad \qquad+
4\sum_{k=1}^{n}
\Big(\cos^2(\theta_{2k} t)
|\hat v_{2k}|^2+\sin^2(\theta_{2k} t)|\check v_{2k}|^2-2
\cos(\theta_{2k} t)\sin(\theta_{2k} t)
\<\hat v_{2k},
\check v_{2k}
\>\Big)\nonumber\\
&\qquad\qquad+
4\sum_{k\not =k'}
\cos(\theta_{2k} t)\cos(\theta_{2k'} t)
\<\hat v_{2k},\hat v_{2k'}\>
+
4\sum_{k\not =k'}
\sin(\theta_{2k} t)\sin(\theta_{2k'} t)
\<\check v_{2k},\check v_{2k'}\>\nonumber\\
&\qquad\qquad
-4\sum_{k\not =k'}
\sin(\theta_{2k} t)\cos(\theta_{2k'} t)
\<\check v_{2k},\hat v_{2k'}\>
-
4\sum_{k\not =k'}
\cos(\theta_{2k} t)\sin(\theta_{2k'} t)
\<\hat v_{2k},\check v_{2k'}\>.
\end{align}
After rearrangement of the sums, we obtain
\begin{align}\label{eq:sumsquare}
|\sum_{k=1}^{m}e^{i\theta_k t}v_k|^2&=
|{v_1}|^2+4
\sum_{k=1}^{n}
\cos(\theta_{2k} t)
\<{v_1},\hat v_{2k}\>
-4\sum_{k=1}^{n}
\sin(\theta_{2k} t)
\<{v_1},\check v_{2k}\>\\
&\qquad \qquad+
4|\sum_{k=1}^{n}
\cos(\theta_{2k} t)\hat v_{2k}|^2+4|\sum_{k=1}^{n}
\sin(\theta_{2k} t)\check v_{2k}|^2
\nonumber\\
&\qquad\qquad
-8\sum_{k,k'}
\sin(\theta_{2k} t)\cos(\theta_{2k'} t)
\<\check v_{2k},\hat v_{2k'}\>
.\nonumber
\end{align}

\begin{lem}\label{lem:ida}
Assume that the family $({v_1},\hat v_2,\check v_2\ldots,\hat v_{2n},\check v_{2n})$ is orthogonal and
$|\hat v_{2k}|=|\check v_{2k}|$ for all $k$.
Then the function
\[
t\mapsto 
\Big|\sum_{k=1}^{m}e^{i\theta_k t}v_k\Big|
\]
is constant and has the value $|{v_1}|^2+4\sum_{k=1}^{n}|\hat v_{2k}|^2$ and consequently
$\omega(x)\ni u\mapsto |u|$
is constant.
\end{lem}
\begin{proof}
The orthogonality hypothesis in relation (\ref{eq:sumacuadrado}) yields
\begin{align*}
|\sum_{k=1}^{m}e^{i\theta_k t}v_k|^2&=
|{v_1}|^2+
4\sum_{k=1}^{n}
\cos^2(\theta_{2k} t)|\hat v_{2k}|^2+4\sum_{k=1}^{n}
\sin^2(\theta_{2k} t)|\check v_{2k}|^2.
\end{align*}
Since $|\hat v_{2k}|=|\check v_{2k}|$ for all $k$, the Pythagoras identity yields the desired result.
\end{proof}

\begin{lem}\label{lem:vuelta}
If the function
\begin{equation}\label{eq:constantefunction}
\omega(x)\ni u\mapsto |u|
\end{equation}
is constant and the angles $\theta_2,\ldots,\theta_{2n}$ are rationally independent according to Remark \ref{rem:rationallyindependent1} then the family of  $\RR^d$-valued vectors  $({v_1},\hat v_2,\check v_2\ldots,\hat v_{2n},\check v_{2n})$ is orthogonal and satisfies
$|\hat v_{2k}|=|\check v_{2k}|$ for all $k=1,\ldots,n$.
\end{lem}
\begin{proof}
By (\ref{eq:sumsquare}) we have
\begin{align*}
|\sum_{k=1}^{m}e^{i\theta_k t}v_k|^2&=
|{v_1}|^2+4
\sum_{k=1}^{n}
\cos(\theta_{2k} t)
\<{v_1},\hat v_{2k}\>
-4\sum_{k=1}^{n}
\sin(\theta_{2k} t)
\<{v_1},\check v_{2k}\>\\
&+
4|\sum_{k=1}^{n}
\cos(\theta_{2k} t)\hat v_{2k}|^2+4|\sum_{k=1}^{n}
\sin(\theta_{2k} t)\check v_{2k}|^2
\nonumber\\
&-8\sum_{k,k'}
\sin(\theta_{2k} t)\cos(\theta_{2k'} t)
\<\check v_{2k},\hat v_{2k'}\>
.\nonumber
\end{align*}
Since the angles $\theta_2,\ldots,\theta_{2n}$ are rationally independent, 
Corollary 4.2.3 in \cite{VIANA} and
the assumption that the function (\ref{eq:constantefunction}) is constant  
implies that the following function $F$ is constant:
\begin{align}\label{eq:funcionF}
F(x)&=|{v_1}|^2+
4\sum_{k=1}^{n}x_k\<{v_1},\hat v_{2k}\>
-4\sum_{k=1}^{n}y_k\<{v_1},\check v_{2k}\>+
4|\sum_{k=1}^{n}
x_k\hat v_{2k}|^2+4|\sum_{k=1}^{n} y_k\check v_{2k}|^2
\\
&\qquad\qquad
-8\sum_{k,k'}
y_k x_{k'}
\<\check v_{2k},\hat v_{2k'}\>,\nonumber
\end{align}
where $x=(x_1,\ldots,x_n)\in [-1,1]^n$
 and $x^2_k+y^2_k=1$ for all $k$.
 We point out that for each $x_k$ there are two solutions $y_k=\pm \sqrt{1-x^2_k}$ for the equation 
 $x^2_k+y^2_k=1$ and 
all combinations of signs of $y_k$ are admitted. 

\textbf{Step 1: } We start with the proof of $\<{v_1},\hat v_{2k}\>=0$ for all $k$. Since the function $F$ is constant, comparing the choices $x_k=1$ and $y_k=0$ for all $k$ with $x_k=-1$ and $y_k=0$ for all $k$ yields
\[
|{v_1}|^2+4\sum_{k=1}^{n}\<{v_1},\hat v_{2k}\>+
\sum_{k=1}^{n}|\hat v_{2k}|^2=|{v_1}|^2-4\sum_{k=1}^{n}\<{v_1},\hat v_{2k}\>+
\sum_{k=1}^{n}|\hat v_{2k}|^2,
\] 
which implies
\begin{equation}\label{eq:sumaw}
\sum_{k=1}^{n}\<{v_1},\hat v_{2k}\>=0.
\end{equation}
Comparing the choice $x_1=-1$, $x_k=1$ for $k\gqq 2$,  and consequently $y_k=0$ for all $k$
 with $x_1=1$, $x_k=-1$ for $k\gqq 2$,  and consequently $y_k=0$ for all $k$ we have
\begin{align*}
&|{v_1}|^2-4\<{v_1},\hat v_2\>+4\sum_{k=2}^{n}\<{v_1},\hat v_{2k}\>+4|-\hat v_2+\sum_{k=2}^{n}\hat v_{2k}|^2\\
&\qquad=
|{v_1}|^2+4\<{v_1},\hat v_2\>-4\sum_{k=2}^{n}\<{v_1},\hat v_{2k}\>+4|\hat v_2-\sum_{k=2}^{n}\hat v_{2k}|^2\\
&\qquad=
|{v_1}|^2+4\<{v_1},\hat v_2\>-4\sum_{k=2}^{n}\<{v_1},\hat v_{2k}\>+4|-\hat v_2+\sum_{k=2}^{n}\hat v_{2k}|^2,
\end{align*}
which implies
\begin{equation}\label{eq:sumaw1}
\<{v_1},\hat v_2\>-\sum_{k=2}^{n}\<{v_1},\hat v_{2k}\>=0.
\end{equation}
Combining (\ref{eq:sumaw}) and (\ref{eq:sumaw1}) we obtain $\<{v_1},\hat v_2\>=0$. Analogously it is shown
$\<{v_1},\hat v_{2k}\>=0$ for all $k\gqq 2$.
Switching the role of $x_k$ and $y_k$ in the preceding reasoning also shows
$\<{v_1},\check v_{2k}\>=0$
for all $k$.

\textbf{Step 2: } By Step 1 formula  (\ref{eq:funcionF}) boils down to
\begin{align*}
F(x)=&|{v_1}|^2+4|\sum_{k=1}^{n}
x_k\hat v_{2k}|^2+4|\sum_{k=1}^{n} y_k\check v_{2k}|^2
-8\sum_{k,k'}x_{k}y_{k'}
C_{k,k'},
\end{align*}
where $C_{k,k'}=\<\check v_{2k},\hat v_{2k'}\>$.
The choice 
$x_1=\frac{\sqrt{2}}{2}$, $x_k=\pm1$ for all $k\gqq 2$ implies $y_1=\pm
\frac{\sqrt{2}}{2}$ and $y_k=0$ for $k\gqq 2$.
Hence 
\begin{align*}
&|{v_1}|^2+
4|\frac{\sqrt{2}}{2}\hat v_{2}+\sum_{k=2}^{n}(\pm)\hat v_{2k}|^2+2 |\check v_{2}|^2
-4 C_{1,1}-4\sqrt{2} \sum_{k\gqq 2} (\pm)C_{k,1}\\
&=|{v_1}|^2+
4|\frac{\sqrt{2}}{2}\hat v_{2}+\sum_{k=2}^{n}(\pm)\hat v_{2k}|^2+2 |\check v_{2}|^2
+4 C_{1,1}+4\sqrt{2} \sum_{k\gqq 2} (\pm)C_{k,1}\\
\end{align*}
for any sequence of signs $\pm$.
Consequently, 
\[
C_{1,1}+\sqrt{2} \sum_{k\gqq 2} (\pm)C_{k,1}=0
\]
for any sequence of signs $\pm$. Therefore 
$C_{k,1}=0$ for all $k\gqq 1$.
Analogously, we infer that 
$C_{k,k'}=0$ for all $k,k'\gqq 1$.
This proves that $\<\hat v_{2k},\check v_{2k'}\>=0$ for all $k,k'$.

\textbf{Step 3: } By Step 2 we obtain
\begin{align*}
F(x)=&|{v_1}|^2+4|\sum_{k=1}^{n}
x_k\hat v_{2k}|^2+4|\sum_{k=1}^{n} y_k\check v_{2k}|^2.
\end{align*}
For any choice $x_1\in\{-1,1\}$, $x_k=1$ (which implies $y_k=0$) for all $k\gqq 2$,
the Pythagoras theorem yields
\begin{align*}
F(x)=&|{v_1}|^2+4|x_1\hat v_{2k}+\sum_{k=2}^{n}
\hat v_{2k}|^2=
|{v_1}|^2+4|\hat v_{2k}|^2+4|\sum_{k=2}^{n}
\hat v_{2k}|^2+8\<x_1\hat v_{2},\sum_{k=2}^{n}
\hat v_{2k}\>,
\end{align*}
which implies
\[
\<\hat v_{2},\sum_{k=2}^{n}
\hat v_{2k}\>=0,
\]
due to $x_1$ can be chosen $\pm 1$.
An analogous reasoning yields
\begin{equation}\label{eq:kkprime}
\sum_{\substack{
k=1\\
k\neq k'}
}^{n}\<\hat v_{2k'},\hat v_{2k}\>=0.
\end{equation}
for any $k'\gqq 1$.

For any choice $x_k\in\{-1,1\}$ (which implies $y_k=0$) for all $k$,
the Pythagoras theorem yields
\begin{align*}
F(x)&=|{v_1}|^2+4|\sum_{k=1}^{n} x_k\hat v_{2k}|^2=|{v_1}|^2+
4|x_1\hat v_{2}+x_2\hat v_{4}+\sum_{k\gqq 3} x_k\hat v_{2k}|^2\\
&=|{v_1}|^2+4|x_1\hat v_{2}+x_2\hat v_{4}|^2
+4|\sum_{k\gqq 3} x_k\hat v_{2k}|^2
+8\<x_1\hat v_{2}+x_2\hat v_{4},\sum_{k\gqq 3} x_k\hat v_{2k}\>.
\end{align*}
Specifying $x_k=1$ for all $k\gqq 3$ and comparing it with $x_k=-1$ for all $k\gqq 3$ implies
\[
\<x_1\hat v_{2}+x_2\hat v_{4},\sum_{k\gqq 3} \hat v_{2k}\>=0.
\]
Then 
\[
\<x_1\hat v_{2},\sum_{k\gqq 3} \hat v_{2k}\>+
\<x_2\hat v_{4},\sum_{k\gqq 3} \hat v_{2k}\>=0.
\]
Bearing in mind (\ref{eq:kkprime}),
summing and subtracting $\hat v_{4}$ and $\hat v_{2}$ in the corresponding sums above yields
\[
\<x_1\hat v_{2},-\hat v_{4}\>+
\<x_2\hat v_{4},-\hat v_{2}\>=0.
\]
Since $x_1,x_2\in \{-1,1\}$ is arbitrary, we infer
$\<\hat v_{2},\hat v_{4}\>=0$.
By analogous reasoning it is shown that
$\<\hat v_{2k},\hat v_{2k'}\>=0$ for all $k\neq k'$.
Switching the role of $x_k$ and $y_k$ in the preceding reasoning also shows
$\<\check v_{2k},\check v_{2k'}\>=0$ for all $k\neq k'$.

Combining Step 1-3 shows that the family
$({v_1},\hat v_2,\check v_2\ldots,\hat v_{2n},\check v_{2n})$ is orthogonal.

\textbf{Step 4: } In the sequel, we prove $|\hat v_{2k}|=|\check v_{2k}|$ for all $k$.
By Step 3 we obtain
\begin{align*}
F(x)=&|{v_1}|^2+4\sum_{k=1}^{n}
x^2_k|\hat v_{2k}|^2+4\sum_{k=1}^{n} y^2_k|\check v_{2k}|^2.
\end{align*}
The choice $x_1=\z$, $x_k=1$ for all $k\gqq 2$ implies $y_1=\pm \sqrt{1-\z^2}$, $y_k=0$ for all $k\gqq 2$. Therefore the quadratic polynomial
\begin{align*}
\z \mapsto F((\z,1,\ldots,1))&=|{v_1}|^2+4\sum_{k=2}^{n}
|\hat v_{2k}|^2
+4\z^2 |\hat v_{2k}|^2+
4(1-\z^2)|\check v_{2}|^2\\
&=
|{v_1}|^2+4\sum_{k=2}^{n}
|\hat v_{2k}|^2+4|\check v_{2}|^2
+4\z^2 (|\hat v_{2}|^2-|\check v_{2}|^2)
\end{align*}
is constant. Consequently, 
$|\hat v_{2}|^2=|\check v_{2}|^2$.
Analogously it is shown that 
$|\hat v_{2k}|^2=|\check v_{2k}|^2$ for all $k\gqq 2$.
\end{proof}

\section*{Acknowledgments}
The research of GBV has been supported by the Academy of Finland, via
the Matter and Materials Profi4 university profiling action.
GBV also would like to express
his gratitude to University of Helsinki for all the facilities used along
the realization of this work. 
The research of MAH has been supported by the 
proyecto de la Convocatoria 2020-2021: ``Stochastic dynamics of systems perturbed with small Markovian noise with applications in
biophysics, climatology and statistics'' of the School of Sciences (Facultad de Ciencias) at Universidad de los Andes. 
JCP acknowledges support from  CONACyT-MEXICO CB-250590.
The authors would like to thank prof. J. M. Pedraza (Physics Department, Universidad de los Andes) for useful comments which have led to the Subsection~\ref{ss:ext3} and prof. J. Goodrick (Mathematics department, Universidad de los Andes) for helpful comments on the introduction.

\bibliographystyle{amsplain}

\end{document}